\documentclass[]{article}

\usepackage[a4paper, margin=3cm]{geometry}
\usepackage{microtype}
\usepackage{amsmath,amssymb,mathtools,amsthm}
\usepackage{xfrac}
\usepackage{graphicx}
\usepackage{subcaption}
\usepackage{epstopdf}
\usepackage[table]{xcolor}
\usepackage{tikz}
\usepackage{booktabs}
\usepackage{makecell,multirow,diagbox,rotating,tabularx}
\usepackage[noend,ruled,vlined,linesnumbered]{algorithm2e}
\SetInd{1em}{1em}

\usepackage{natbib}
\usepackage{url}
\usepackage{hyperref}
\usepackage{authblk}
\usepackage{cleveref}
\let\cite\citep
\newcommand{\email}[1]{\href{mailto:#1}{#1}}
\newenvironment{keywords}{\paragraph{Keywords:}}{}

\newcommand{\E}{\mathbb{E}}
\newcommand{\IR}{\mathbb{R}}

\renewcommand{\P}{\mathcal{P}}
\newcommand{\LL}{\psi}

\DeclareMathOperator{\dist}{dist}
\newcommand{\F}{\mathcal{F}}
\newcommand{\ML}{\mathcal{L}}

\definecolor{Gray}{gray}{0.92}
\definecolor{spacoRed}{HTML}{DA4C4C}
\definecolor{mgdGreen}{HTML}{009E8E}
\definecolor{gdaOrange}{HTML}{495D79}
\definecolor{mmpenPurple}{HTML}{9467BD}
\definecolor{gbalBlue}{HTML}{1F77B4}
\definecolor{hingePurple}{RGB}{148,103,189}
\definecolor{ganBlue}{RGB}{31,119,191}
\definecolor{gancGreen}{RGB}{0,158,142}
\definecolor{lenBlue}{RGB}{31,119,191}
\definecolor{egGreen}{RGB}{0,158,142}
\definecolor{boxgreen}{RGB}{0,158,142}
\definecolor{boxorange}{RGB}{255,108,0}
\definecolor{fariness-c1}{HTML}{96C047}
\definecolor{fariness-c2}{HTML}{1F77B4}
\definecolor{fariness-c3}{HTML}{E377C2}
\definecolor{fariness-c4}{HTML}{DA4C4C}
\definecolor{fariness-c5}{HTML}{9467BD}

\DeclareRobustCommand{\legendItem}[3][0.6]{%
  \tikz[baseline=-0.8ex]{%
    \draw[#2, opacity=0.25, line width=3pt, line cap=round] (0,0) -- (#1,0);
    \draw[#2, line width=1.5pt, line cap=round] (0,0) -- (#1,0);
  }%
  \hspace{0.3em}{#3}%
}
\DeclareRobustCommand{\legendItemm}[3][0.6]{%
  \tikz[baseline=-0.8ex]{%
    \draw[#2, line width=1.5pt, line cap=round] (0,0) -- (#1,0);
  }%
  \hspace{0.3em}{#3}%
}
\DeclareRobustCommand{\legendsquare}[1]{%
  \tikz[baseline=0ex, inner sep=0pt]{%
    \draw[draw=black, fill=#1, line width=0.2pt] (0,0) rectangle (1.5ex,1.5ex);
  }%
}

\theoremstyle{plain}
\newtheorem{theorem}{Theorem}[section]
\newtheorem{example}[theorem]{Example}
\newtheorem{proposition}[theorem]{Proposition}
\newtheorem{lemma}[theorem]{Lemma}
\newtheorem{corollary}[theorem]{Corollary}
\theoremstyle{definition}
\newtheorem{definition}[theorem]{Definition}
\newtheorem{assumption}[theorem]{Assumption}
\theoremstyle{remark}
\newtheorem{remark}[theorem]{Remark}

\title{A Single-Loop Penalty-based Algorithm for Stochastic Minimax Optimization with Nonlinear Coupled Constraints}

\author{
	Qichao Cao\thanks{Department of Mathematics, Southern University of Science and Technology, Shenzhen 518055, People's Republic of China. (\email{caoqc2024@mail.sustech.edu.cn}).} \quad
	Shangzhi Zeng\thanks{National Center for Applied Mathematics Shenzhen, and Department of Mathematics, Southern University of Science and Technology, Shenzhen 518055, People's Republic of China. (\email{zengsz@sustech.edu.cn}).} \quad
	Jin Zhang\thanks{Corresponding author. Department of Mathematics, and National Center for Applied Mathematics Shenzhen, Southern University of
		Science and Technology, Shenzhen 518055, People's Republic of China. (\email{zhangj9@sustech.edu.cn}).} \quad
	Yuxuan Zhou\thanks{Department of Mathematics, Southern University of Science and Technology, Shenzhen 518055, People's Republic of China. (\email{zhouyx2025@mail.sustech.edu.cn}).}
}

\begin{document}

\maketitle

\begin{abstract}
    We study stochastic nonconvex-concave minimax optimization with nonlinear
    coupled constraints that are convex in the maximization variable. To address the nonsmoothness arising from such constraints, we develop a
    penalty-based smooth approximation that combines quadratic penalization of the
    coupled constraints with quadratic regularization of the inner maximization
    problem. Based on this approximation, we propose SPACO, a single-loop
    stochastic gradient algorithm that tracks the inner maximizer by one stochastic
    ascent step, updates the outer variable using an inexact stochastic descent
    direction, and adaptively updates the penalty and regularization parameters
    over the iterations. For the penalty-based smooth approximation, we establish
    convergence guarantees from both minimizer and stationarity perspectives. 
    In particular, we introduce enhanced KKT conditions and show that stationary points of the
    smooth approximations can converge to points satisfying these conditions. An
    example illustrates that the enhanced KKT conditions can help exclude KKT
    points that are not local minimizers. For SPACO, we prove non-asymptotic
    complexity bounds for stationarity and feasibility, as well as asymptotic subsequential convergence to enhanced KKT points. Numerical experiments on synthetic examples,
    fairness-aware classification, and constrained generative adversarial network
    training demonstrate the effectiveness of the proposed method.
\end{abstract}


\section{Introduction}

In this paper, we consider stochastic minimax optimization with coupled constraints (MCC), formulated as
\begin{equation}\label{MCC}
\begin{aligned}
    &\min_{x\in X} \max_{y\in Y} \; \{f(x,y) \mid c(x,y)\le 0\},
\end{aligned}
\end{equation}
where $f(x,y) := \mathbb{E}_{\xi \sim D}\big[ F(x,y;\xi) \big]$ is a stochastic objective defined by the expectation over a distribution $D$, and $F(x,y;\xi)$ denotes a stochastic realization associated with a random sample $\xi \sim D$. Here,  $X \subset \mathbb{R}^n$ and $Y \subset \mathbb{R}^m$ are nonempty, closed and convex sets; the precise standing assumptions used in the analysis are stated in Section~\ref{sec:preliminaries}. The coupled constraint function $c : X \times Y \to \mathbb{R}^p$ is continuously differentiable and potentially nonlinear. We assume that $f(x,y)$ is concave in $y$ and that $c(x,y)$ is convex in $y$.

When the coupled constraint $c(x,y)\le 0$ is absent, problem~\eqref{MCC} reduces to the well-studied classical minimax optimization problem, which has been extensively studied due to its broad applicability in many applications including robust optimization~\cite{ben1998robust,bertsimas2004price}, adversarial learning~\cite{szegedy2014intriguing,goodfellow2015explaining}, and generative adversarial networks~\cite{goodfellow2014generative,arjovsky2017wasserstein}. Consequently, numerous algorithms have been developed for solving these minimax problems in both deterministic~\cite{korpelevich1976extragradient,nemirovski2004prox,chambolle2011first} and stochastic settings~\cite{lin2020gradient,nemirovski2009robust,juditsky2011solving}.

Despite this substantial progress, the development of algorithms for minimax optimization with coupled constraints remains limited. 
The inclusion of coupled constraints provides a more powerful modeling framework for a wide range of challenging applications, including constrained generative adversarial network training~\cite{chao2021constrained}, perceptual adversarial robustness~\cite{laidlaw2020perceptual}, adversarial attacks in network flow problems~\cite{tsaknakis2023minimax}, and linear projection equations~\cite{dai2024optimality}. However, from a computational perspective, coupled constraints introduce significant complexity. Even when the objective is strongly convex–strongly concave and the coupled constraints are linear, solving the MCC problem is NP-hard~\cite{tsaknakis2023minimax}.

\textbf{Min--min--max reformulation.} Most existing numerical approaches for solving MCC rely on duality-based techniques. In particular, problem~\eqref{MCC} can be reformulated via the value function
\begin{equation}\label{deterministicMCC}
    \min_{x\in X} \; \varphi(x),\;\text{where }\; \varphi(x):=\;\max_{y\in Y}\{f(x,y)\mid c(x,y)\le 0\}.
\end{equation}
The value function is the optimal value of a constrained inner maximization problem. By introducing Lagrange multipliers for the coupled constraints and assuming strong duality holds for this inner problem, the value function admits the Lagrangian representation
\begin{equation*}
\varphi(x)=\max_{y\in Y}\min_{\lambda\in \mathbb{R}_{+}^{p}} \; \ML(x,y,\lambda) = \min_{\lambda\in \mathbb{R}_{+}^{p}} \max_{y\in Y}\;\ML(x,y,\lambda),
\end{equation*}
where $\ML(x,y,\lambda) := f(x,y) - \lambda^\top c(x,y)$ is the Lagrangian for the inner maximization subproblem. This leads to a min--min--max reformulation of the MCC:
\begin{equation}\label{minminmax}
\min_{x\in X, \lambda\in \mathbb{R}_{+}^{p}} \max_{y\in Y}\;\ML(x,y,\lambda),
\end{equation}
where the multiplier $\lambda$ is introduced as an additional optimization variable. A key advantage of this reformulation is that it enables the use of a broad range of existing algorithmic techniques developed for min--max optimization, especially when the constraints are linear. In particular, \cite{tsaknakis2023minimax,zhang2024zeroth,zhang2024alternating} proposed several gradient-based descent--ascent methods for minimax problems with linearly coupled constraints.

Minimax optimization with nonlinear coupled constraints is substantially more
challenging. Several methods have been developed for deterministic problems in
this setting based on the min--min--max reformulation. \cite{lu2023first}
proposed a class of nested augmented Lagrangian methods and provided a
comprehensive convergence analysis. \cite{hu2024minimization} further
reformulated \eqref{minminmax} as a minimization problem via Moreau envelope
techniques and developed a novel subgradient method. More recently,
\cite{hu2026convergence} developed a value-function-based hypergradient
descent method and computed an approximate hypergradient of $\varphi(x)$ using
an augmented Lagrangian approach for the inner problem.

While these methods provide effective algorithmic frameworks for the
min--min--max reformulation, directly working with this reformulation has
limitations. In particular, first-order methods applied to
\eqref{minminmax} naturally seek points satisfying the KKT stationarity
condition~\eqref{KKT} of the reformulated problem. Such stationarity, however,
need not characterize local minimizers of the original value-function problem:
a KKT point of the reformulation may fail to be a local minimizer of
\(\varphi\). We refer to such points as spurious KKT points. The following
example illustrates this limitation.

\begin{example}\label{exm}
Consider the following MCC:
\begin{equation}\label{example}
\begin{aligned}
\min_{x\in X}\max_{\substack{y\in Y\\ \mathbf{e}^\top y-\|x\|^2\le 0}}
\left(\frac{\|x\|^2}{2}-1\right)^2
+2\|x\|^2-(\mathbf{e}^\top x)^2
-\frac{\|y-\mathbf{e}\|^2}{2}
+\frac{x^\top y}{2},
\end{aligned}
\end{equation}
where $X=[-\frac{5}{4},\frac{5}{4}]^2\subset\IR^2$,
$Y=[-10,10]^2\subset\IR^2$, and $\mathbf{e}=(1,1)$.
The optimal solution is
$(x^*,y^*)=(\alpha\mathbf{e},(1+\alpha/2)\mathbf{e})$, where
\(\alpha\approx-1.0545\) is the negative root of
\(8\alpha^3-7\alpha+2=0\). Nevertheless, a direct verification shows that
$(\mathbf{0},\mathbf{0},1)$ also satisfies the KKT system~\eqref{KKT}, whereas
$x=\mathbf{0}$ is not a local minimizer of the value function \(\varphi\).
A detailed analysis is provided in Proposition~\ref{prop:spurious-kkt}.
\end{example}

\textbf{Penalty-based Smooth Approximation.} Motivated by the limitation above,
we develop a penalty-based approach for solving MCC. Penalty methods are
classical tools for constrained optimization
\cite{nocedal2006numerical,bertsekas1997nonlinear} and are closely related to
enhanced Fritz John/KKT-type and sequential optimality conditions
\cite{hestenes1975optimization,bertsekas2002pseudonormality,andreani2019sequential}.
The MCC can be viewed as the minimization of the value function in
\eqref{deterministicMCC}. Although this formulation is compact, minimizing
\(\varphi(x)\) is difficult because the value function may be nonsmooth. This
nonsmoothness is induced both by the coupled constraints \(c(x,y)\le0\) and by
the possible non-uniqueness of the inner maximizer \cite{guo2024sensitivity}.

To address this difficulty, we incorporate the coupled constraints through a
quadratic penalty and add a quadratic regularization term in the maximization
variable. This gives the regularized penalized objective
\begin{equation*}
    \LL_{\rho,\sigma}(x,y) := f(x,y) - \frac{\rho}{2}\big\|[c(x,y)]_+\big\|^2 - \frac{\sigma}{2}\|y\|^2,
\end{equation*}
where $\rho > 0$ is a penalty parameter and $\sigma > 0$ is a regularization parameter. The corresponding smooth approximation of $\varphi(x)$ is
\begin{equation}\label{smoothapproximation}
    \varphi_{\rho,\sigma}(x) := \max_{y \in Y}\, \LL_{\rho,\sigma}(x,y).
\end{equation}
We analyze the limiting behavior of this approximation and connect its
approximate stationarity to newly introduced enhanced KKT conditions for the
original MCC. Building on this framework, we then develop a practical
single-loop stochastic gradient algorithm.

In contrast to the min--min--max reformulation, a penalty scheme retains
sequential information about constraint violation, which is useful for
excluding spurious KKT stationary points. Figure~\ref{fig:intro-example}
provides a compact numerical preview on Example~\ref{exm}; implementation
details are given in Section~\ref{sec:experiments}. The representative
min--min--max methods MGD~\cite{tsaknakis2023minimax},
GBAL~\cite{hu2026convergence}, and MMPen~\cite{hu2024minimization} can be
attracted to the spurious KKT point, whereas SPACO reaches the true solution
from all tested initializations.

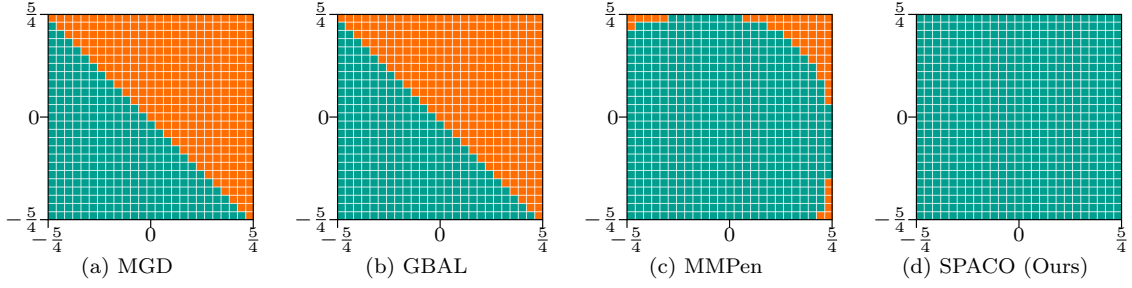
\begin{figure}[t]
	\centering
	\captionsetup[subfigure]{font=footnotesize,skip=1pt}
	\captionsetup{font=small,skip=3pt}
	\newlength{\imagewidth}
	\setlength{\imagewidth}{0.235\linewidth}
	
	\begin{subfigure}[t]{\imagewidth}
		\centering
		\resizebox{\linewidth}{!}{
			\begin{tikzpicture}[inner sep=0pt, outer sep=0pt]
				\node[anchor=south west, inner sep=0] (image) at (0,0) {
					\includegraphics[width=0.7\imagewidth]{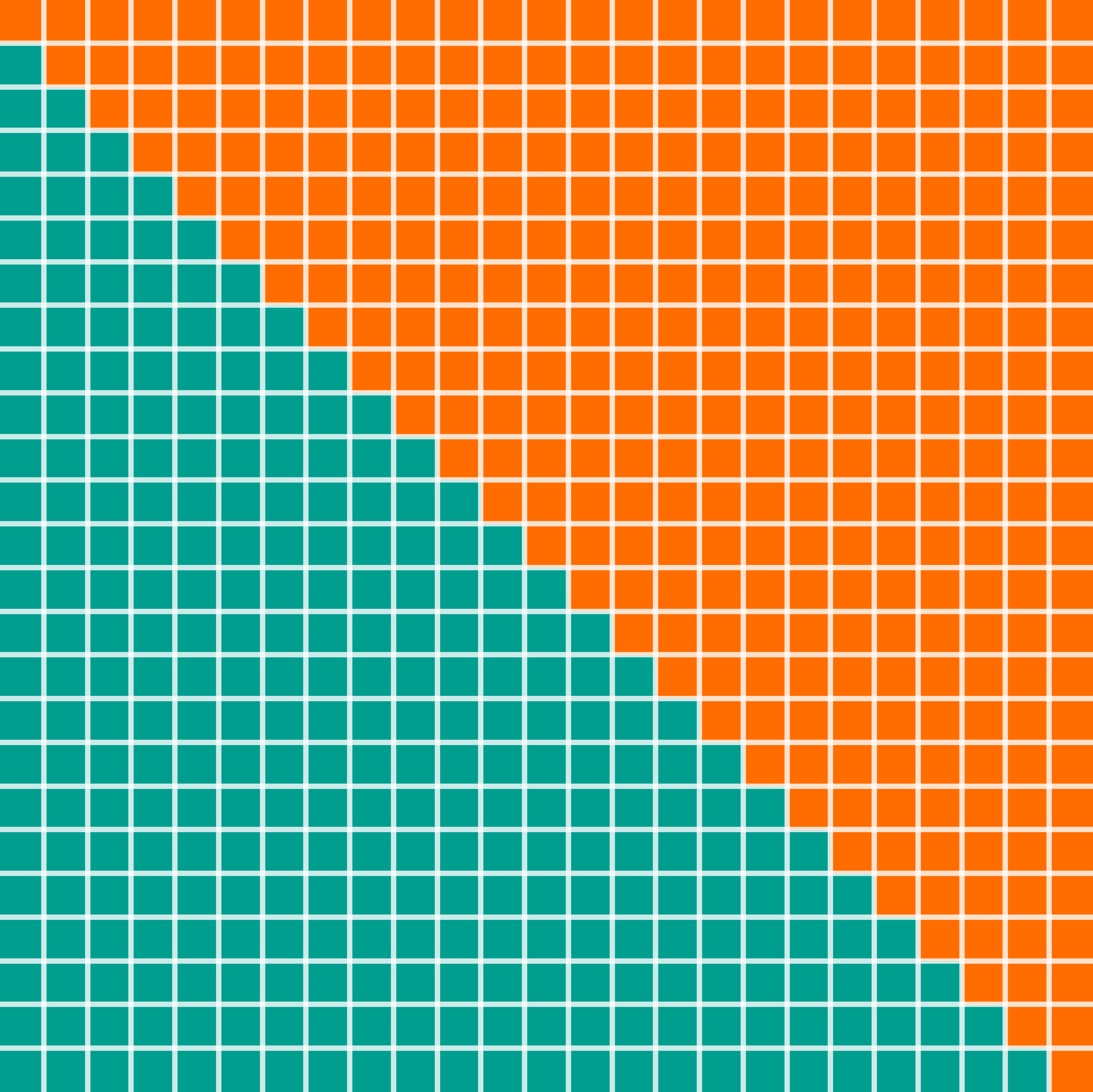}
				};
				
				\begin{scope}[
					x={(image.south east)},
					y={(image.north west)}
					]   
					\foreach \pos/\label in {0.0/{-\frac{5}{4}}, 0.5/{0}, 1.0/{\frac{5}{4}}} {
						\draw (\pos, 0) -- (\pos, -0.04) node[below, font=\scriptsize] {$\label$};
					}
					
					\foreach \pos/\label in {0.0/{-\frac{5}{4}}, 0.5/{0}, 1.0/{\frac{5}{4}}} {
						\draw (0, \pos) -- (-0.04, \pos) node[left, font=\scriptsize] {$\label$};
					}
					
					\draw (0,0) rectangle (1,1);
					\path[use as bounding box] (-0.14,-0.13) rectangle (1.02,1.02);
				\end{scope}
			\end{tikzpicture}%
		}
		\caption{MGD}\label{fig:mgd-initialization}
	\end{subfigure}%
	\hfill
	\begin{subfigure}[t]{\imagewidth}
		\centering
		\resizebox{\linewidth}{!}{
			\begin{tikzpicture}[inner sep=0pt, outer sep=0pt]
				\node[anchor=south west, inner sep=0] (image) at (0,0) {
					\includegraphics[width=0.7\imagewidth]{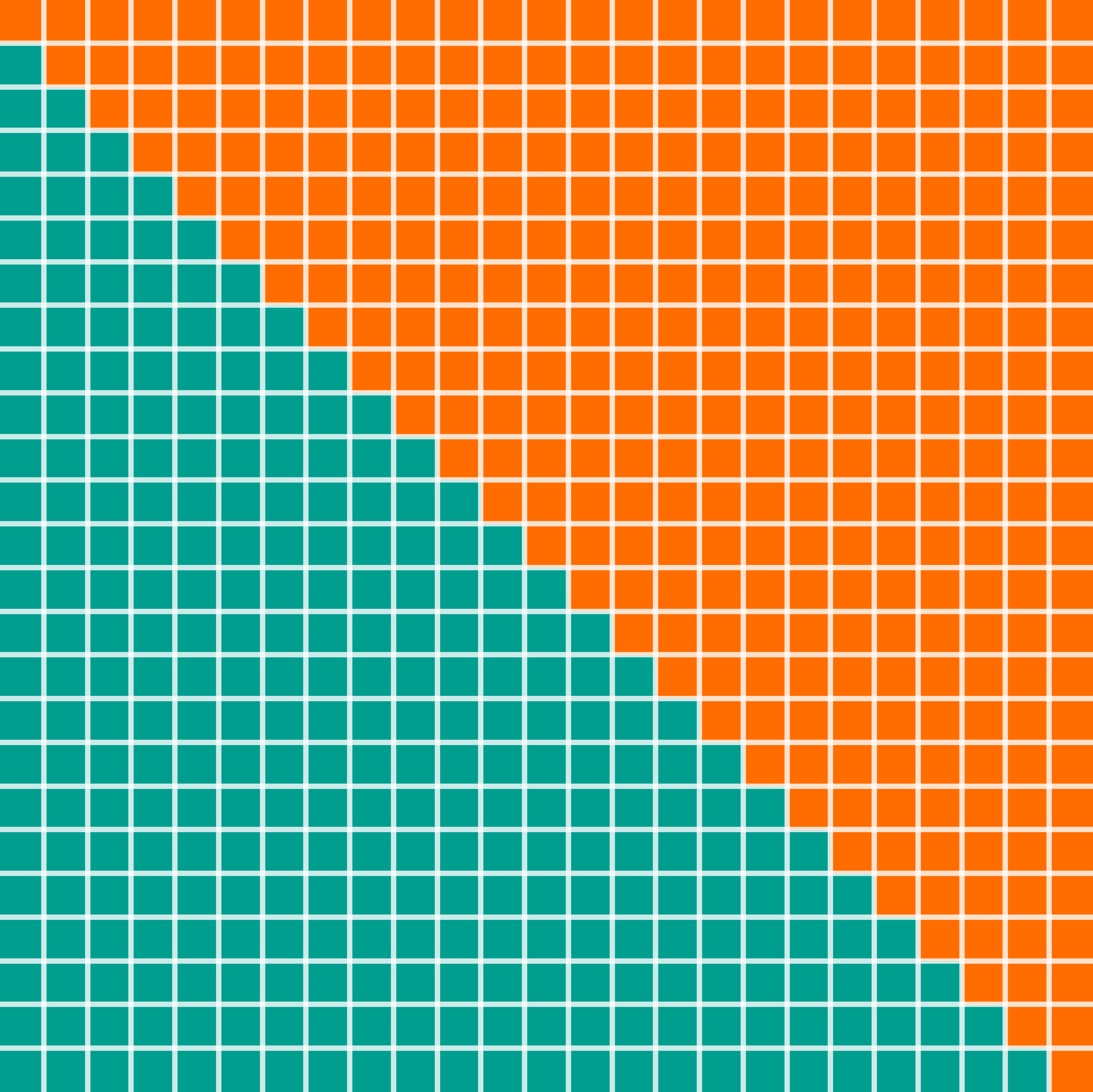}
				};
				
				\begin{scope}[
					x={(image.south east)},
					y={(image.north west)}
					]   
					\foreach \pos/\label in {0.0/{-\frac{5}{4}}, 0.5/{0}, 1.0/{\frac{5}{4}}} {
						\draw (\pos, 0) -- (\pos, -0.04) node[below, font=\scriptsize] {$\label$};
					}
					
					\foreach \pos/\label in {0.0/{-\frac{5}{4}}, 0.5/{0}, 1.0/{\frac{5}{4}}} {
						\draw (0, \pos) -- (-0.04, \pos) node[left, font=\scriptsize] {$\label$};
					}
					
					\draw (0,0) rectangle (1,1);
					\path[use as bounding box] (-0.14,-0.13) rectangle (1.02,1.02);
				\end{scope}
			\end{tikzpicture}%
		}
		\caption{GBAL}\label{fig:gbal-initialization}
	\end{subfigure}%
	\hfill
	\begin{subfigure}[t]{\imagewidth}
		\centering
		\resizebox{\linewidth}{!}{
			\begin{tikzpicture}[inner sep=0pt, outer sep=0pt]
				\node[anchor=south west, inner sep=0] (image) at (0,0) {
					\includegraphics[width=0.7\imagewidth]{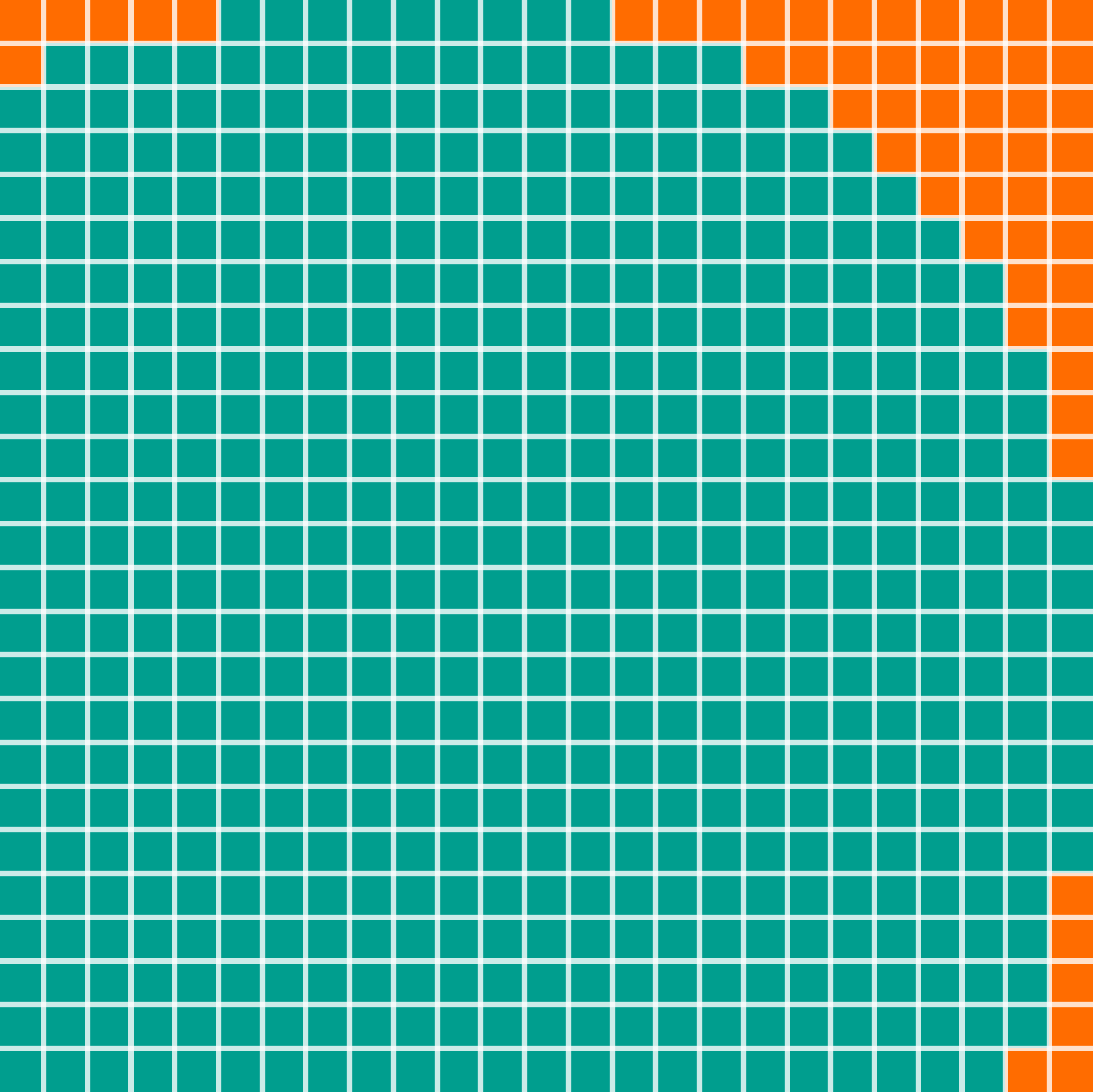}
				};
				
				\begin{scope}[
					x={(image.south east)},
					y={(image.north west)}
					]   
					\foreach \pos/\label in {-0.0/{-\frac{5}{4}}, 0.5/{0}, 1.0/{\frac{5}{4}}} {
						\draw (\pos, 0.0) -- (\pos, -0.04) node[below, font=\scriptsize] {$\label$};
					}
					
					\foreach \pos/\label in {0.0/{-\frac{5}{4}}, 0.5/{0}, 1.0/{\frac{5}{4}}} {
						\draw (0, \pos) -- (-0.04, \pos) node[left, font=\scriptsize] {$\label$};
					}
					
					\draw (0,0) rectangle (1,1);
					\path[use as bounding box] (-0.14,-0.13) rectangle (1.02,1.02);
				\end{scope}
			\end{tikzpicture}%
		}
		\caption{MMPen}\label{fig:mmpen-initialization}
	\end{subfigure}%
	\hfill
	\begin{subfigure}[t]{\imagewidth}
		\centering
		\resizebox{\linewidth}{!}{
			\begin{tikzpicture}[inner sep=0pt, outer sep=0pt]
				\node[anchor=south west, inner sep=0] (image) at (0,0) {
					\includegraphics[width=0.7\imagewidth]{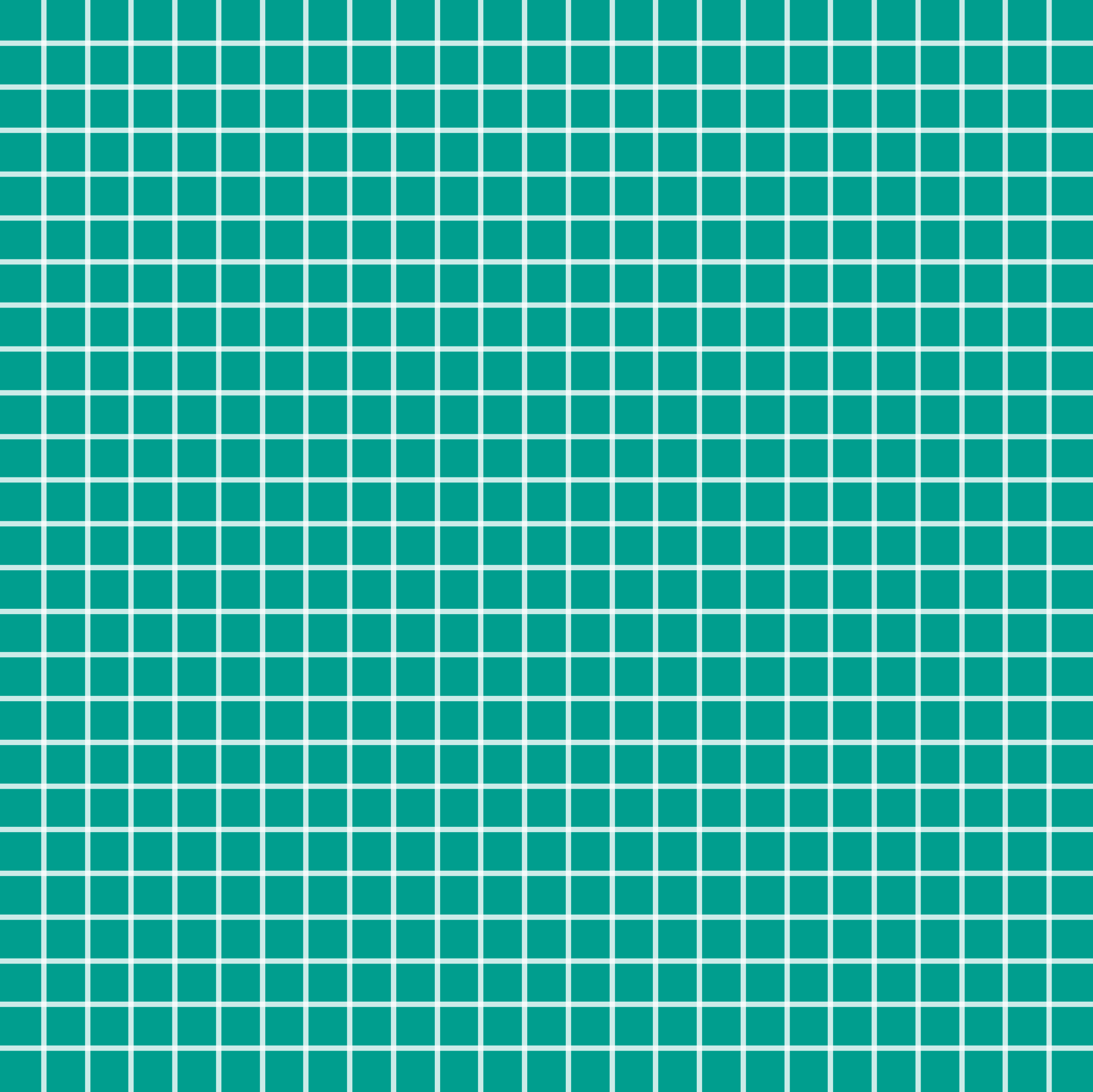}
				};
				
				\begin{scope}[
					x={(image.south east)},
					y={(image.north west)}
					]   
					\foreach \pos/\label in {0.0/{-\frac{5}{4}}, 0.5/{0}, 1.0/{\frac{5}{4}}} {
						\draw (\pos, 0) -- (\pos, -0.04) node[below, font=\scriptsize] {$\label$};
					}
					
					\foreach \pos/\label in {0.0/{-\frac{5}{4}}, 0.5/{0}, 1.0/{\frac{5}{4}}} {
						\draw (0, \pos) -- (-0.04, \pos) node[left, font=\scriptsize] {$\label$};
					}
					
					\draw (0,0) rectangle (1,1);
					\path[use as bounding box] (-0.14,-0.13) rectangle (1.02,1.02);
				\end{scope}
			\end{tikzpicture}%
		}
		\caption{SPACO (Ours)}\label{fig:spaco-initialization}
	\end{subfigure}
	\caption{
		Convergence outcomes from different initial points on Example~\ref{exm}. Panels (a)--(c) show the min--min--max methods MGD, GBAL, and MMPen, while panel (d) shows SPACO. Initial points are taken on a grid over \(X=[-\frac{5}{4},\frac{5}{4}]^2\) with mesh size \(0.1\). Green \legendsquare{boxgreen} indicates convergence to the true solution, whereas orange \legendsquare{boxorange} indicates convergence to the spurious KKT point.
	}\label{fig:intro-example}
\end{figure}

\subsection{Contribution}
We summarize the contributions as follows:

  \textbf{Theoretical validation of penalty-based smooth approximation.}  We justify the proposed approximation from both minimizer and stationarity perspectives. First, we prove that accumulation points of minimizers of the smooth approximations solve the original MCC. Second, we introduce Type-I and Type-II enhanced KKT conditions, which strengthen standard KKT stationarity by incorporating sequential optimality information generated by the penalty scheme. We establish Type-I enhanced KKT as a necessary condition for local minimizers and Type-II enhanced KKT as a necessary condition for strict local minimizers. We further show that accumulation points of approximate stationary points of the smooth approximations satisfy these enhanced KKT conditions.

\textbf{Stochastic single-loop gradient algorithm for MCC.} Building on the smooth approximation framework, we develop SPACO (Stochastic Penalty-based Algorithm for minimax optimization with COupled constraints), a single-loop stochastic gradient algorithm for MCC with nonlinear coupled constraints. SPACO avoids this inner loop by tracking the inner solution through one stochastic ascent step and then performing an inexact stochastic descent update for the outer variable. We establish non-asymptotic complexity bounds for stationarity and feasibility, and prove asymptotic convergence to enhanced KKT points. Table~\ref{compare} highlights our primary algorithmic contributions compared with existing methods. Finally, we empirically validate the effectiveness and efficiency of SPACO on synthetic examples and real-world applications, including generative adversarial networks and fairness-aware classification.

\begin{table}[tbp]
    \centering
        \caption{NC, SC, and C denote Non-Convex, Strongly-Convex/Strongly-Concave, and Convex/Concave, respectively. ``Non-Asym'' and ``Asym'' indicate whether the method provides a convergence rate for a specific stationarity measure, and whether it establishes convergence of the iterate sequence to a limit point, respectively.}
        	\label{compare}
    \resizebox{1.0\linewidth}{!}%
    {
	\begin{tabular}{lcccccc}
		\hline
		Methods & Objective & Constraint & Loop Structure & Stochastic  & Convergence \\
		\hline
        \cite{tsaknakis2023minimax}         & SC-SC & Linear    & Nested-Loop & No  & Non-Asym \\
		\cite{lu2023first}          & NC-C  & Nonlinear & Nested-Loop & No  & Non-Asym \\
        \cite{hu2026convergence}         & NC-NC & Nonlinear & Nested-Loop & No  & Non-Asym     \\
        \cite{hu2024minimization}         & NC-SC & Nonlinear & Single-Loop & No  & Non-Asym     \\
		Ours & NC-C  & Nonlinear & Single-Loop & Yes  & Non-Asym \& Asym \\
		\hline
	\end{tabular}}
\end{table}

\textbf{Paper Organization.} The remainder of this paper is organized as follows. Section~\ref{sec:preliminaries} introduces the notation and preliminary results. Section~\ref{approximaiton} constructs the penalty-based smooth approximation and studies the limiting behavior of minimizers and stationary points. Section~\ref{convergence analysis} presents SPACO and establishes its non-asymptotic and asymptotic convergence guarantees. Section~\ref{sec:experiments} reports numerical experiments that validate the effectiveness of SPACO.

\section{Notation and Preliminaries}\label{sec:preliminaries}
This section fixes the notation and standing assumptions used throughout the paper. We also recall several basic tools from variational analysis and constrained optimization that will be used to analyze the penalty-based approximation in Section~\ref{approximaiton}.
\subsection{Notation and Standing Assumptions}
Throughout the paper, $\langle\cdot,\cdot\rangle$ and $\|\cdot\|$ denote the Euclidean inner product and norm, respectively. For a positive integer $p$, let $[p]:=\{1,\ldots,p\}$. For $z=(z_1,\ldots,z_n)\in\IR^n$, define
$[z]_+:=(\max\{z_1,0\},\ldots,\max\{z_n,0\}).$
Given $z\in\IR^n$ and $\epsilon>0$, let $B_\epsilon(z)$ be the closed Euclidean ball centered at $z$ with radius $\epsilon$. We write $\operatorname{ri}(S)$ for the relative interior of a set $S$, $\dist(z,S):=\inf_{u\in S}\|z-u\|$ for the distance from $z$ to a nonempty closed set $S$, and $\mathcal{P}_S(z)$ for the Euclidean projection of $z$ onto a nonempty closed convex set $S$. For a closed convex set $S$, we denote by $\mathcal{N}_S(z)$ the normal cone to $S$ at $z$. Specifically,
$
\mathcal{N}_S(z) = \{ v : \langle v, z' - z \rangle \leq 0, \forall z' \in S \},$
for $z\in S$, and $\mathcal{N}_S(z)=\emptyset$ for $z\notin S$.

For the MCC problem~\eqref{MCC}, we write
\[
\Gamma(x):=\{y\in Y:\ c(x,y)\le 0\},
\qquad
\varphi(x):=\max_{y\in\Gamma(x)} f(x,y),
\]
for the feasible-set mapping and the value function of the inner maximization problem. We impose the following standing assumptions throughout the paper.

\begin{assumption}\label{assum1}
The following conditions hold.
\begin{enumerate}
\item \(X\subset\mathbb R^n\) is nonempty, convex, and closed, and \(Y\subset\mathbb R^m\) is nonempty, convex, and compact.
\item There exist open sets \(U_x\subset\mathbb R^n\) and \(U_y\subset\mathbb R^m\) such that \(X\subset U_x\), \(Y\subset U_y\), and \(f\) and each \(c_i\), \(i\in[p]\), are defined on \(U_x\times U_y\), continuously differentiable, and have Lipschitz continuous gradients on \(U_x\times U_y\).
\item For every fixed \(x\in U_x\), \(f(x,\cdot)\) is concave on \(Y\), and each \(c_i(x,\cdot)\) is convex on \(Y\).
\item The set \(\Gamma(x):=\{y\in Y:c(x,y)\le 0\}\) is nonempty for every \(x\in X\).
\end{enumerate}
\end{assumption}

Under Assumption~\ref{assum1}, $\Gamma(x)$ is nonempty and compact for every $x\in X$; hence the maximum defining $\varphi(x)$ is attained and the value function is well defined.

\subsection{Continuity and Epi-convergence}
We next recall the continuity and convergence notions used to analyze the limiting behavior of the penalty-based smooth approximations developed in Section~\ref{approximaiton}.

\begin{definition}
A function $\phi:X\to\mathbb{R}$ is called lower semicontinuous on $X$ if, for any $\bar{x}\in X$ and any sequence $\{x_k\}\subset X$ converging to $\bar{x}$,
$
\liminf_{k\to\infty}\;\phi(x_k)\ge \phi(\bar{x}).
$
\end{definition}

Lower semicontinuity is equivalent to closedness of the epigraph and is a standard condition for the existence of minimizers; see, e.g., \citep[Theorem~1.9]{rockafellar2009variational}. We now introduce the corresponding continuity notions for set-valued mappings.
\begin{definition}
A set-valued mapping $S:X\rightrightarrows Y$ is called inner semicontinuous (resp. outer semicontinuous) at $\bar{x}\in X$ if
$\operatorname{Liminf}_{x\to\bar{x}}S(x)\supseteq S(\bar{x})
\;
\left(\text{resp. }\operatorname{Limsup}_{x\to\bar{x}}S(x)\subseteq S(\bar{x})\right),$
where the limits are taken along sequences in $X$ and
\begin{equation*}
\begin{aligned}
\underset{x\to\bar{x}}{\operatorname{Liminf}} \;S(x)
&:=\{y\mid \forall x^k\to\bar{x},\ \exists y^k\in S(x^k)\text{ such that }y^k\to y\},\\
\underset{x\to\bar{x}}{\operatorname{Limsup}} \;S(x)
&:=\{y\mid \exists x^k\to\bar{x},\ y^k\in S(x^k)\text{ such that }y^k\to y\}.
\end{aligned}
\end{equation*}
\end{definition}

The following lemma shows that inner semicontinuity of the feasible-set mapping is sufficient for lower semicontinuity of the value function.
\begin{lemma}\label{isc and lsc}
    Suppose that $\Gamma(x):=\{y\in Y:\ c(x,y)\le 0\}$ is inner semicontinuous on $X$. Then $\varphi$ is lower semicontinuous on $X$.
\end{lemma}

Epi-convergence is a set-convergence notion for epigraphs; see, e.g., \citep[Definition~7.1]{rockafellar2009variational}. Rather than working directly with the set-convergence definition, we use the equivalent sequential characterization in \citep[Proposition~7.2]{rockafellar2009variational} and adopt it as our working definition.

\begin{definition}\label{def:epiconvergence}
A sequence of functions $\{\phi_k\}$ defined on $X$ is said to epi-converge to $\phi$ on $X$, written $\phi_k\xrightarrow{e}\phi$, if for every $\bar x\in X$, the following two conditions hold:
\begin{enumerate}
    \item for every sequence $\{x_k\}\subset X$ with $x_k\to \bar x$,
 $
        \liminf_{k\to\infty}\;\phi_k(x_k)\ge \phi(\bar x);
$
    \item there exists a sequence $\{x_k\}\subset X$ with $x_k\to \bar x$ such that
$
        \limsup_{k\to\infty}\;\phi_k(x_k)\le \phi(\bar x).
$
\end{enumerate}
\end{definition}

Epi-convergence is useful for studying the limiting behavior of minimizers. The following result formalizes this connection \citep[Proposition~4.6]{bonnans2013perturbation}.

\begin{lemma}\label{lem:epi-minimizers}
Let \(\phi_k \xrightarrow{e} \phi\) on \(X\). If \(x_k\in\arg\min_{x\in X}\phi_k(x)\) and \(x_k\to\bar x\), then 
\[\bar x\in\arg\min_{x\in X}\phi(x).\]
\end{lemma}

\subsection{KKT Stationarity and Constraint Qualification}

We first recall the standard KKT stationarity condition for the MCC problem~\eqref{MCC}, which has also been used in related analyses of minimax problems with coupled constraints \cite{hu2024minimization,guo2024sensitivity,lu2023first,zhang2024zeroth}.

\begin{definition}\label{def:kkt-mcc}
Let $\ML(x,y,\lambda):=f(x,y)-\lambda^\top c(x,y)$ denote the Lagrangian of the inner maximization problem. A pair $(x,y)\in X\times Y$ is called a stationary \emph{(or KKT)} point of~\eqref{MCC} if there exists a multiplier $\lambda\in\IR^p_+$ such that
\begin{align}\label{KKT}
    \begin{cases}
        0\in \nabla_x \ML(x,y,\lambda)+\mathcal{N}_{X}(x), \quad
        0\in -\nabla_y \ML(x,y,\lambda)+\mathcal{N}_{Y}(y),\\
        c(x,y)\le 0,\quad \lambda^\top c(x,y)=0.
    \end{cases}
\end{align}
\end{definition}

The KKT condition is a standard necessary stationarity condition for local minimizers of MCC under suitable regularity assumptions \cite{hu2024minimization,guo2024sensitivity,ma2025calm}.
We next recall a pointwise Slater condition \cite{lu2023first} for the inner feasible system.  

\begin{definition}\label{def:slater}
For a fixed point $\bar{x} \in X$, we say that the Slater condition holds for the feasible system $y \in Y,\ c(\bar{x},y) \le 0$ if there exists a point $\tilde{y} \in \operatorname{ri}(Y)$ such that $c(\bar{x},\tilde{y}) < 0$.
\end{definition}

Inspired by the P\L{}CQ studied in~\cite{andreani2025primal}, we introduce the following local error-bound-type constraint qualification for coupled constraints in the minimax setting.

\begin{definition}\label{def:gplcq}
The \emph{generalized uniform Polyak--\L{}ojasiewicz constraint qualification} (GP\L{}CQ) is said to hold at $(\bar x,\bar y)$ if there exist constants $\delta>0$, $\beta>0$, and $\gamma>0$ such that, for all $(x,y)\in (B_\delta(\bar x)\cap X)\times(B_\delta(\bar y)\cap Y)$,
\begin{equation}\label{GPL-eq}
		\sqrt{p(x,y)}
		\le \frac{\gamma}{\beta}
	\left\|
	\begin{bmatrix}x\\ y\end{bmatrix}
	-\mathcal{P}_{X\times Y}\left(
	\begin{bmatrix}x\\ y\end{bmatrix}
	+\beta
	\begin{bmatrix}
		\nabla_x p(x,y)\\
		-\nabla_y p(x,y)
	\end{bmatrix}
		\right)
		\right\|,
	\end{equation}
where $p(x,y):=\frac12\|[c(x,y)]_+\|^2$ denotes the scalar constraint-violation measure.
\end{definition}
\begin{remark}
In the GP\L{}CQ condition above, the residual direction
$(\nabla_x p,\,-\nabla_y p)$
is not the usual projected-gradient residual for minimizing constraint violation jointly in \((x,y)\). Its sign pattern is aligned with the min--max penalty dynamics. A direct parametric generalization of the P\L{}CQ in~\cite{andreani2025primal} would instead use the following uniformly local \(y\)-component residual bound:
\[
    \sqrt{p(x,y)}
    \le \frac{\gamma}{\beta}
    \left\|
    y-\mathcal{P}_{Y}\bigl(y-\beta\nabla_y p(x,y)\bigr)
    \right\|,
\]
for all $(x,y)$ near $(\bar x,\bar y)$. This condition is stronger than Definition~\ref{def:gplcq}; in particular, it is sufficient for \eqref{GPL-eq}. The weaker GP\L{}CQ adopted here allows regularity in the \(x\)-component to also contribute to the error bound.
\end{remark}

We close this subsection by recording two consequences of Slater's condition. First, it guarantees inner semicontinuity of the feasible-set mapping.

\begin{lemma}\label{slaterisc}
Fix $\bar{x}\in X$ and assume that the Slater condition holds for the feasible system $y\in Y,\; c(\bar{x},y)\le 0$. Then the feasible-set mapping $\Gamma$ is inner semicontinuous at $\bar x$.
\end{lemma}

Second, the same condition implies the GP\L{}CQ at every feasible inner point.

\begin{proposition}\label{thm:slater-gplcq}
Fix $\bar{x}\in X$ and assume that the Slater condition holds for the feasible system $y\in Y,\; c(\bar{x},y)\le 0$. Then the GP\L{}CQ holds at $(\bar x,\bar y)$ for any $\bar y\in \Gamma(\bar{x})$.
\end{proposition}

\section{Penalty-based Smooth Approximation}\label{approximaiton}
In this section, we construct a continuously differentiable approximation of the value function $\varphi$ and analyze its limiting behavior.

\subsection{Smooth Approximation Construction}
To handle the nonsmoothness induced by the nonlinear coupled constraints, we first introduce a quadratic penalty for the inner maximization problem.
This yields the following approximation of the value function in \eqref{deterministicMCC}:
\begin{equation*}
\max_{y\in Y} \left\{f(x,y)-\frac{\rho}{2}\big\|[c(x,y)]_+\big\|^2 \right\},
\end{equation*}
where $\rho > 0$ is a penalty parameter.
Although the quadratic penalty accounts for constraint violation, the resulting value function may remain nonsmooth because the inner maximizer need not be unique. We therefore add a quadratic regularization term:
\begin{equation*}
\LL_{\rho,\sigma}(x,y) :=  f(x,y) -\frac{\rho}{2}\big\|[c(x,y)]_+\big\|^2 -\frac{\sigma}{2}\|y\|^2,
\end{equation*}
where $\sigma > 0$ is a regularization parameter. The corresponding approximate value function is
\begin{equation}\label{smoothapproximation}
\varphi_{\rho,\sigma}(x)
:= \max_{y \in Y}\, \LL_{\rho,\sigma}(x,y).
\end{equation}
The differentiability of $\varphi_{\rho,\sigma}$ follows from Danskin's theorem
(see, e.g., \citep[Proposition~B.25]{bertsekas1997nonlinear}): the regularization term makes the inner maximizer unique, and $\LL_{\rho,\sigma}$ is smooth with respect to $x$. The resulting gradient formula is given below.
\begin{proposition}\label{differentiable}
For any \(\rho>0\) and \(\sigma>0\), \(\varphi_{\rho,\sigma}\) is continuously differentiable on \(X\).
Moreover, for every \(x\in X\),
\begin{equation}\label{differentiability}
    \nabla \varphi_{\rho,\sigma}(x)
    = \nabla_x \LL_{\rho,\sigma}\bigl(x, y^*_{\rho,\sigma}(x)\bigr),
\end{equation}
where \(y^*_{\rho,\sigma}(x):= \arg\max_{y\in Y} \LL_{\rho,\sigma}(x,y)\) denotes the unique inner maximizer.
\end{proposition}

This penalty-based approximation motivates solving the MCC~\eqref{MCC} through a sequence of smooth problems of the form
\begin{equation}\label{approxproblem}
    \min_{x\in X} \; \varphi_{\rho_k,\sigma_k}(x),
\end{equation}
with parameter sequences $\rho_k \to \infty$ and $\sigma_k \to 0$.
The next two subsections characterize the asymptotic relationship between $\varphi_{\rho_k,\sigma_k}$ and $\varphi$ from two perspectives: convergence of minimizers and convergence of stationary points.

\subsection{Convergence of Minimizers}\label{section:epiconvergence}

In this subsection, we establish the asymptotic convergence of global and local minimizers of the smooth approximations. The analysis relies on epi-convergence, as introduced in Definition~\ref{def:epiconvergence}. We first state two auxiliary lemmas that establish the epi-convergence of $\varphi_{\rho_k,\sigma_k}$.

\begin{lemma}\label{limsup}
    Let $\rho_k\to\infty$ and $\sigma_k\to 0$ as $k\to\infty$. Then, for any fixed $x\in X$,
    \begin{equation*}
       \limsup_{k\to\infty} \, \varphi_{\rho_k,\sigma_k}(x)\le \varphi(x).
    \end{equation*}
\end{lemma}
The lower-bound argument requires the following continuity assumption on the value function.
\begin{assumption}\label{ass:lsc}
The value function $\varphi:X\to\mathbb{R}$ is lower semicontinuous on $X$.
\end{assumption}

\begin{remark}
For a fixed point $\bar x\in X$, lower semicontinuity of $\varphi$ at
$\bar x$ is guaranteed by inner semicontinuity of the feasible-set
mapping $\Gamma$ at $\bar x$; see Lemma~\ref{isc and lsc}. This requirement can be weakened to local graph-point inner 
semicontinuity of $\Gamma$ at $(\bar x,y^*(\bar x))$
\citep[Definition~1.63]{mordukhovich2006variational} for some $y^*(\bar x)\in \arg\max_{y\in\Gamma(\bar x)} f(\bar x,y)$: for every sequence $x_k\to\bar x$, there exists a sequence
$y_k\in\Gamma(x_k)$ such that $y_k\to y^*(\bar x)$.
\end{remark}
\begin{remark}
GP\L{}CQ at the pair
$(\bar x,y^*(\bar x))$ alone does not imply lower semicontinuity of
$\varphi$ at $\bar x$. A simple example is the MCC with
\[
\begin{aligned}
    \min_{x\in [0,1]}\max_{\substack{y\in [0,2]\\ x(y-1)\le 0}}
    y.
\end{aligned}
\]
In this case, $\Gamma(0)=[0,2]$ and $\Gamma(x)=[0,1]$ for $x>0$, so
$\varphi(0)=2$ while $\varphi(x)=1$ for $x>0$. Thus $\varphi$ is not lower
semicontinuous at $0$, although GP\L{}CQ holds at the pairs
$(0,2)$ and $(x,1)$ for $x>0$.
\end{remark}

\begin{lemma}\label{liminf}
    Suppose $\varphi$ is lower semicontinuous on $X$. Let $\rho_k\to\infty$ and $\sigma_k\to 0$ as $k\to\infty$. Then, for any sequence $\{x_k\}\subset X$ with $x_k \to \bar{x}$, we have
    \begin{align}
        \underset{k\to\infty}{\lim\inf} \;\varphi_{\rho_k,\sigma_k}(x_k)\ge \varphi(\bar{x}).
    \end{align}
\end{lemma}
Combining Lemmas~\ref{limsup} and~\ref{liminf}, we obtain the epi-convergence of $\varphi_{\rho_k,\sigma_k}$ to $\varphi$. Lemma~\ref{lem:epi-minimizers} then yields convergence of global minimizers.
\begin{theorem}\label{epiconvergence}
     Assume $\varphi$ is lower semicontinuous on $X$. Let $\rho_k\to\infty$ and $\sigma_k\to 0$, and let $x_k \in \mathrm{argmin}_{x \in X}\;\varphi_{\rho_k,\sigma_k}(x)$. Then any accumulation point $\bar{x}$ of $\{x_k\}$ is an optimal solution of the MCC~\eqref{MCC}, i.e., $\bar{x} \in \mathrm{argmin}_{x \in X} \varphi(x)$. 
\end{theorem}

The same argument yields convergence of local minimizers when the local optimality neighborhoods have a common positive radius.
\begin{corollary}
    Suppose the assumptions of Theorem~\ref{epiconvergence} hold. Let $x_k$ be a local minimizer of $\varphi_{\rho_k,\sigma_k}$, i.e., there exist constants $\delta_k>0$ such that $x_k\in \mathrm{argmin}_{x \in X\cap B_{\delta_k}(x_k)}\;\varphi_{\rho_k,\sigma_k}(x)$. If $\delta_k\ge\delta>0$ for all $k$, then any accumulation point $\bar{x}$ of $\{x_k\}$ is a local minimizer of the MCC~\eqref{MCC}.
\end{corollary}

We also characterize the limiting behavior of the corresponding inner maximizers.
\begin{proposition}
\label{ystarconvergence}
         Suppose that $\varphi$ is lower semicontinuous on $X$. Let $\rho_k\to\infty$ and $\sigma_k\to 0$. Then, for any sequence $\{x_k\}\subset X$ with $x_k \to \bar{x}$, every accumulation point $\bar{y}$ of $\{y^*_{\rho_k,\sigma_k}(x_k)\}$ satisfies $\bar{y}\in\arg\max_{y\in Y}\{f(\bar{x},y)\mid c(\bar{x},y)\le 0\}$.
\end{proposition}

\subsection{Convergence of Stationary Points}

In this subsection, we characterize the asymptotic behavior of stationary points of the smooth approximations.
Since $\varphi_{\rho,\sigma}$ is smooth but generally nonconvex, first-order methods applied to the approximate problems naturally generate sequences with asymptotically vanishing stationarity residuals, namely,
\[
u_k\in \nabla \varphi_{\rho_k,\sigma_k}(x_k)+\mathcal{N}_X(x_k),
\qquad u_k\to 0.
\]
As $\rho_k\to\infty$ and $\sigma_k\to0$, accumulation points of such approximate stationary sequences satisfy the KKT stationarity condition~\eqref{KKT}; see Theorem~\ref{thm:approx-to-type1}.

Beyond standard KKT stationarity, the penalty approximation yields additional sequential information that can rule out certain spurious KKT points.

We therefore introduce an enhanced KKT condition tailored to the proposed smooth approximation.
This condition strengthens the classical KKT system~\eqref{KKT} by requiring a vanishing-residual sequence generated by the penalized smooth approximations.
\begin{definition}\label{def:type1-ekkt}
A pair $(\bar x,\bar y)$ is called a \emph{Type-I enhanced KKT point}
of~\eqref{MCC} if there exists a multiplier $\bar\lambda\in\IR^p_+$ such that
$(\bar x,\bar y,\bar\lambda)$ satisfies the KKT system~\eqref{KKT}.
Moreover, there exist sequences $\{\rho_k\},\{\sigma_k\}\subset(0,\infty)$,
$\{x_k\}\subset X$, and $\{u_k\}\subset\IR^n$ such that
$\rho_k\to\infty$, $\sigma_k\to0$, and, with
$y_k^*:=y^*_{\rho_k,\sigma_k}(x_k)$,
$x_k\to\bar x$, $y_k^*\to\bar y$, $u_k\to0$, and
\[
u_k\in \nabla \varphi_{\rho_k,\sigma_k}(x_k)+\mathcal{N}_X(x_k),
\]
while the active limiting multiplier indices are strictly violated along the sequence:
\[
c_i(x_k,y_k^*)>0
\qquad\text{for all }k\text{ and all }i\in I^+(\bar\lambda):=\{j\in[p]:\bar\lambda_j>0\}.
\]
\end{definition}

The following theorem establishes Type-I enhanced KKT as a necessary condition for local minimizers of the MCC problem~\eqref{MCC}.
\begin{theorem}\label{thm:local-to-type1}
Assume that $\varphi$ is lower semicontinuous on $X$. Let $x^*$ be a local minimizer of the MCC problem~\eqref{MCC}, namely, there exists $\delta>0$ such that
$\varphi(x^*)\le \varphi(x),
\; \forall x\in X\cap B_{\delta}(x^*).$
Assume that GP\L{}CQ holds at every $(x^*,\hat{y})$ with
$\hat{y}\in\arg\max_{y\in Y}\{f(x^*,y)\mid c(x^*,y)\le 0\}.$
Then there exists a point $y^*\in\arg\max_{y\in Y}\{f(x^*,y)\mid c(x^*,y)\le 0\}$
such that $(x^*,y^*)$ is a Type-I enhanced KKT point of~\eqref{MCC}.
\end{theorem}

Adding a penalty-scale decay requirement to the residual sequence in Definition~\ref{def:type1-ekkt} yields a stronger stationarity notion.
\begin{definition}\label{def:type2-ekkt}
A Type-I enhanced KKT point $(\bar x,\bar y)$ is called a \emph{Type-II enhanced KKT point} if the sequence in Definition~\ref{def:type1-ekkt} can be chosen so that $\rho_k\|u_k\|\to 0$.
\end{definition}

By definition, every Type-II enhanced KKT point is a Type-I enhanced KKT point, while the converse need not hold. The stronger Type-II condition is necessary for strict local minimizers of the MCC problem~\eqref{MCC}.

\begin{theorem}\label{thm:strict-to-type2}
Assume that $\varphi$ is lower semicontinuous on $X$. Let $x^*$ be a strict local minimizer of~\eqref{MCC}, namely, there exists $\delta>0$ such that
$\varphi(x^*)<\varphi(x),
\;\forall x\in X\cap B_{\delta}(x^*)\setminus\{x^*\}.$
Assume that GP\L{}CQ holds at every $(x^*,\hat{y})$ with
$\hat{y}\in\arg\max_{y\in Y}\{f(x^*,y)\mid c(x^*,y)\le 0\}.$
Then there exists a point
$y^*\in\arg\max_{y\in Y}\{f(x^*,y)\mid c(x^*,y)\le 0\}$
such that $(x^*,y^*)$ is a Type-II enhanced KKT point of~\eqref{MCC}.
\end{theorem}

\begin{remark}
In general, the Type-II enhanced KKT condition is not necessary for all local minimizers. 
When the value function is locally flat, exact stationary points of the penalty-based approximations may converge only to a subset of stationary points of the value function. 
A simple example illustrating this phenomenon is
\[
\begin{aligned}
    \min_{x\in X}\max_{\substack{y\in Y\\ y\le 0}}
   -(y-1)^2+xy,
\end{aligned}
\]
where $X=[-1,1]$ and $Y=[-2,2]$. In this case, the value function is constant on $[-1,1]$. 
One can verify that $(0,0)$ is a local minimizer and a Type-I enhanced KKT point, but not a Type-II enhanced KKT point.
\end{remark}

Example~\ref{exm} shows that Type-I enhanced KKT may still include a spurious KKT point, whereas Type-II enhanced KKT rules it out through the stronger penalty-scale residual decay. The precise statement is as follows.
\begin{proposition}\label{prop:spurious-kkt}
For the MCC \eqref{example} in Example~\ref{exm}, the point $(\mathbf{0},\mathbf{0})$ has the following properties:
\begin{enumerate}
    \item $(\mathbf{0}, \mathbf{0})$ is a standard KKT point;
    \item $x=\mathbf{0}$ is not a local minimizer of the value function $\varphi$;
    \item $(\mathbf{0}, \mathbf{0})$ is a Type-I enhanced KKT point but not a Type-II enhanced KKT point.
\end{enumerate}
Consequently, Type-II enhanced KKT excludes this spurious KKT point.
\end{proposition}

\begin{remark}
The augmented Lagrangian method is often viewed as a hybrid of penalty and
multiplier methods. A standard augmented-Lagrangian reformulation of
\eqref{MCC}, as considered in~\cite{lu2023first}, is
\[
\min_{x\in X}\min_{\lambda \ge 0}\max_{y\in Y}\;
f(x,y) - \frac{1}{2\rho}
\left(\|[\lambda+\rho c(x,y)]_+\|^2-\|\lambda\|^2\right).
\]
However, for any $\rho>0$, one can directly verify that the triple
$(x,y,\lambda)=(\mathbf 0,\mathbf 0,1)$
is stationary for the augmented-Lagrangian minimax system.
Thus, at the stationarity level, the augmented-Lagrangian reformulation does not exclude this spurious point.
\end{remark}

Finally, the following result shows that the enhanced KKT notions are intrinsic to the penalty-based approximation scheme: they arise as limiting stationarity conditions of the smooth approximation problems.

\begin{theorem}\label{thm:approx-to-type1}
Let $\rho_k\to\infty$ and $\sigma_k\to 0$, and let $\{x_k\}\subset X$ satisfy
\[
u_k\in \nabla \varphi_{\rho_k,\sigma_k}(x_k)+\mathcal{N}_X(x_k),
\qquad
u_k\to 0.
\]
Assume that $x_k\to \bar x$, and let $\bar y$ be an accumulation point of
$\{y^*_{\rho_k,\sigma_k}(x_k)\}$. If GP\L{}CQ holds at $(\bar x,\bar y)$, then
$(\bar x,\bar y)$ is a Type-I enhanced KKT point of~\eqref{MCC}. Moreover, if the residuals can be chosen so that
$\rho_k\|u_k\|\to 0,$
then $(\bar x,\bar y)$ is a Type-II enhanced KKT point of~\eqref{MCC}.
\end{theorem}


\section{Single-Loop Stochastic Gradient Algorithm}\label{convergence analysis}

In the previous section, we introduced a sequence of smooth problems $\min_{x\in X} \; \varphi_{\rho_k,\sigma_k}(x)$ to approximate the MCC. The smoothness of $\varphi_{\rho_k,\sigma_k}(x)$ facilitates the application of gradient-based optimization methods. Building on this, we now develop a practical single-loop stochastic gradient algorithm, named SPACO, to solve the stochastic MCC~\eqref{MCC}.
\subsection{Algorithm Description}

We now describe the proposed Stochastic Penalty-based Algorithm for minimax
optimization with COupled constraints (SPACO). For brevity, write
\(\LL_k=\LL_{\rho_k,\sigma_k}\) and
\(\varphi_k=\varphi_{\rho_k,\sigma_k}\).

\textbf{Motivation.}
For the smooth approximation problem \(\min_{x\in X}\varphi_k(x)\), the ideal
projected-gradient step is
\begin{equation*}
    x^{k+1} = \P_X\!\left(x^k - \alpha_k \nabla \varphi_k(x^k)\right).
\end{equation*}
By Proposition~\ref{differentiable},
\[
\nabla \varphi_k(x^k)=\nabla_x\LL_k(x^k,y_k^*(x^k)),
\qquad
y_k^*(x):=y^*_{\rho_k,\sigma_k}(x).
\]
Thus, an exact implementation would require solving the inner maximization
problem at every iteration. SPACO avoids this nested-loop computation by
tracking \(y_k^*(x^k)\) with one projected ascent step and then using the
tracked point to form an inexact stochastic descent direction for \(x\).

\textbf{Inner tracking step.}
Given \((x^k,y^k)\), SPACO first performs one stochastic projected-ascent step
for the regularized penalized inner problem. Draw \(\xi_k^y\sim D\) and set
\begin{equation}\label{direction_y}
    d_y^k = \nabla_y \Psi_k(x^k, y^k; \xi_k^y),
\end{equation}
where 
$\Psi_k(x, y; \xi) := F(x, y; \xi) - \frac{\rho_k}{2}\big\|[c(x,y)]_+\big\|^2 - \frac{\sigma_k}{2}\|y\|^2$
is the stochastic counterpart of \(\LL_k(x,y)\). The inner variable is updated by
\begin{equation*}
    y^{k+1} = \P_Y\!\left(y^k + \beta_k d_y^k\right),
\end{equation*}
where \(\beta_k>0\) is the inner stepsize.

\textbf{Outer inexact projected-gradient step.}
The point \(y^{k+1}\) is then used in place of the exact maximizer
\(y_k^*(x^k)\). To control the stochastic error in the resulting outer
direction, SPACO employs a momentum-based variance-reduced estimator
\cite{cutkosky2019momentum}. The estimator is initialized and updated as
\begin{equation}\label{direction_x_iterative average} 
\begin{aligned}
        d_x^0 &= \nabla_x \Psi_0(x^0, y^1; \xi_0^x),\\
        d_x^k &= (1 - \eta_k) \left( d_x^{k-1} - \nabla_x \Psi_{k-1}(x^{k-1}, y^k; \xi_k^x) \right) + \nabla_x \Psi_k(x^k, y^{k+1}; \xi_k^x), \qquad k \geq 1,
\end{aligned}
\end{equation} 
where \(\eta_k\in(0,1]\) is the momentum parameter. For \(k\ge1\), the two
\(x\)-gradient queries in \eqref{direction_x_iterative average} use the same
sample \(\xi_k^x\), drawn independently of \(\xi_k^y\). The outer variable is
updated by
\begin{equation*}
    x^{k+1} = \P_X\!\left(x^k - \alpha_k d_x^k\right),
\end{equation*}
where \(\alpha_k>0\) is the outer stepsize.

\textbf{Parameter strategy and theoretical role.}
The penalty approximation requires \(\rho_k\to\infty\) and \(\sigma_k\to0\).
Consequently, the smooth approximation, its curvature, and the inner
maximizer \(y_k^*(x)\) all vary with \(k\). The stepsizes
\(\alpha_k,\beta_k\) and momentum parameters \(\eta_k\) are therefore chosen
jointly with \(\rho_k,\sigma_k\). The subsequent analysis controls the
projected-gradient residual for \(\varphi_k\), the tracking error
\(\|y^k-y_k^*(x^k)\|\), and the variance of the estimator \(d_x^k\).
The complete single-loop implementation is summarized in
Algorithm~\ref{algorithm:spacom}.

\begin{algorithm}[htbp]
  \DontPrintSemicolon
  \caption{\textbf{S}tochastic \textbf{P}enalty-based \textbf{A}lgorithm for minimax optimization with \textbf{CO}upled constraints (SPACO)}
  \label{algorithm:spacom}
  \KwIn{Initial points $(x^0, y^0) \in X \times Y$, penalty parameters $\{\rho_k\}$, regularization parameters $\{\sigma_k\}$, stepsizes $\alpha_k,\beta_k > 0$, and momentum parameters $\{\eta_k\}$.}
  \For{$k = 0,1,\dots$}{
    Sample $\xi_k^y \sim D$, compute $d_y^k$ by \eqref{direction_y}, and update
    \[
    y^{k+1} = \P_{Y}(y^k + \beta_k d_y^k).
    \]
    \;\vspace{-\baselineskip}
    Sample $\xi_k^x \sim D$ independently of $\xi_k^y$.\;
    For $k\ge1$, use the same $\xi_k^x$ in both $x$-gradient queries in \eqref{direction_x_iterative average}.\;
    Compute $d_x^k$ by \eqref{direction_x_iterative average} and update
    \[x^{k+1} = \P_{X}(x^k - \alpha_k d_x^k).\]
  }
\end{algorithm}

\subsection{Assumptions and Notation for the Stochastic Analysis}
We collect the additional assumptions and constants used in the convergence
analysis of SPACO. Besides the standing assumptions, we require compactness of $X$.

\begin{assumption}\label{assum:Xcompact}
    The set $X$ is compact.
\end{assumption}

Let \(\F_k\) denote the \(\sigma\)-algebra generated by all samples observed
before iteration \(k\), and let \(\F_{k+\frac{1}{2}}\) further include the
current sample \(\xi_k^y\):
\[
\F_k = \sigma \{ \xi_0^y, \xi_0^x, \ldots, \xi_{k-1}^y, \xi_{k-1}^x \},
\qquad
\F_{k+\frac{1}{2}} = \sigma \{ \F_k, \xi_{k}^y \}.
\]
We impose the following standard stochastic-oracle conditions.
\begin{assumption}\label{assumgradient}
For all \((x,y)\in X\times Y\), the stochastic gradients
\(\nabla_x F(x,y;\xi)\) and  \allowbreak\
\(\nabla_y F(x,y;\xi)\) satisfy:
\begin{enumerate}
\item \emph{Unbiasedness and independence.}
The sample \(\xi_k^y\) is independent of \(\F_k\), the sample \(\xi_k^x\)
is independent of \(\F_{k+\frac{1}{2}}\), and
\begin{align*}
        &\mathbb{E}_{\xi \sim D} \left[\nabla_x F(x,y;\xi) \right] = \nabla_x f(x,y),\;
        \mathbb{E}_{\xi \sim D} \left[\nabla_y F(x,y;\xi) \right] = \nabla_y f(x,y).
\end{align*}
\item \emph{Bounded variance.}
There exists \(\delta>0\) such that
\begin{align*}
        &\mathbb{E}_{\xi \sim D}\left[\|\nabla_x F(x,y;\xi) -\nabla_x f(x,y)\|^2 \right] \le \delta^2, \;\mathbb{E}_{\xi \sim D}\left[\|\nabla_y F(x,y;\xi) -\nabla_y f(x,y)\|^2 \right] \le \delta^2.
\end{align*}
\end{enumerate}
\end{assumption}

\begin{assumption}\label{stochasticsmooth}
The stochastic oracle \(\nabla_x F(x,y;\xi)\) admits simultaneous queries:
using the same sample \(\xi\), the algorithm can evaluate stochastic
\(x\)-gradients at two distinct points \((x_1,y_1)\) and \((x_2,y_2)\).
Moreover, for some constant \(L_F>0\),
\begin{align*}
    &\E_{\xi}\!\left[\|\nabla_x F(x_1,y_1;\xi)-\nabla_x F(x_2,y_2;\xi)\|^2\right]
    \le {L}_F^2\!\left(\|x_1-x_2\|^2+\|y_1-y_2\|^2\right).
\end{align*}
\end{assumption}

Let \(L_f\) and \(L_c\) denote Lipschitz constants of \(\nabla f\) and
\(\nabla c\) on \(X\times Y\), respectively. By enlarging \(L_f\) if
necessary, we assume \(L_f\ge L_F\); hence \(L_f\) serves as a common bound for
the deterministic smoothness and the stochastic mean-squared smoothness used
below. Since \(X\) and \(Y\) are compact, define
\begin{align*}
    M := \max \left\{
    \sup_{x \in X, y \in Y} \|\nabla f(x, y)\|,
    \sup_{x \in X, y \in Y} \|\nabla c(x, y)\|,
    \sup_{x \in X, y \in Y} |f(x, y)|,
    \sup_{x \in X, y \in Y} \|c(x, y)\|
    \right\},
\end{align*}
and set \(D_x:=\sup_{x \in X}\|x\|\) and
\(D_y:=\sup_{y\in Y}\|y\|\).

\subsection{Auxiliary Lemmas}
We next provide the estimates used to prove convergence of SPACO. These
lemmas control, respectively, the smoothness of the penalized objective, the
tracking error of the inner variable, the descent of the approximate value
function, and the variance of the recursive outer-gradient estimator.

First, we establish the smoothness of the regularized penalized objective and a
uniform lower bound for the approximate value function.
\begin{lemma}\label{gradientLipschitz}
Let $\rho_k\to\infty$ and $\sigma_k\to 0$ as $k\to\infty$. Then
$\nabla \LL_k(x,y)$ is $L_k$-Lipschitz continuous on $X \times Y$ with
constant
\[
L_k:= L_f+\rho_k pML_c+\rho_k pM^2+\sigma_k .
\]
Moreover, there exists a constant \(\underline{\varphi}\) such that
\(\varphi_k(x)\ge \underline{\varphi}\) for all \(x\in X\) and all \(k\).
\end{lemma}

The next lemma gives the one-step contraction of the inner tracking error.

\begin{lemma}\label{lineardescent}
Suppose that the step-size sequence $\{\beta_k\}$ satisfies
$0<\beta_k\le \frac{1}{L_k}$ for each $k$.
Let $\{(x^k, y^k)\}$ be the sequence generated by Algorithm~\ref{algorithm:spacom}.
Then the iterates $y^k$ and $y^{k+1}$ satisfy
\begin{equation}\label{udescentequ}
\E[\|y^{k+1}- y_k^*(x^k)\|^2\mid \F_k]
\le(1-\sigma_k \beta_k)\|y^{k}-y_k^*(x^k)\|^2+\beta_k^2 \delta^2.
\end{equation}
\end{lemma}

Because \(x^k\), \(\rho_k\), and \(\sigma_k\) vary with \(k\), the exact
maximizer \(y_k^*(x^k)\) also moves along the trajectory. The following bound gives a recursion for the tracking error.

\begin{lemma}\label{ydescentlemma}
Let $\{(x^k,y^k)\}$ be the sequence generated by Algorithm~\ref{algorithm:spacom}.
Suppose that $\{\rho_k\}$ and $\{\sigma_k\}$ satisfy
$\rho_{k+1}\ge \rho_k>0$ and $\sigma_k\ge\sigma_{k+1}>0$, and that the
stepsize sequence $\{\beta_k\}$ satisfies \(0<\beta_k\le 1/L_k\) for each
\(k\). Then
\begin{equation}
    \begin{aligned}
    &\E[\|y^{k+1}-y_{k+1}^*(x^{k+1})\|^2\mid \F_k]-\|y^{k}-y_{k}^*(x^{k})\|^2\\
    \le\;&-\frac{1}{2}\beta_k\sigma_k\|y^{k}-y_{k}^*(x^{k})\|^2+(1+\frac{1}{2}\beta_k\sigma_k)\beta_k^2\delta^2+2(1+\frac{2}{\beta_k\sigma_k})\frac{L_k^2}{\sigma_k^2}\E[\|x^{k+1}-x^{k}\|^2\mid \F_k] \\
            & +2(1+\frac{2}{\beta_k\sigma_k})\left(\frac{2(\rho_{k+1}-\rho_{k})^2}{\sigma_k^2}M^4+\frac{2(\sigma_{k}-\sigma_{k+1})^2}{\sigma_k^2} D_y^2\right).
\end{aligned}
\end{equation}
\end{lemma}

We then relate the outer projected-gradient step to descent of the approximate
value function. The bound below separates the effects of inner tracking error,
stochastic gradient error, and the change in the regularization parameter.
\begin{lemma}\label{xdescentlemma}
Let $\{\rho_k\}$ and $\{\sigma_k\}$ be sequences such that
$\rho_{k+1} \ge \rho_{k}>0$ and $\sigma_{k} \ge\sigma_{k+1}>0$. Suppose the
step-size sequence $\{\beta_k\}$ satisfies
$0<\beta_k\le \frac{1}{L_k}$ for each $k$. Then
    \begin{equation}
    \begin{aligned}
        &\E[\varphi_{k+1}(x^{k+1})\mid \F_{k}]-\varphi_{k}(x^k)+\left( \frac{1}{4\alpha_k} - \frac{L_{\varphi_k}}{2}\right) \E[\|x^{k+1}-x^{k}\|^2\mid \F_{k}] \\
        \le\;&\frac{\alpha_k}{2}L_k^2\|y^{k}-y_{k}^*(x^k)\|^2+\frac{\alpha_k}{2}L_k^2\beta_k^2\delta^2+\alpha_k \E[\|e_x^k\|^2\mid \F_{k}]+\frac{1}{2}(\sigma_{k}-\sigma_{k+1})D_y^2.
    \end{aligned}
    \end{equation}
\end{lemma}

It remains to control the estimator error appearing in the previous descent
inequality. The recursive structure of \(d_x^k\) yields the following variance
bound.

\begin{lemma}\label{variancereduction}
Let $\{(x^k,y^k)\}$ be the sequence generated by Algorithm~\ref{algorithm:spacom},
and define
\[
e_x^k:=d_x^k-\nabla_x\LL_k(x^k,y^{k+1}).
\]
Suppose that $\{\rho_k\}$ and $\{\sigma_k\}$ satisfy
$\rho_{k+1}\ge \rho_k>0$ and $\sigma_k\ge\sigma_{k+1}>0$, that
\(0<\beta_k\le 1/L_k\), and that \(0\le \eta_k\le 1\). Then, for every
\(k>0\),
\begin{equation}
    \begin{aligned}
        \E[\|e_x^{k}\|^2\mid \F_{k}]\le\;&(1-\eta_{k})^2\|e_x^{k-1}\|^2+2\eta_{k}^2\delta^2+18L_{k}^2\left(\|x^{k}-x^{k-1}\|^2+3\beta_{k}^2\delta^2\right.\\&\left.+3\beta_{k}^2L_{k}^2\|y^{k}-y^*_{k}(x^{k})\|^2
        +3\beta_{k}^2(M+\rho_{k}pM^2+\sigma_k D_y )^2\right),
    \end{aligned}
\end{equation}
\end{lemma}

\subsection{Convergence Results}

We now combine the auxiliary estimates into convergence guarantees for SPACO.
The analysis is based on the following merit function, defined for \(k\ge1\):
\begin{equation}\label{V_def}
    \begin{aligned}
        V_k = \, &a_k(\varphi_k(x^k) - \underline{\varphi})
        + b_k\| y^{k}- y_{k}^*(x^k)\|^2
        + c_k\|e_x^{k-1}\|^2
        + d_k\|x^k-x^{k-1}\|^2,
    \end{aligned}
\end{equation}
where \(a_k,b_k,c_k,d_k>0\) are iteration-dependent coefficients,
\(e_x^{k}:=d_x^k-\nabla_x\LL_k(x^k,y^{k+1})\) is the error induced by the
inexact outer-gradient estimator, and \(\underline{\varphi}\) is the uniform
lower bound from Lemma~\ref{gradientLipschitz}.

By balancing these coefficients with the dynamic stepsizes and penalty
parameters, the preceding lemmas yield the following one-step descent
inequality.
\begin{proposition}\label{lyapunovproposition}
Let $\{(x^k, y^k)\}$ be generated by SPACO
(Algorithm~\ref{algorithm:spacom}) with
\(\sigma_k=\sigma_0(k+1)^{-t}\), \(\rho_k=\rho_0(k+1)^t\),
\(\alpha_k=\alpha_0(k+1)^{-6t-s}\),
\(\beta_k=\beta_0(k+1)^{-t-s}\), and
\(\eta_k=\eta_0(k+1)^{-s}\), where
\(\alpha_0,\beta_0,\eta_0,\sigma_0,\rho_0,t,s>0\). Set
\(a_k=(k+1)^{-2t}\), \(b_k=(k+1)^{-3t}\),
\(c_k=(k+1)^{-7t}\), and \(d_k=(k+1)^{-4t}\). If
\(0<t,s<1\), \(s>3t\), and \(8t+s<1\), then, for all sufficiently large
\(k\),
\begin{equation*}
    \E[V_{k+1}\mid \F_k]-V_{k}
    \le-\frac{a_k\alpha_k}{24}\E[\|\mathcal{G}_k(x^k)\|^2\mid \F_k]
    -\frac{b_k\beta_k\sigma_k}{4}\|y^{k}-y_{k}^*(x^k)\|^2
    +Cb_k\beta_k^2\delta^2 +\zeta_k ,
\end{equation*}
where
\(\mathcal{G}_k(x)=\frac{1}{\alpha_{k}}(x-P_{X}(x-\alpha_{k}\nabla \varphi_k(x)))\)
is the generalized gradient residual for \(\min_{x\in X}\varphi_k(x)\),
\(C>0\), and \(\{\zeta_k\}\) is summable and nonnegative.
\end{proposition}

Proposition~\ref{lyapunovproposition} is the main descent estimate used below.
It yields both finite-time complexity bounds and asymptotic convergence
guarantees.

For a possibly random triple taking values in
\(X\times Y\times\mathbb R_+^p\), define the expected reduced KKT residual
associated with the projected KKT system by
\[
\begin{aligned}
\mathcal{R}_{\rm KKT}(x,y,\lambda)
:= \max\Bigl\{&
\E\bigl[\|x-\mathcal P_X(x-\nabla_x \ML(x,y,\lambda))\|\bigr],\\
&\E\bigl[\|y-\mathcal P_Y(y+\nabla_y \ML(x,y,\lambda))\|\bigr],\;
\E\bigl[\|[c(x,y)]_+\|\bigr]
\Bigr\}.
\end{aligned}
\]
The expectations are taken with respect to all randomness in the triple.
This residual omits the complementarity component
\(|\lambda^\top c(x,y)|\) and should therefore be viewed as a reduced,
or equivalently complementarity-free, KKT residual.
Similar complementarity-free KKT residuals have also been considered
in the literature; see, e.g., \cite{zhang2024zeroth,zhang2024alternating}.

The next theorem gives non-asymptotic rates for stationarity and feasibility,
and, with the above modified definition, for the expected reduced KKT residual.
\begin{theorem}\label{convergethm}
Let $\{(x^k,y^k)\}$ be the sequence generated by SPACO
(Algorithm~\ref{algorithm:spacom}) with parameters selected as in
Proposition~\ref{lyapunovproposition}, and suppose that \(2s+5t\neq1\).
Then, for any \(\epsilon>0\), the number of iterations required to find an
iterate \(k\) satisfying
\[
\max\Big\{
\E[\|\mathcal{G}_k(x^k)\|],\;
\E[\|y^k-y_k^*(x^k)\|],\;
\E[\|[c(x^k,y^k)]_+\|]
\Big\}\le \epsilon,
\]
is at most
\(\mathcal O\big(\epsilon^{-\frac{1}{\tau}}\big)\),
where
$\tau:=\min\Big\{\frac{1-8t-s}{2},\;\frac{s-3t}{2},\;\frac{t}{2}\Big\}.$

Moreover, with \(\bar\lambda_{k}:=\rho_{k}[c(x^{k},y^{k})]_+\), the number of
iterations required to find an iterate \(k\) such that
\[
\mathcal{R}_{\rm KKT}(x^k,y^k,\bar\lambda_k) \le \epsilon
\]
is at most \(\mathcal O(\epsilon^{-\frac{1}{\tau}})\).
\end{theorem}
We next state the asymptotic consequence of
Proposition~\ref{lyapunovproposition}. Combined with the limiting results in
Section~\ref{approximaiton}, the descent estimate implies that accumulation
points of SPACO satisfy enhanced KKT conditions.
\begin{theorem}\label{thm:converge_to_enhanced_kkt}
Let $\{(x^k, y^k)\}$ be the sequence generated by SPACO
(Algorithm~\ref{algorithm:spacom}) with parameters selected as in
Proposition~\ref{lyapunovproposition}, and suppose that \(2s+5t\neq1\).
Then, almost surely, there exists a convergent subsequence
$\{(x^{k_i}, y^{k_i})\}$ such that
\[
\lim_{i\to\infty}\|\mathcal{G}_{k_i}(x^{k_i})\|
=
\lim_{i\to\infty}\|y^{k_i}-y_{k_i}^*(x^{k_i})\|
=0.
\]
If such a subsequence converges to $(\bar x,\bar y)$ and GP\L{}CQ holds at
$(\bar x,\bar y)$, then $(\bar x,\bar y)$ is a Type-I enhanced KKT point of
the MCC~\eqref{MCC}. Furthermore, if $s>5t$ and $10t+s<1$, then this
subsequence can be chosen so that
\[
\lim_{i\to\infty}\rho_{k_i}\|\mathcal{G}_{k_i}(x^{k_i})\|
=
\lim_{i\to\infty}\rho_{k_i}\|y^{k_i}-y_{k_i}^*(x^{k_i})\|
=0,
\]
and, under the same GP\L{}CQ assumption, $(\bar x,\bar y)$ is a Type-II
enhanced KKT point of the MCC~\eqref{MCC}.
\end{theorem}

\section{Numerical Experiments}\label{sec:experiments}
In this section, we evaluate the convergence behavior of SPACO on controlled synthetic problems and illustrate the application potential of the constrained minimax formulation through two representative machine-learning applications: fairness-aware classification and GAN training.
Additional implementation details are provided in Appendix~\ref{app:experiment-details}, and the source code is available online\footnote{\url{https://github.com/zyxamos/SPACO}}.
All experiments were run on a server with dual Intel Xeon Gold 5218R CPUs and two NVIDIA H100 PCIe GPUs.

\subsection{Synthetic Examples}

We first test SPACO on synthetic examples with known solutions.
The experiments consist of the deterministic Example~\ref{exm} and two stochastic problems with nonlinear and linear coupled constraints, respectively.
Across the two stochastic examples, we fix the dimension at $n=100$, set the noise variance to $\delta^2=1$, and sample the initial point uniformly from the feasible set over $10$ independent runs.
We compare SPACO with MGD~\cite{tsaknakis2023minimax}, GBAL~\cite{hu2026convergence}, MMPen~\cite{hu2024minimization}, and gradient descent-ascent with a fixed penalty (GDA-FP).
Performance is assessed by the solution error $\epsilon_x=\tfrac{\|x-x^*\|^2}{1 + \|x^0-x^*\|^2}$ and the inner problem error $\epsilon_y=\tfrac{\|y-y^*(x)\|^2}{1 + \|y^0-y^*(x)\|^2}$.

\textbf{Deterministic nonlinear case.} 
The first problem is a two-dimension deterministic instance with nonlinear coupled constraints, illustrated in Example~\ref{exm}.
Its min--min--max reformulation admits a spurious KKT point.
We run the algorithms with initial points sampled uniformly from the feasible set with mesh size $0.1$.

\textbf{Stochastic nonlinear case.}
The second example is a high-dimensional stochastic extension of Example~\ref{exm}:
\begin{equation}\label{eq:nonlinear-toy}
\begin{aligned}
    \min_{x\in X_{1}}\max_{y\in \Gamma_{1}(x)}
     \E_w&\left[\frac{n}{2}(\frac{\|x\|^2}{n}-1)^2 + 2 \left( \|x\|^2 - \frac{(e^{\top} x)^2}{n}\right) -\frac{\|y-\mathbf{e}\|^2}{2}+\frac{x^\top(y+w)}{2}\right],
\end{aligned}
\end{equation}
where $X_{1}=[-\tfrac{5}{4}, \tfrac{5}{4}]^n$ and $Y_1:=[-10, 10]^{n}$, the feasible set is 
\begin{equation*}
    \Gamma_{1}(x)=\left\{y\in Y_1 \mid \tfrac{1}{\sqrt{n}} \left( e^\top y - \tfrac{(e^{\top} x)^2}{n} \right) \le 0\right\},
\end{equation*}
$w$ is a Gaussian random vector with independent entries $w_{i} \sim \mathcal{N}(0,\delta^2)$, and $\mathbf{e}\in \mathbb{R}^n$ denotes the all-ones vector. For a fixed $x$, the expected inner solution is given by
$y^*(x) = \tfrac{x}{2} + ( 1 - [ 1 + \frac{e^\top x}{2n} - ( \frac{e^\top x}{n} )^2 ]_+ ) e$
, yielding the expected optimal solution $(x^*,y^*)=(\alpha \mathbf{e}, (1 + \tfrac{\alpha}{2}) \mathbf{e})$, where $\alpha \approx -1.0545$ is the negative root of $8 \alpha^3 - 7 \alpha + 2 = 0$.

\textbf{Stochastic linear case.}
The third instance considers a linear coupled constraint:
\begin{equation}\label{eq:linear-toy}
    \begin{aligned}
        \min_{x\in X_{2}} \max_{y\in \Gamma_{2}(x)}\;&
        \E_{W}\left[
            \frac{1}{2}x^\top (\bar{A} + W) x
            - \left(
                \frac{1}{2}\|y_1\|^2 - x^\top y_1 + e^\top y_2 + \frac{1}{2}\|y_2+2x+e\|^2
            \right)
        \right],
    \end{aligned}
\end{equation}
where $X_{2}=[-10,10]^n$ and $Y_2=[-20,20]^{2n}$ with $y=(y_1, y_2)$, and the feasible set is
\begin{equation*}
    \Gamma_2(x)=\{y\in Y_2 \mid \tfrac{1}{\sqrt{n}} \left( e^\top x + e^\top y_1 + e^\top y_2 \right) = 0\},
\end{equation*}
$W$ is a Gaussian random matrix with independent entries $W_{ij} \sim \mathcal{N}(0,\delta^2)$, and $\bar{A} \in \IR^{n\times n}$ satisfying $\bar{A}+I\succ 0$ is generated as $\bar{A}=\frac{1}{n}MM^\top + I_n$, where the entries of $M\in\IR^{n\times n}$ are drawn from $\mathcal{N}(0,1)$. For a fixed $x$, the inner solutions are $(y_1^*(x),y_2^*(x))=(x+e,-2x-e)$, which leads to the expected optimal solution $(x^*,y_1^*,y_2^*)=(-2(\bar{A}+I)^{-1}e,x^*+e,-2x^*-e)$.

\begin{figure}[t]
    \centering
    {   
        \small
        \hspace{1em}
        \legendItem[0.6]{spacoRed}{SPACO}~
        \legendItem[0.6]{gdaOrange}{GDA-FP}~
        \legendItem[0.6]{mgdGreen}{MGD}~
        \legendItem[0.6]{gbalBlue}{GBAL}~
        \legendItem[0.6]{mmpenPurple}{MMPen}
    }
    \vspace*{4pt}
    
    \begin{subfigure}[t]{0.35\linewidth}
        \centering
        \includegraphics[width=\linewidth]{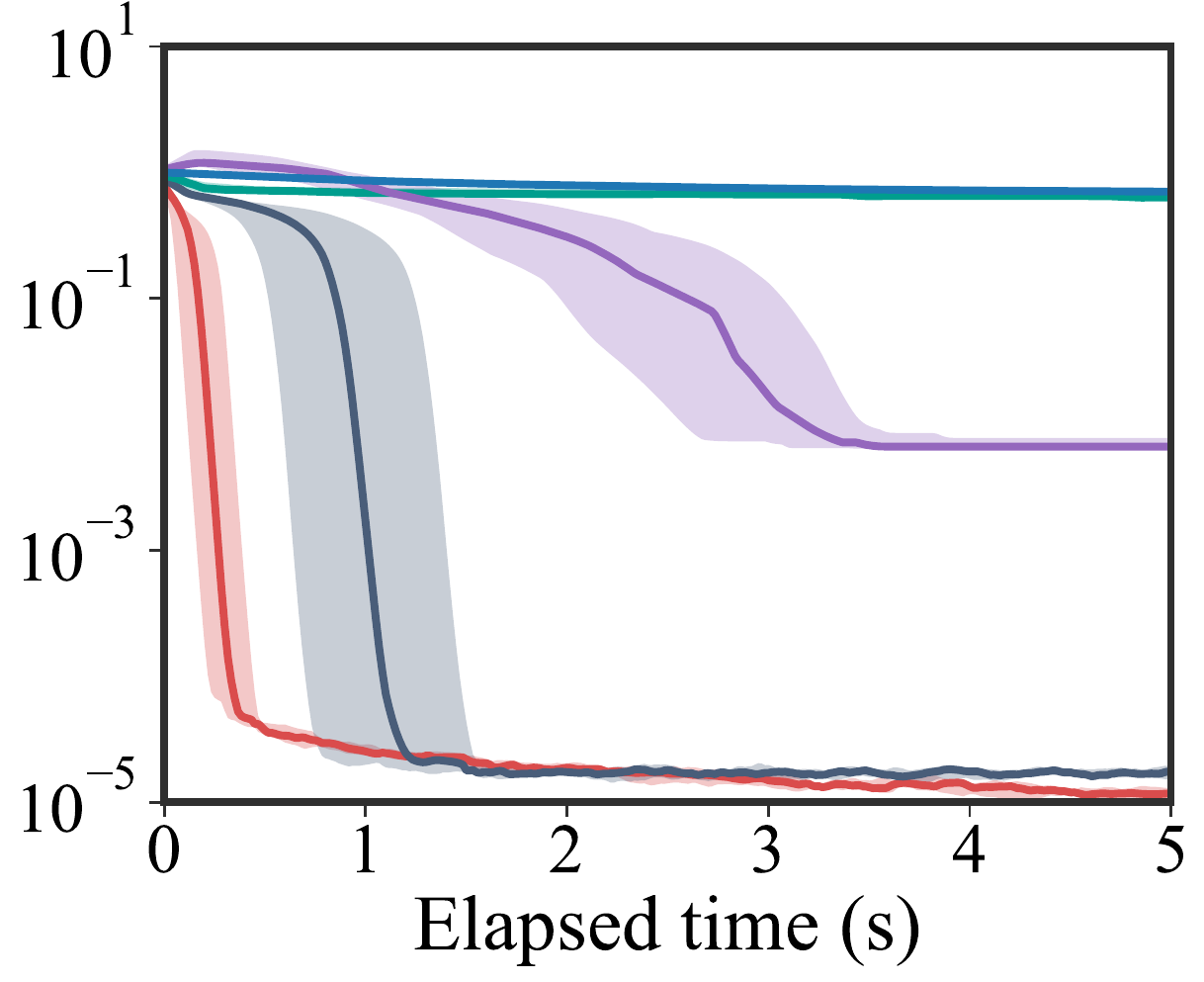}
        \caption{Solution error $\epsilon_x$}\label{fig:toy-nonlinear-x}
    \end{subfigure}
    \hspace{1em}
    \begin{subfigure}[t]{0.35\linewidth}
        \centering
        \includegraphics[width=\linewidth]{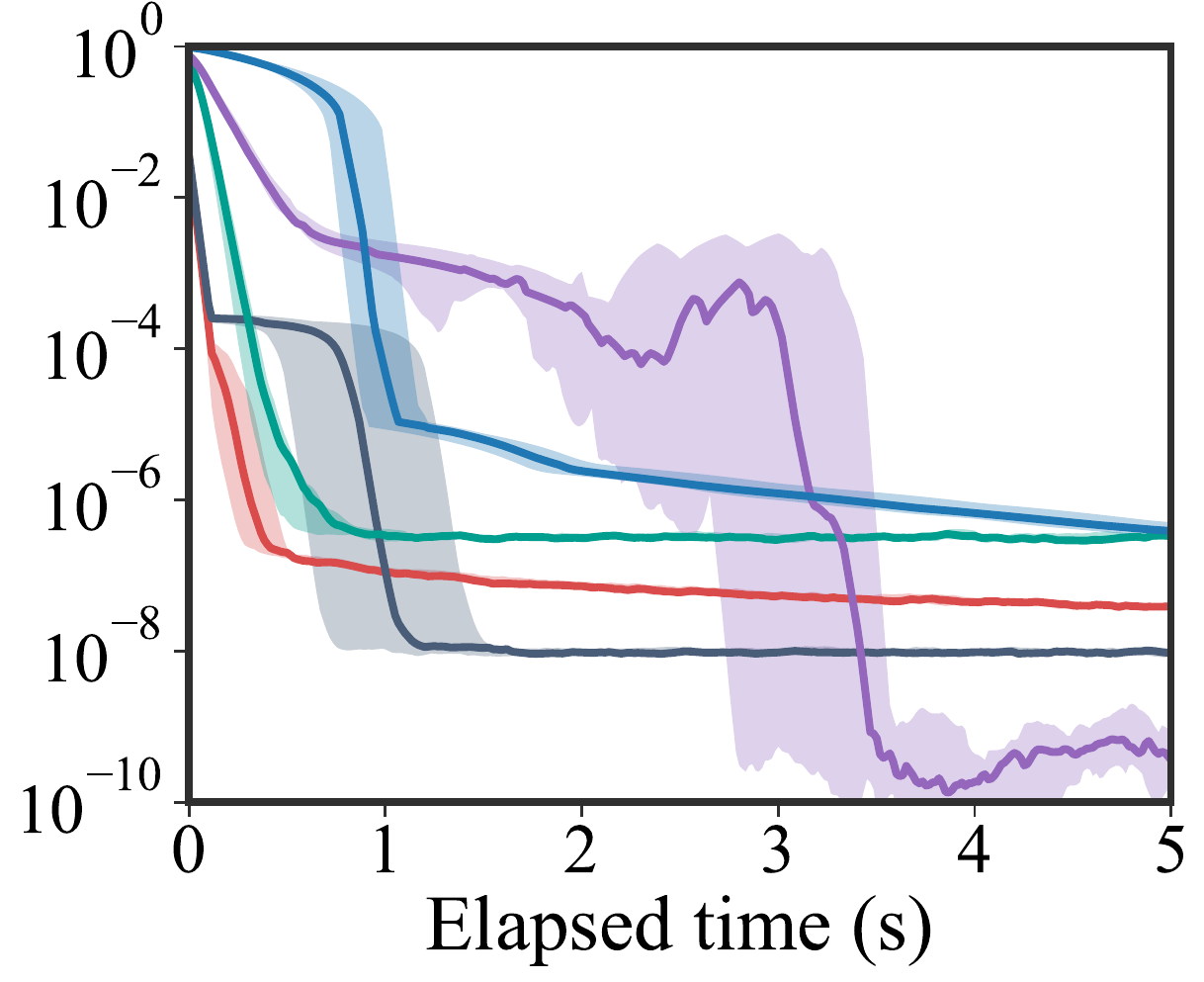}
        \caption{Inner problem error $\epsilon_y$}\label{fig:toy-nonlinear-y}
    \end{subfigure}
    \caption{
        Convergence curves for solving the nonlinearly constrained example~\eqref{eq:nonlinear-toy}.
    }\label{fig:toy-nonlinear}
\end{figure}

\begin{figure}[t]
    \centering
    {
        \small
        \hspace{1em}
        \legendItem[0.6]{spacoRed}{SPACO}~
        \legendItem[0.6]{gdaOrange}{GDA-FP}~
        \legendItem[0.6]{mgdGreen}{MGD}~
        \legendItem[0.6]{gbalBlue}{GBAL}~
        \legendItem[0.6]{mmpenPurple}{MMPen}
    }
    \vspace*{4pt}

    \begin{subfigure}[t]{0.35\linewidth}
        \centering
        \includegraphics[width=\linewidth]{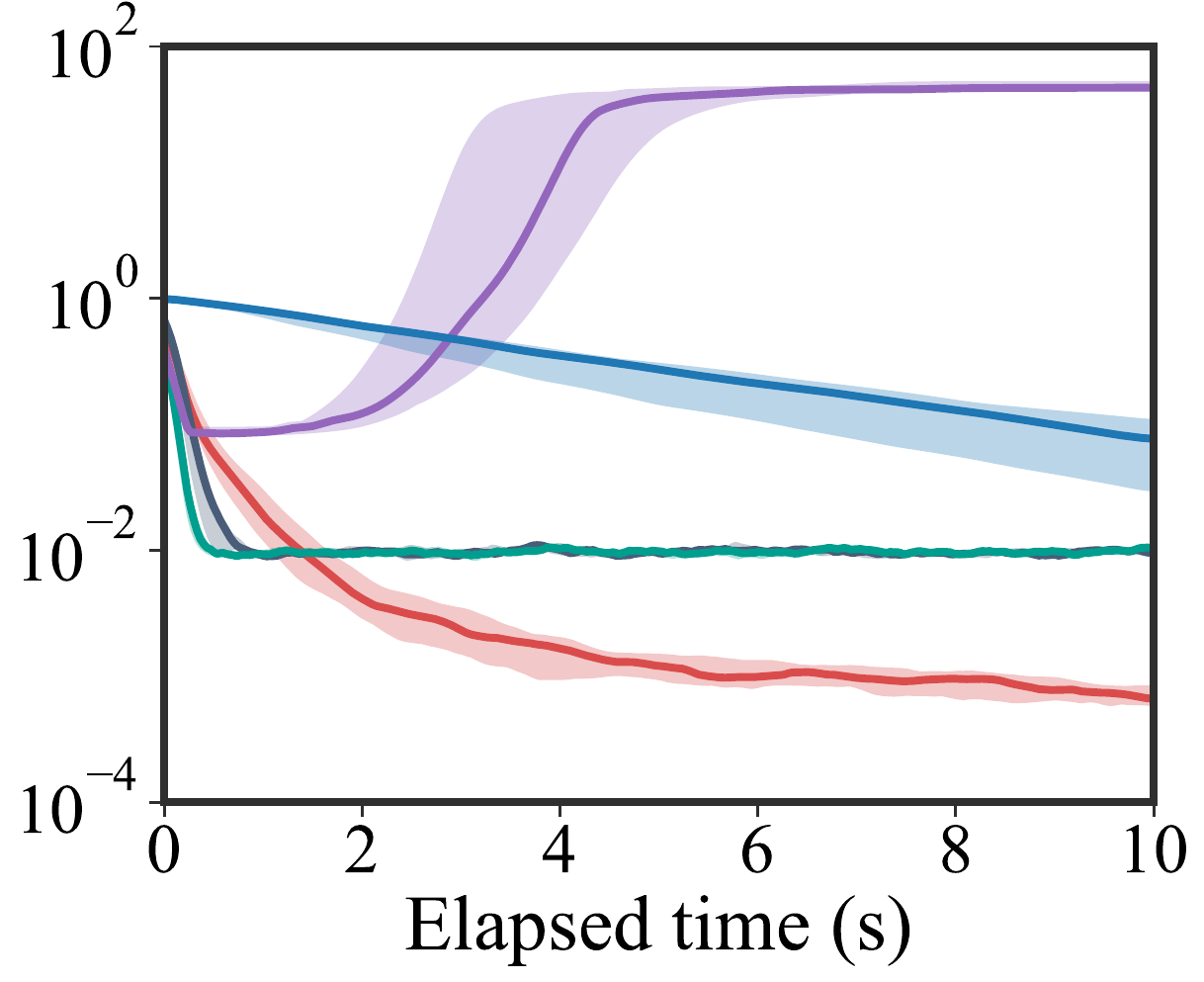}
        \caption{Solution error $\epsilon_x$}\label{fig:toy-linear-x}
    \end{subfigure}
    \hspace{1em}
    \begin{subfigure}[t]{0.35\linewidth}
        \centering
        \includegraphics[width=\linewidth]{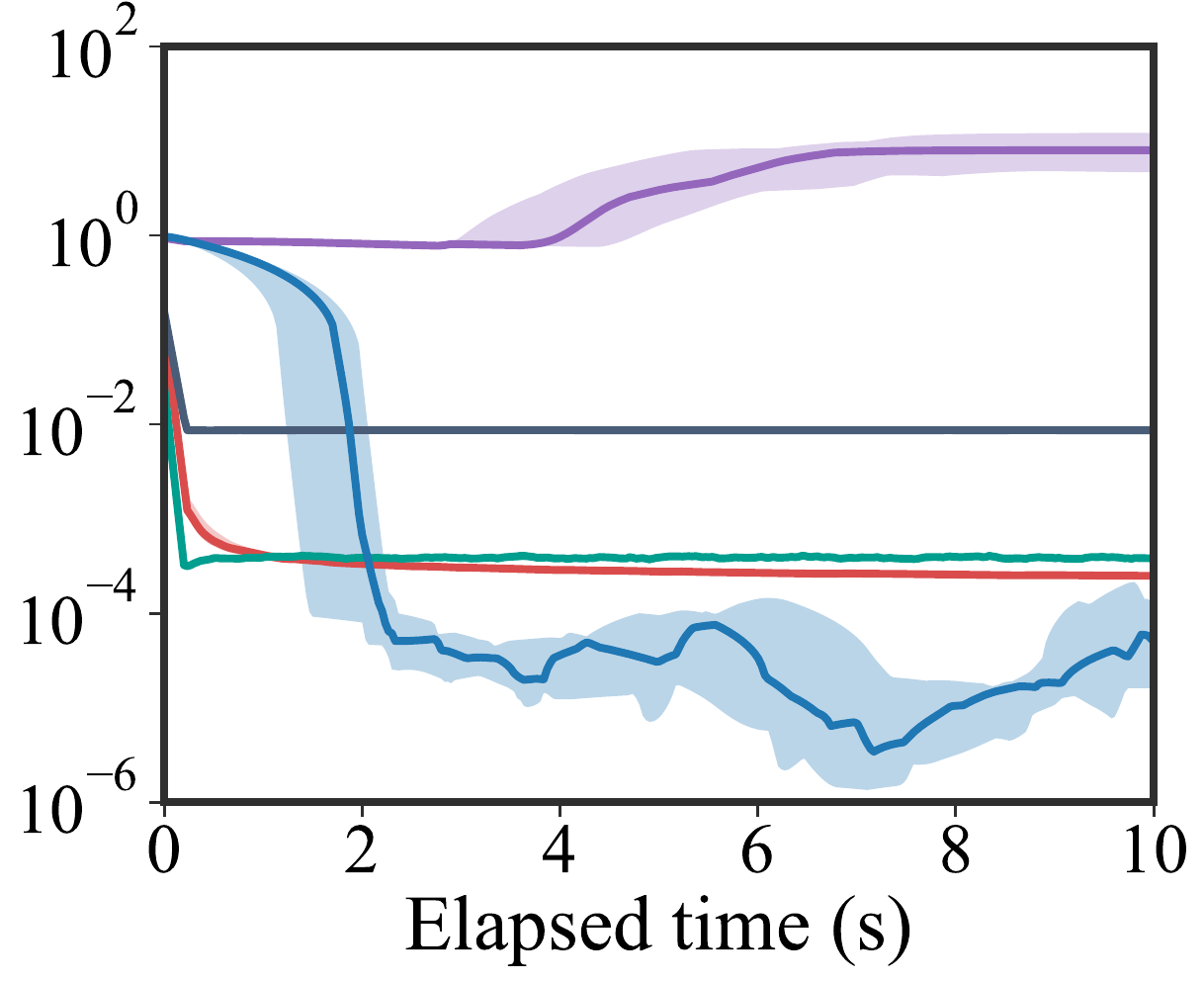}
        \caption{Inner problem error $\epsilon_y$}\label{fig:toy-linear-y}
    \end{subfigure}
    \caption{
        Convergence curves for solving the linearly constrained example~\eqref{eq:linear-toy}.
    }\label{fig:toy-linear}
\end{figure}

Figure~\ref{fig:intro-example} reports the convergence outcomes on deterministic Example~\ref{exm}: SPACO reaches the true solution from all tested initializations, whereas the min--min--max methods may be attracted to the spurious KKT point.
For the stochastic cases, Figures~\ref{fig:toy-nonlinear} and \ref{fig:toy-linear} show that SPACO drives the solution error $\epsilon_x$ to a lower value than the compared algorithms, which reflects the primary convergence progress because $x$ is the variable of the value function.
At the same time, SPACO achieves a comparable decrease in the inner problem error $\epsilon_y$, indicating that the inner problem is solved accurately enough to support the outer descent direction.
In the stochastic nonlinear example, which inherits the spurious-stationarity geometry of Example~\ref{exm}, the min--min--max methods are more prone to be affected by the spurious KKT points, while SPACO decreases $\epsilon_x$ faster than GDA-FP among the penalty-based methods.
In the linear example, where the coupled constraint is well structured, SPACO still solves the problem efficiently and accurately.

To assess the sensitivity to algorithmic parameters, Table~\ref{tab:ablation} reports the iterations required to achieve $\max\{\epsilon_x,\epsilon_y\}\le 10^{-4}$ on the nonlinearly constrained example~\eqref{eq:nonlinear-toy}.
Within the tested range, the weak dependence on the initial penalty $\rho_0$ and regularization $\sigma_0$ suggests that these parameters mainly serve as initial values of the dynamic schedule, rather than fixed modeling constants.
By contrast, the exponents $(t,s)$ have a more visible effect because they control the balance between feasibility enforcement, conditioning of the smoothed value function, and stochastic tracking accuracy, in line with the trade-off reflected in Theorem~\ref{convergethm}.

\begin{table}
    \centering
    \caption{Parameter sensitivity of SPACO on the nonlinearly constrained example~\eqref{eq:nonlinear-toy}. The entries in the Iter. column are reported as $\mathrm{mean}_{\pm\mathrm{std}}$ over $10$ independent runs.}%
    \label{tab:ablation}
    \resizebox{\linewidth}{!}{%
        \setlength{\tabcolsep}{5pt}%
        \begin{tabular}{ccccccc @{}}
            \toprule
            $\alpha_0$ & $\beta_0$ & $\rho_0$ & $\sigma_0$ & $t$ & $s$ & Iter. \\
            \midrule
            $0.1$ & $0.1$ & $10$ & $10^{-4}$ & $0.05$ & $0.2$ & $1171_{\pm 52}$ \\
            \midrule
            \cellcolor{Gray}{$\mathbf{0.05}$} & $0.1$ & $10$ & $10^{-4}$ & $0.05$ & $0.2$ & $4299_{\pm 230}$ \\
            \cellcolor{Gray}{$\mathbf{0.2}$}  & $0.1$ & $10$ & $10^{-4}$ & $0.05$ & $0.2$ & $452_{\pm 29}$ \\
            $0.1$ & \cellcolor{Gray}{$\mathbf{0.05}$} & $10$ & $10^{-4}$ & $0.05$ & $0.2$ & $1349_{\pm 70}$ \\
            $0.1$ & \cellcolor{Gray}{$\mathbf{0.2}$}  & $10$ & $10^{-4}$ & $0.05$ & $0.2$ & $1280_{\pm 118}$ \\
            $0.1$ & $0.1$ & \cellcolor{Gray}{$\mathbf{5}$} & $10^{-4}$ & $0.05$ & $0.2$ & $1194_{\pm 68}$ \\
            $0.1$ & $0.1$ & \cellcolor{Gray}{$\mathbf{20}$} & $10^{-4}$ & $0.05$ & $0.2$ & $1199_{\pm 25}$ \\
            \bottomrule
        \end{tabular}%
        \hspace{2em}%
        \begin{tabular}{ccccccc @{}}
            \toprule
            $\alpha_0$ & $\beta_0$ & $\rho_0$ & $\sigma_0$ & $t$ & $s$ & Iter. \\
            \midrule
            $0.1$ & $0.1$ & $10$ & $10^{-4}$ & $0.05$ & $0.2$ & $1171_{\pm 52}$ \\
            \midrule
            $0.1$ & $0.1$ & $10$ & \cellcolor{Gray}{$\mathbf{10^{-5}}$} & $0.05$ & $0.2$ & $1171_{\pm 52}$ \\
            $0.1$ & $0.1$ & $10$ & \cellcolor{Gray}{$\mathbf{10^{-3}}$} & $0.05$ & $0.2$ & $1172_{\pm 53}$ \\
            $0.1$ & $0.1$ & $10$ & $10^{-4}$ & \cellcolor{Gray}{$\mathbf{0.04}$} & $0.2$ & $713_{\pm 26}$ \\
            $0.1$ & $0.1$ & $10$ & $10^{-4}$ & \cellcolor{Gray}{$\mathbf{0.06}$} & $0.2$ & $2224_{\pm 95}$ \\
            $0.1$ & $0.1$ & $10$ & $10^{-4}$ & $0.05$ & \cellcolor{Gray}{$\mathbf{0.15}$} & $775_{\pm 42}$ \\
            $0.1$ & $0.1$ & $10$ & $10^{-4}$ & $0.05$ & \cellcolor{Gray}{$\mathbf{0.25}$} & $1985_{\pm 89}$ \\
            \bottomrule
        \end{tabular}%
    }
\end{table}

\subsection{Fairness-aware Classification}

We next evaluate SPACO on fairness-aware classification, where the goal is to learn an accurate predictor while mitigating bias with respect to a sensitive attribute.
Classical adversarial debiasing~\cite{zhang2018mitigating} trains an auxiliary adversary $\varphi$ to infer sensitive information from the predictor $\theta$, and updates the predictor through a gradient aggregation heuristic to reduce the dependence of predictions on this information.
\cite{chen2025secondorder} recently recast this task as an unconstrained minimax problem.
We augment their minimax formulation with an explicit competency constraint and solve
\begin{equation}\label{eq:fairness-minimax-con}
    \min_{\theta}  \max_{\varphi}\left\{ \mathcal{L}_{\text{pred}}(\theta) - \beta \mathcal{L}_{\text{adv}} (\theta,\varphi) \mid \mathcal{L}_{\text{adv}}(\theta, \varphi) \le \kappa \right\}.
\end{equation}
Here, $\mathcal{L}_{\text{pred}}$ and $\mathcal{L}_{\text{adv}}$ denote the losses of the predictor and the adversary, respectively.
The constraint $\mathcal{L}_{\text{adv}} \le \kappa$ enforces a minimum competency level $\kappa$ on the adversary, preventing the predictor from being over-optimized against a weak opponent.
We test this formulation in two regimes: a controlled convex logistic model and a large-scale nonconvex deep model.
Performance is assessed by predictive accuracy (Acc), demographic parity difference (DPD)~\cite{feldman2015demographic}, and equalized odds difference (EOD)~\cite{hardt2016eqodd}.
Here DPD measures disparity in positive prediction rates, while EOD measures disparity in true-positive and false-positive rates across sensitive groups.

\textbf{Convex regime.}
We follow~\cite{chen2025secondorder} and use the UCI Adult dataset~\cite{changlibsvm2011} to predict annual income with gender as the sensitive attribute.
We solve a convex-concave $\ell_2$-regularized logistic regression game with $\beta=0.5$ and regularization parameters $\lambda=\gamma=10^{-4}$.
We compare SPACO against the canonical first-order baseline, ExtraGradient (EG)~\cite{korpelevich1976extragradient}, and an advanced Newton-type solver LEN~\cite{chen2025secondorder}.
All methods are run for $100$ seconds with stepsizes $0.1$ for both variables.
For SPACO, we set $\kappa=0.65$, $\rho_0=50$, and exponents $(t,s)=(0.01,0.04)$.

\textbf{Nonconvex regime.}
We use the CelebA dataset~\cite{liu2015celeba} to predict \emph{Blond Hair} while mitigating bias with respect to \emph{Gender}.
The model consists of a pretrained ResNet-18 predictor and an MLP adversary, both trained for $20$ epochs with a learning rate of $0.002$.
We compare SPACO against vanilla classification (Vanilla) and Adversarial Debiasing~\cite{zhang2018mitigating} (ADVER), in which $\alpha$ controls the strength of the adversarial-gradient aggregation.
For SPACO, we set $\kappa=0.65$, $\beta=4.0$, $\rho_0=20$, $\sigma_0=10^{-6}$, and $(t,s)=(0.01,0.04)$.
Further experimental details are given in Appendix~\ref{app:experiment-details}.

\begin{table}[t]
    \centering
    \caption{Comparison of fairness-utility trade-offs. Values are means over three independent runs, with subscripts denoting standard deviations. For the nonconvex regime, we report the checkpoint with the smallest DPD subject to the desired prediction level $\mathrm{Acc}\ge 90\%$.}%
    \label{tab:fairness_main}
    \begin{subtable}[t]{0.49\textwidth}
        \centering
        \caption{Convex Regime (Adult)}
        \setlength{\tabcolsep}{0.5pt}
        \begin{tabular}{lccc}
        \toprule
        \textbf{Method} & {Acc} ($\uparrow$) & {DPD} ($\downarrow$) & {EOD} ($\downarrow$) \\
        \midrule
        EG & $84.3_{\pm 0.0}$ & $0.191_{\pm 0.001}$ & $0.282_{\pm 0.008}$ \\
        LEN & $\mathbf{84.9_{\pm 0.0}}$ & $0.184_{\pm 0.000}$ & $0.212_{\pm 0.000}$ \\
        \cellcolor{Gray}\textbf{SPACO} & \cellcolor{Gray}{$84.8_{\pm 0.0}$} & \cellcolor{Gray}{$\mathbf{0.177_{\pm 0.001}}$} & \cellcolor{Gray}{$\mathbf{0.195_{\pm 0.002}}$} \\
        \bottomrule
        \end{tabular}
    \end{subtable}
    \hfill
    \begin{subtable}[t]{0.49\textwidth}
        \centering
        \caption{Nonconvex Regime (CelebA)}
        \setlength{\tabcolsep}{0.5pt}
        \begin{tabular}{lccc}
        \toprule
        \textbf{Method} & {Acc} ($\uparrow$) & {DPD} ($\downarrow$) & {EOD} ($\downarrow$) \\
        \midrule
        Vanilla & $\mathbf{96.0_{\pm 0.0}}$ & $0.188_{\pm 0.003}$ & $0.506_{\pm 0.016}$ \\
        ADVER & $91.3_{\pm 1.5}$ & $0.080_{\pm 0.023}$ & $0.273_{\pm 0.021}$ \\
        \cellcolor{Gray}\textbf{SPACO} & \cellcolor{Gray}{$90.8_{\pm 0.3}$} & \cellcolor{Gray}{$\mathbf{0.057_{\pm 0.010}}$} & \cellcolor{Gray}{$\mathbf{0.076_{\pm 0.020}}$} \\
        \bottomrule
        \end{tabular}
    \end{subtable}
\end{table}

\begin{figure}[tb]
    \centering
    \begin{subfigure}[t]{0.45\linewidth}
        \centering
        {
            \fontsize{9}{6}\selectfont
            \hfill
            \legendItemm[0.4]{fariness-c1}{$0.1$}\hfill
            \legendItemm[0.4]{fariness-c2}{$0.3$}\hfill
            \legendItemm[0.4]{fariness-c3}{$0.5$}\hfill
            \legendItemm[0.4]{fariness-c4}{$0.7$}\hfill
            \legendItemm[0.4]{fariness-c5}{$0.9$}\hfill
        }
        
        \includegraphics[width=.9\linewidth]{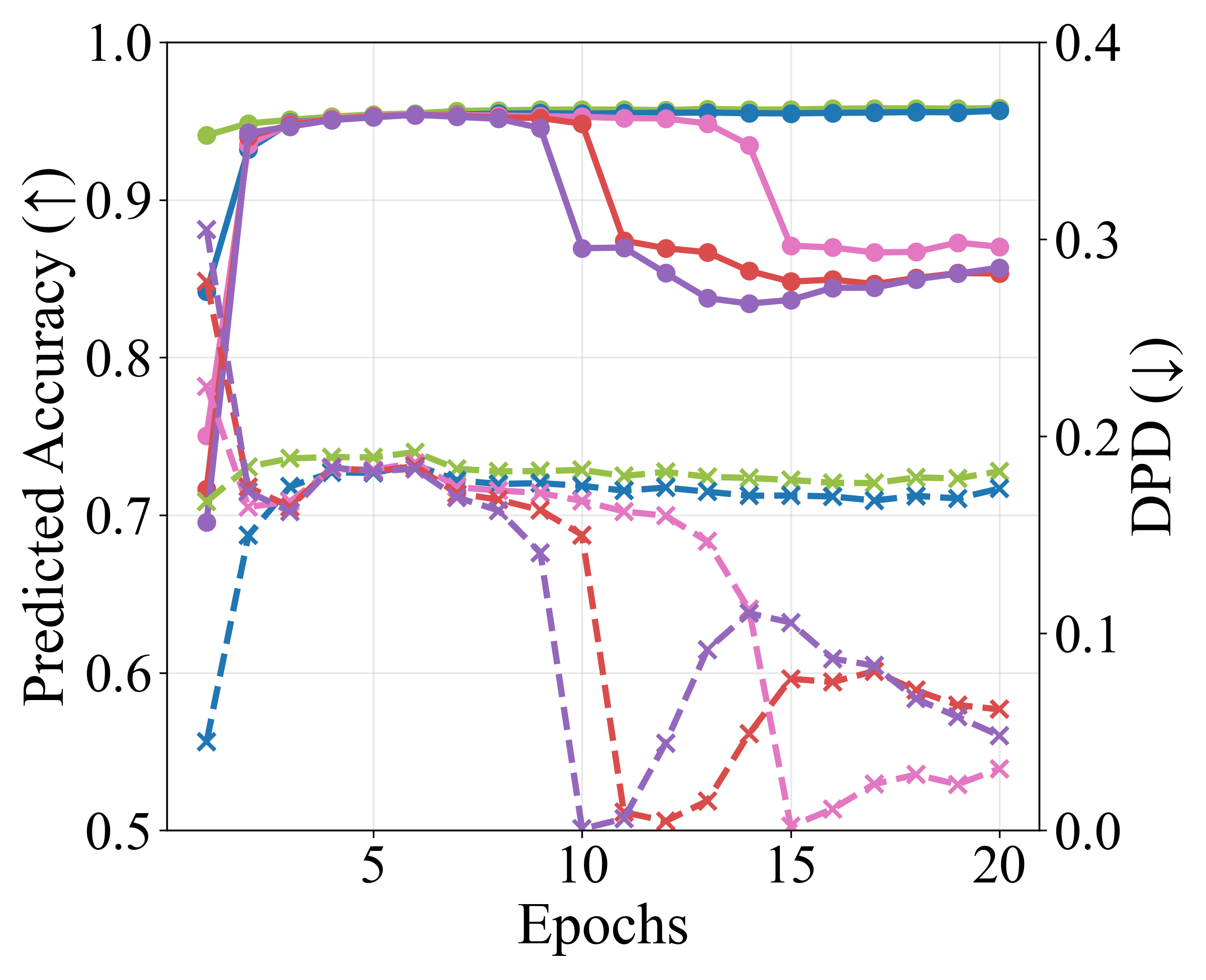}
        \caption{ADVER with different hyperparameter $\alpha$}\label{fig:fairness-deep-adv}
    \end{subfigure}
    \hspace*{1em}
    \begin{subfigure}[t]{0.45\linewidth}
        \centering
        {
            \fontsize{9}{6}\selectfont
            \hspace{1em}\hfill
            \legendItemm[0.4]{fariness-c1}{$2.0$}\hfill
            \legendItemm[0.4]{fariness-c2}{$3.0$}\hfill
            \legendItemm[0.4]{fariness-c3}{$4.0$}\hfill
            \legendItemm[0.4]{fariness-c4}{$5.0$}\hfill\hspace{1em}
        }
        
        \includegraphics[width=.9\linewidth]{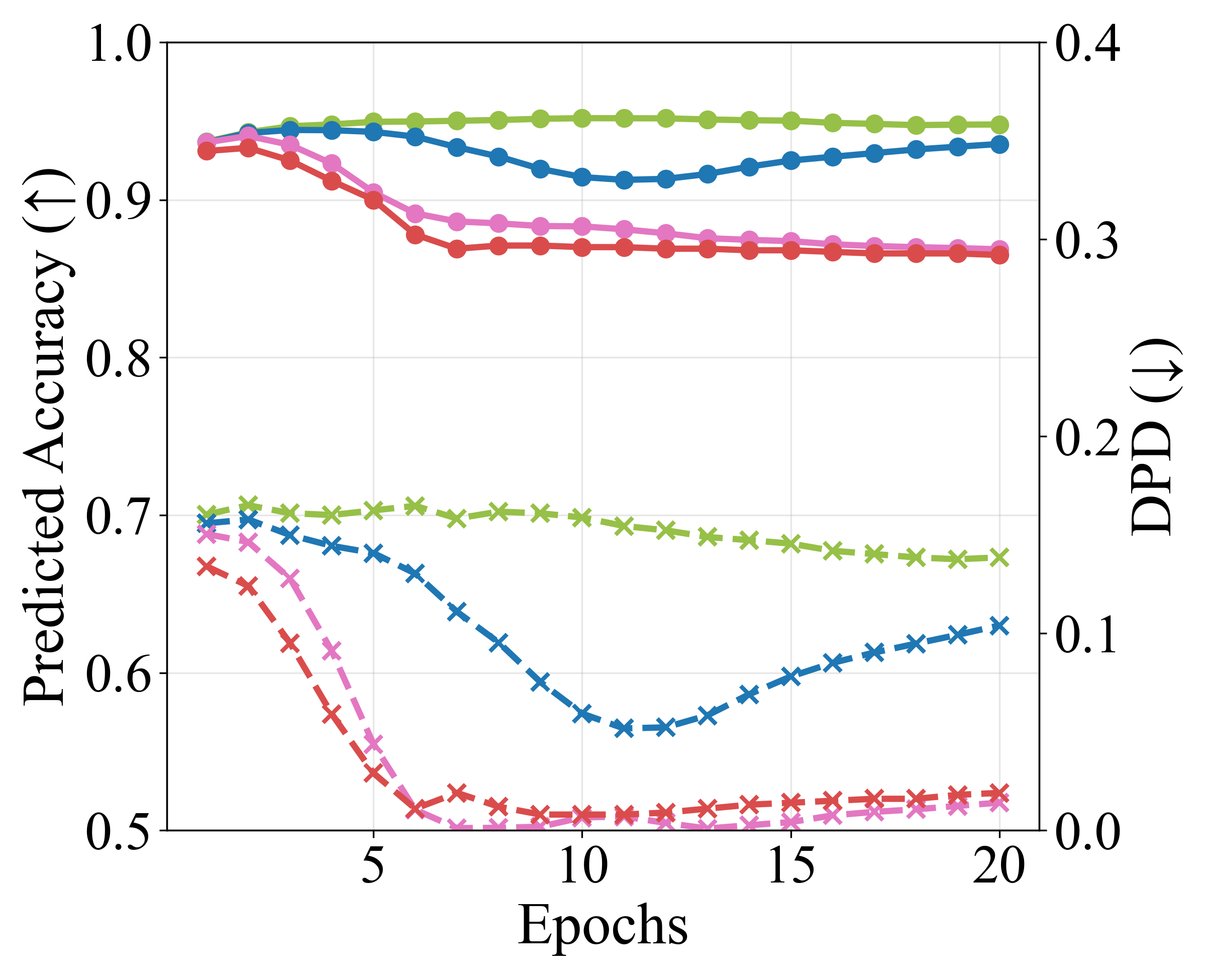}
        \caption{SPACO with different hyperparameter $\beta$}\label{fig:fairness-spaco-beta}
    \end{subfigure}
    \caption{
        Sensitivity plots in the nonconvex regime for deep fairness learning on CelebA. Solid lines ($\bullet$) and dashed lines ($\times$) denote predictive accuracy and DPD, respectively.
    }\label{fig:fairness-deep}
\end{figure}

Table~\ref{tab:fairness_main} reports the quantitative fairness-utility trade-off in the two regimes.
In the controlled convex regime, SPACO attains essentially the same predictive accuracy as LEN while reducing both DPD and EOD, indicating that the constrained minimax formulation improves the fairness side of the trade-off without an apparent loss in utility.
In the large-scale nonconvex experiment, SPACO gives the best fairness metrics among the compared methods while maintaining the desired prediction level.
In addition, Figure~\ref{fig:fairness-deep} shows that ADVER can exhibit abrupt changes in both accuracy and DPD during training, whereas SPACO maintains smoother accuracy--fairness trajectories across a broad range of $\beta$ values.
These trajectories suggest a practical benefit of the constrained minimax formulation: the competency constraint may stabilize the adversarial debiasing process by keeping the adversary informative during training.



\subsection{Generative Adversarial Networks}
We finally evaluate SPACO on constrained generative adversarial training, where the coupled constraint is used to control the interaction between the generator and the discriminator.
Following GAN-C~\cite{chao2021constrained}, we consider
\begin{equation*}
    \begin{aligned}
        \min_{G}\max_{D}\;&
         \mathbb{E}_{x_r,z}\! \left[\log D(x_r)+\log\big(1-D(G(z))\big)\right], \\
        \text{s.t.}\;& \mathbb{E}_{x_r,z}\big[\log D(x_r)-\log D(G(z))\big]^2 \leq \epsilon.
    \end{aligned}
\end{equation*}
The coupled constraint controls the discrepancy between discriminator outputs on real and generated samples and is intended to prevent the discriminator from dominating the generator too early.
The original GAN-C baseline relaxes this hard constraint into a fixed quadratic penalty term in the discriminator objective.
We compare SPACO with the unconstrained GAN baseline and the original GAN-C solver~\cite{chao2021constrained}.

\begin{table}[t]
\centering
    \caption{Quantitative results for GAN training on the CIFAR-10 and AFHQ-v2 datasets. Values are means over three independent runs, with subscripts denoting standard deviations.}%
    \label{tab:gan}
    \begin{tabular}{lcccc}
        \toprule
        \multirow{2}{*}{\textbf{Method}} & \multicolumn{2}{c}{CIFAR-10} & \multicolumn{2}{c}{AFHQ-v2} \\
        \cmidrule(lr){2-3} \cmidrule(lr){4-5}
            & FID ($\downarrow$) & IS ($\uparrow$) & FID ($\downarrow$) & IS ($\uparrow$) \\
        \midrule
        GAN & $35.33_{\pm 2.58}$ & $6.23_{\pm 0.24}$ & $28.50_{\pm 1.52}$ & $6.47_{\pm 0.14}$ \\
        GAN-C & $21.19_{\pm 1.34}$ & $7.39_{\pm 0.07}$ & $26.46_{\pm 0.74}$ & $6.62_{\pm 0.15}$ \\
        \rowcolor{Gray}
        \textbf{SPACO} & $\mathbf{18.78_{\pm 1.41}}$ & $\mathbf{7.73_{\pm 0.15}}$ & $\mathbf{24.44_{\pm 0.95}}$ & $\mathbf{6.89_{\pm 0.09}}$ \\
        \bottomrule
    \end{tabular}
\end{table}

\begin{figure}[t]
    \centering
        \includegraphics[width=.48\linewidth]{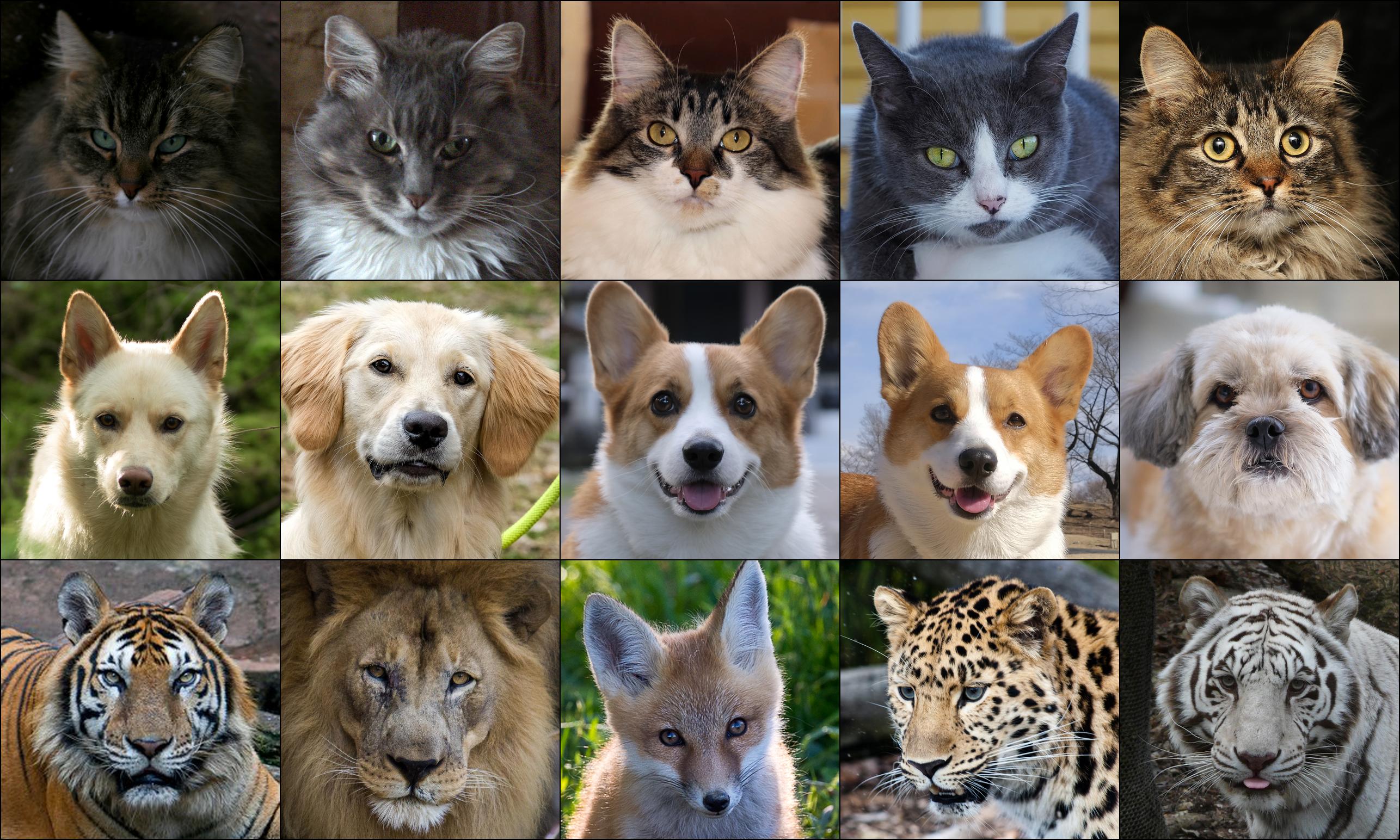}
    \hfill
    \includegraphics[width=.48\linewidth]{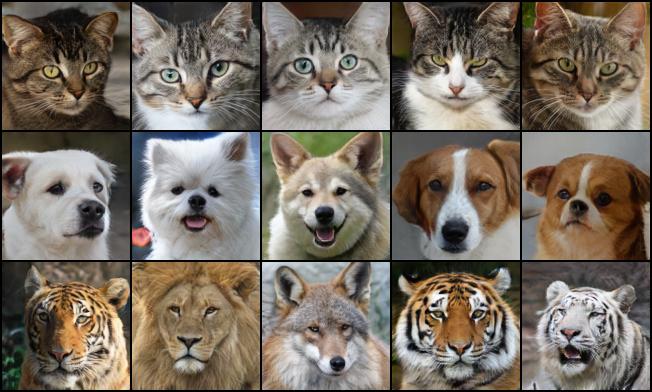}
    \caption{
        Representative AFHQ-v2 samples are shown, with real images on the left and samples generated by SPACO on the right.
    }\label{fig:gan-inference}
\end{figure}

We use CIFAR-10~\cite{krizhevsky2009learning} and AFHQ-v2~\cite{choi2020starganv2}.
All methods employ the same Spectral Normalization GAN architecture~\cite{miyato2018spectral}, the Adam optimizer with learning rate $2\times 10^{-4}$ and momentum $(\beta_1,\beta_2)=(0.0,0.9)$, and $100$ training epochs.
Following GAN-C~\cite{chao2021constrained}, both constrained methods keep learning rates fixed and apply the discrepancy constraint term only when updating discriminator $D$.
This choice keeps the objective for generator $G$ focused on improving image quality and fooling $D$, instead of adding discrepancy reduction to the generator objective.
For the constrained methods, we set $\epsilon=0$ and use penalty $\rho=5$ for GAN-C, while SPACO uses initial penalty $\rho_0=5$ with exponents $(t,s)=(0.09,0.28)$.
Performance is assessed by Fr\'{e}chet Inception Distance (FID)~\cite{heusel2017gans} and Inception Score (IS)~\cite{salimans2016improved}, averaged over three runs.

Table~\ref{tab:gan} reports the quantitative results on both datasets.
The improvement of GAN-C over the unconstrained GAN baseline supports the role of the discriminator-discrepancy constraint in stabilizing training.
This control helps prevent the discriminator from overwhelming the generator while preserving a useful training signal.
SPACO further improves over GAN-C, suggesting that the proposed constrained minimax update can be a practical alternative to the fixed-penalty implementation.
This distinction is consistent with the dynamic penalty treatment in SPACO:\ a fixed penalty may either under-enforce the discriminator-discrepancy control or over-penalize useful discriminator learning, whereas the adaptive constrained formulation can preserve the stabilizing role of the constraint while reducing the bias introduced by a fixed relaxation.
Finally, Figure~\ref{fig:gan-inference} gives a qualitative view on AFHQ-v2 and shows that SPACO generates diverse samples with realistic visual details.

\newpage

\appendix

\section{Additional Experimental Details}\label{app:experiment-details}

\paragraph{Synthetic examples.}%
Because the exact stationarity gap is unavailable under stochastic gradients, we follow~\cite{ghadimi2013stochastic} and tune all methods by minimizing the norm of the same large-batch estimate of the generalized gradient, using a fixed batch of $T=100$ samples.
For each configuration, we run $K=10^4$ iterations and compute the tuning score by averaging this norm every $100$ iterations over the last $1000$ steps.

Based on this protocol, we tune each method on the two stochastic synthetic examples over the following hyperparameter grids.
MGD uses primal stepsizes $\alpha,\beta\in\{0.1,0.01,0.001\}$, dual stepsize $\gamma\in\{1,0.1,0.01\}$, and inner steps $L\in\{1,5,10\}$.
GBAL uses stepsizes $\alpha,\beta\in\{0.1,0.01,0.001\}$ and penalty $\rho\in\{0.1,1,10\}$.
MMPen uses $\alpha_{\mathrm{pen}}\in\{20,100,500\}$ and $\eta\in\{0.5,1,2\}$.
GDA-FP uses penalty $\rho\in\{5,10,20,50\}$ and stepsizes $\alpha,\beta\in\{0.1,0.01,0.001\}$.
SPACO uses stepsizes $\alpha_0,\beta_0\in\{0.1,0.01,0.001\}$, penalty $\rho_0\in\{5,10,20,50\}$,  and regularization $\sigma_0\in\{10^{-3},10^{-4},10^{-5}\}$, with fixed exponents $(t,s)=(0.05,0.2)$.
Table~\ref{tab:hyperparams} lists the best configurations.
For deterministic Example~\ref{exm}, each method uses the hyperparameters selected for the stochastic nonlinear case~\eqref{eq:nonlinear-toy}.

\begin{table}[ht]
    \centering
    \caption{Best hyperparameters on the stochastic synthetic examples under the protocol.}%
    \label{tab:hyperparams}
    {
    \small\setlength{\tabcolsep}{10pt}
    \begin{tabular}{@{}lll@{}}
        \toprule
        \textbf{Algorithm} & \textbf{Nonlinear Constraint Case~\eqref{eq:nonlinear-toy}} & \textbf{Linear Constraint Case~\eqref{eq:linear-toy}} \\
        \midrule
        \textbf{MGD} & $\alpha=0.01,\beta=0.01,\gamma=0.1,L=1$ & $\alpha=0.001,\beta=0.1,\gamma=0.1,L=1$ \\
        \textbf{GBAL} & $\alpha=0.001,\beta=0.1,\rho=1.0$ & $\alpha=0.001,\beta=0.01,\rho=0.1$ \\
        \textbf{MMPen} & $\alpha_{\mathrm{pen}}=20,\eta=1.0$ & $\alpha_{\mathrm{pen}}=100,\eta=1.0$ \\
        \textbf{GDA-FP} & $\rho_0=20,\alpha_0=0.001,\beta_0=0.01$ & $\rho_0=20,\alpha_0=0.001,\beta_0=0.01$ \\
        \textbf{SPACO} & $\rho_0=10,\alpha_0=0.1,\beta_0=0.1,\sigma_0=10^{-4}$ & $\rho_0=20,\alpha_0=0.01,\beta_0=0.1,\sigma_0=10^{-4}$ \\
        \bottomrule
    \end{tabular}
    }
\end{table}

\paragraph{Fairness-aware classification.}%
In the convex Adult experiment, SPACO uses batch size $512$, and LEN uses the recommended lazy Hessian frequency $m=10$.
In the nonconvex CelebA experiment, images are center-cropped to $178\times178$ and resized to $224\times224$ before training, and the batch size is $256$.
For Adversarial Debiasing, the adversarial strength is tuned over $\alpha\in\{0.1,0.3,0.5,0.7,0.9\}$.
For Table~\ref{tab:fairness_main}, the reported nonconvex checkpoint is selected as the one with the smallest DPD subject to $\mathrm{Acc}\ge 90\%$; EOD is then reported as a complementary fairness metric.

\paragraph{Generative adversarial networks.}%
For the GAN experiments, CIFAR-10 images are resized to $64\times64$, and AFHQ-v2 images are resized to $128\times128$.
During training, the batch size is $128$ for CIFAR-10 and $64$ for AFHQ-v2, with latent dimensions $128$ and $256$, respectively.
Both the generator and discriminator use base channel width $64$.
Evaluation uses the exponential moving average of the generator parameters with decay rate $0.999$ when computing FID and IS scores on generated samples.

\section{Proof for Section \ref{sec:preliminaries}}
\subsection{Proof for Lemma \ref{isc and lsc}}
\begin{proof}
    Fix $\bar{x}\in X$. If $\varphi(\bar{x})=-\infty$, then the desired
    lower-semicontinuity inequality holds trivially. We therefore assume that
    $\varphi(\bar{x})>-\infty$. Let $\{x_k\}\subset X$ be any sequence such
    that $x_k\to\bar{x}$.
    By Assumption~\ref{assum1}, $\Gamma(\bar{x})$ is nonempty and compact.
    Since $f$ is continuous, there exists $\bar{y}\in\Gamma(\bar{x})$ such that
    $f(\bar{x},\bar{y})=\varphi(\bar{x})$.
    The inner semicontinuity of $\Gamma$ at $\bar{x}$ implies that there
    exists a sequence $y_k\in\Gamma(x_k)$ such that $y_k\to\bar{y}$. Hence,
    for every $k$,
    \[
        \varphi(x_k)=\max_{y\in\Gamma(x_k)} f(x_k,y)\ge f(x_k,y_k).
    \]
    Taking the lower limit and using the continuity of $f$, we obtain
    \[
        \liminf_{k\to\infty}\varphi(x_k)
        \ge \lim_{k\to\infty} f(x_k,y_k)
        = f(\bar{x},\bar{y})
        = \varphi(\bar{x}).
    \]
    Since the sequence $\{x_k\}$ was arbitrary, $\varphi$ is lower
    semicontinuous on $X$.
\end{proof}

\subsection{Proof for Lemma \ref{slaterisc}}
\begin{proof}
Fix any $\bar{y}\in \Gamma(\bar{x})$, and let $\{x_k\}\subset X$ be any
sequence such that $x_k\to \bar{x}$. We construct a sequence
$y_k\in \Gamma(x_k)$ with $y_k\to \bar{y}$.

By the Slater condition, there exists $\tilde{y}\in\operatorname{ri}(Y)$ such
that $c(\bar{x},\tilde{y})<0$ componentwise. Since there are finitely many
constraints, we may choose $\eta>0$ such that
$c_i(\bar{x},\tilde{y})\le -2\eta$ for all $i=1,\ldots,p$. By continuity,
$c_i(x_k,\tilde{y})\le -\eta$ for all $i=1,\ldots,p$ and all sufficiently
large $k$.

Since $\bar{y}\in \Gamma(\bar{x})$, we have
$c_i(\bar{x},\bar{y})\le 0$ for all $i$. Define
$\varepsilon_k:=\max_{1\le i\le p}[c_i(x_k,\bar{y})]_+$. The continuity of
each $c_i(\cdot,\bar{y})$ and the convergence $x_k\to\bar{x}$ imply that
$\varepsilon_k\to0$. For all sufficiently large $k$, set
$t_k:=\varepsilon_k/(\varepsilon_k+\eta)\in[0,1)$ and
$y_k:=(1-t_k)\bar{y}+t_k\tilde{y}$.
Because $Y$ is convex and $\bar{y},\tilde{y}\in Y$, we have $y_k\in Y$.
Moreover, $t_k\to0$, and hence $y_k\to\bar{y}$.

It remains to verify feasibility. For each $i=1,\ldots,p$, the convexity of
$c_i(x_k,\cdot)$ on $Y$ yields
\[
\begin{aligned}
c_i(x_k,y_k)
&\le (1-t_k)c_i(x_k,\bar{y})+t_k c_i(x_k,\tilde{y}) \le (1-t_k)\varepsilon_k-t_k\eta
 =0 .
\end{aligned}
\]
Thus $y_k\in\Gamma(x_k)$ for all sufficiently large $k$. For the finitely
many remaining indices, choose any $y_k\in\Gamma(x_k)$; changing finitely many terms does not affect the
convergence $y_k\to\bar{y}$. Hence, for the arbitrary sequence
$x_k\to\bar{x}$ and arbitrary $\bar{y}\in\Gamma(\bar{x})$, we have constructed
$y_k\in\Gamma(x_k)$ with $y_k\to\bar{y}$. This proves the claimed inner
semicontinuity.
\end{proof}

\subsection{Proof for Proposition \ref{thm:slater-gplcq}}
\begin{proof}
Fix any $\beta>0$ and write $p(x,y)=\frac12\|[c(x,y)]_+\|^2$. Suppose, to
the contrary, that the projected residual bound in Definition~\ref{def:gplcq}
fails locally at $(\bar x,\bar y)$ for this choice of $\beta$. Then, for each
$k$, there exists
$(x_k,y_k)\in (B_{1/k}(\bar x)\cap X)\times(B_{1/k}(\bar y)\cap Y)$ such that,
with
\[
 r_k:=\left\|
\begin{bmatrix}x_k\\ y_k\end{bmatrix}
-\P_{X\times Y}\left(
\begin{bmatrix}x_k\\ y_k\end{bmatrix}
+\beta\begin{bmatrix}\nabla_xp(x_k,y_k)\\-\nabla_yp(x_k,y_k)\end{bmatrix} 
 \right)\right\| \rightarrow 0,
\]
we have $\sqrt{p(x_k,y_k)}>(k/\beta)r_k$. Hence
$r_k/\sqrt{p(x_k,y_k)}\to0$. 

Since
$\sqrt{p(x_k,y_k)}=\|[c(x_k,y_k)]_+\|/\sqrt{2}$, it follows that, with
$s_k:=\|[c(x_k,y_k)]_+\|$, we have $s_k>0$ and $r_k/s_k\to0$. Define
$\lambda_k:=[c(x_k,y_k)]_+/s_k$. Passing to a subsequence if necessary,
$\lambda_k\to\lambda$, where $\lambda\in\mathbb R_+^p$ and $\|\lambda\|=1$.
Moreover, if $c_i(\bar x,\bar y)<0$, then $(\lambda_k)_i=0$ for all
sufficiently large $k$, and hence $\lambda_i=0$.

Let $(x_k^+,y_k^+):=\P_{X\times Y}\left(
\begin{bmatrix}x_k\\ y_k\end{bmatrix}
+\beta\begin{bmatrix}\nabla_xp(x_k,y_k)\\-\nabla_yp(x_k,y_k)\end{bmatrix}
\right)$ denote the projection point in the definition of $r_k$. Since
$(\bar x,\bar y)$ is feasible, $s_k\to0$, and therefore $r_k\to0$. Thus
$(x_k^+,y_k^+)\to (\bar{x},\bar{y})$. The projection optimality conditions give
\[
 x_k+\beta\nabla_xp(x_k,y_k)-x_k^+\in\mathcal N_X(x_k^+),\quad
 y_k-\beta\nabla_yp(x_k,y_k)-y_k^+\in\mathcal N_Y(y_k^+).
\]
Using
$    \nabla p(x_k,y_k)
    =s_k\sum_{i=1}^p(\lambda_k)_i\nabla c_i(x_k,y_k),$
dividing the two inclusions by $\beta s_k$, and using $r_k/s_k\to0$, we may
pass to the limit. By the outer semicontinuity of normal cones to closed
convex sets,
 $\sum_{i\in\mathcal A(\bar x,\bar y)}\lambda_i\nabla_xc_i(\bar x,\bar y)\in\mathcal N_X(\bar x),\;
 -\sum_{i\in\mathcal A(\bar x,\bar y)}\lambda_i\nabla_yc_i(\bar x,\bar y)\in\mathcal N_Y(\bar y),$
where $\mathcal A(\bar x,\bar y)=\{i:c_i(\bar x,\bar y)=0\}$.

On the other hand, let $\tilde y\in\operatorname{ri}(Y)$ be the Slater point,
so that $c_i(\bar x,\tilde y)<0$ for all $i$. By the convexity of
$c_i(\bar x,\cdot)$, for every $i\in\mathcal A(\bar x,\bar y)$,
$ \langle \nabla_yc_i(\bar x,\bar y),\tilde y-\bar y\rangle
 \le c_i(\bar x,\tilde y)-c_i(\bar x,\bar y)<0.$

Since $\lambda\ge0$, $\|\lambda\|=1$, and $\lambda_i=0$ for every inactive
constraint, this implies
\[
\left\langle -\sum_{i\in\mathcal A(\bar x,\bar y)}\lambda_i\nabla_yc_i(\bar x,\bar y),\tilde y-\bar y\right\rangle>0,
\]
which contradicts the normal-cone inequality because $\tilde y\in Y$. Hence
the residual bound in Definition~\ref{def:gplcq} must hold locally
for the fixed $\beta$. Therefore the GP\L{}CQ holds at $(\bar x,\bar y)$.
\end{proof}

\section{Proof for Section \ref{approximaiton}}\label{proof for approximation}
In this section, we provide detailed proofs of the results on the smooth
approximation developed in Section~\ref{approximaiton}. To simplify notation,
throughout this section we write $
\LL_k(x,y):=\LL_{\rho_k,\sigma_k}(x,y)$, $\varphi_k(x):=\varphi_{\rho_k,\sigma_k}(x)$, 
$y_k^*(x):=y_{\rho_k,\sigma_k}^*(x)$.

\subsection{Proof for Lemma \ref{limsup} }
Before presenting the proof for Lemma \ref{limsup}, we establish the feasibility of the limiting point by the following lemma.

\begin{lemma}\label{feasibility}
Let $\rho_k\to\infty$ and $\sigma_k\to 0$ as $k\to\infty$. Then, for any
sequence $\{x_k\}\subset X$ such that $x_k\to\bar{x}\in X$, we have
\[
    \lim_{k\to\infty}\left\|[c(x_k,y_k^*(x_k))]_+\right\|=0.
\]
Consequently, any accumulation point $(\bar x,\bar y)$ of
$\{(x_k,y_k^*(x_k))\}$ is feasible, i.e., $\bar y\in\Gamma(\bar x)$.
\end{lemma}
\begin{proof}
By Assumption~\ref{assum1}, choose any $\bar y\in\Gamma(\bar x)$. Since
$y_k^*(x_k)$ maximizes $\LL_k(x_k,\cdot)$ over $Y$,
\[
\begin{aligned}
& f(x_k,y_k^*(x_k))
  -\frac{\rho_k}{2}\left\|[c(x_k,y_k^*(x_k))]_+\right\|^2
  -\frac{\sigma_k}{2}\|y_k^*(x_k)\|^2   \\
\ge\;&
  f(x_k,\bar y)
  -\frac{\rho_k}{2}\left\|[c(x_k,\bar y)]_+\right\|^2
  -\frac{\sigma_k}{2}\|\bar y\|^2 .
\end{aligned}
\]
Rearranging gives
\[
\begin{aligned}
\left\|[c(x_k,y_k^*(x_k))]_+\right\|^2
&\le \left\|[c(x_k,\bar y)]_+\right\|^2
 +\frac{2}{\rho_k}\big(f(x_k,y_k^*(x_k))-f(x_k,\bar y)\big) 
 +\frac{\sigma_k}{\rho_k}\big(\|\bar y\|^2-\|y_k^*(x_k)\|^2\big).
\end{aligned}
\]
The sequence $\{x_k\}$ is convergent and $Y$ is compact, so the terms
$f(x_k,y_k^*(x_k))$, $f(x_k,\bar y)$, $\|y_k^*(x_k)\|$, and $\|\bar y\|$ are
bounded. Moreover, $\|[c(x_k,\bar y)]_+\|\to\|[c(\bar x,\bar y)]_+\|=0$,
$1/\rho_k\to0$, and $\sigma_k/\rho_k\to0$. Taking the upper limit in the
preceding inequality yields
\[
    \limsup_{k\to\infty}
    \left\|[c(x_k,y_k^*(x_k))]_+\right\|^2\le 0.
\]
Since the left-hand side is nonnegative, the claimed convergence follows. The
feasibility of any accumulation point follows immediately from the continuity
of $c$.
\end{proof}

\begin{proof}[Proof for Lemma \ref{limsup}]
Fix $x\in X$.
We prove $\limsup_{k\to\infty}\varphi_{k}(x)\le \varphi(x)$ by contradiction.
Suppose that there exists $\delta>0$ such that 
\[
\limsup_{k\to\infty} \varphi_{k}(x)> \varphi(x)+\delta.
\]
Then there exists a subsequence $\{k_i\}$ such that
$\varphi_{k_i}(x)>\varphi(x)+\delta$ for all $i$. Since
$y_{k_i}^*(x)\in Y$ and $Y$ is compact, by passing to a further subsequence if
necessary, we may assume that $y_{k_i}^*(x)\to\hat{y}$ for some
$\hat{y}\in Y$.
Consequently,
\begin{align}\label{limsup:eq1}
\limsup_{i\to\infty}\varphi_{k_i}(x)
&=
\limsup_{i\to\infty}
\Bigl(
f(x,y^*_{k_i}(x))
-\frac{\rho_{k_i}}{2}\left\|[c(x,y^*_{k_i}(x))]_+\right\|^2
-\frac{\sigma_{k_i}}{2}\|y^*_{k_i}(x)\|^2
\Bigr) \notag\\
&\le
\limsup_{i\to\infty} f(x,y^*_{k_i}(x))
= f(x,\hat{y}),
\end{align}
where the inequality follows from the nonnegativity of the penalty terms.
On the other hand, it follows from Lemma~\ref{feasibility} that
$c(x,\hat{y})\le 0$.
Recalling that
$\varphi(x)=\max_{y\in Y,\; c(x,y)\le 0} f(x,y),$
we have $\varphi(x)\ge f(x,\hat{y})$. This contradicts
$\varphi_{k_i}(x)>\varphi(x)+\delta$ for all $i$ and \eqref{limsup:eq1}.
\end{proof}

\subsection{Proof for Lemma \ref{liminf}}

\begin{proof}
For each $k$, let
$y^*(x_k)\in \underset{y\in Y,\;c(x_k,y)\le 0}{\arg\max} f(x_k,y),$
which is well defined because $\Gamma(x_k)$ is nonempty and compact and $f$ is
continuous. By the definition of $\varphi_k(x_k)$ and the optimality of
$y_k^*(x_k)$, we have
\begin{align*}
\varphi_k(x_k)
&=
f(x_k,y_k^*(x_k))
-\frac{\rho_k}{2}\left\|[c(x_k,y_k^*(x_k))]_+\right\|^2
-\frac{\sigma_k}{2}\|y_k^*(x_k)\|^2\\
&\ge
f(x_k,y^*(x_k))
-\frac{\rho_k}{2}\left\|[c(x_k,y^*(x_k))]_+\right\|^2
-\frac{\sigma_k}{2}\|y^*(x_k)\|^2.
\end{align*}
Since $c(x_k,y^*(x_k))\le 0$, the penalty term vanishes, and thus
\begin{align}\label{lemmaliminf:eq1}
    \varphi_k(x_k)
\ge
f(x_k,y^*(x_k))
-\frac{\sigma_k}{2}\|y^*(x_k)\|^2
=
\varphi(x_k)-\frac{\sigma_k}{2}\|y^*(x_k)\|^2.
\end{align}
Since $y^*(x_k)\in Y$ and $Y$ is compact, the sequence $\{\|y^*(x_k)\|\}$ is
bounded. As $\sigma_k\to0$, the regularization term
$\frac{\sigma_k}{2}\|y^*(x_k)\|^2$ converges to zero. Taking the lower limit in
\eqref{lemmaliminf:eq1} yields
$\liminf_{k\to\infty}\varphi_k(x_k)
\ge
\liminf_{k\to\infty}\varphi(x_k).$
Since $\varphi$ is lower semi-continuous,
it follows that
\[
\liminf_{k\to\infty}\varphi(x_k)\ge \varphi(\bar x),
\]
which completes the proof.
\end{proof}

\subsection{Proof for Proposition \ref{ystarconvergence}}
\begin{proof}
We first note that the feasible-set mapping $\Gamma$ is outer semicontinuous.
Indeed, if $(x_j,y_j)\to(\tilde x,\tilde y)$ and $y_j\in\Gamma(x_j)$, then
$[c(x_j,y_j)]_+=0$ for all $j$. By the continuity of $c$, we obtain
$[c(\tilde x,\tilde y)]_+=0$, and hence $\tilde y\in\Gamma(\tilde x)$.
Next, we show that $\varphi$ is upper semicontinuous at $\bar x$. Choose a
subsequence $\{x_{k_i}\}$ such that $ \lim_{i\to\infty}\varphi(x_{k_i})
    =\limsup_{k\to\infty}\varphi(x_k)$.
For each $i$, let $\hat y_i\in\arg\max_{y\in\Gamma(x_{k_i})} f(x_{k_i},y)$,
so that $\varphi(x_{k_i})=f(x_{k_i},\hat y_i)$. Since $Y$ is compact, after
passing to a further subsequence if necessary, we may assume that
$\hat y_i\to\hat y$ for some
$\hat y\in Y$. By the outer semicontinuity of $\Gamma$, we have
$\hat y\in\Gamma(\bar x)$. Therefore, by the continuity of $f$,
\[
\limsup_{k\to\infty}\varphi(x_k)
=\lim_{i\to\infty} f(x_{k_i},\hat y_i)
=f(\bar x,\hat y)
\le \varphi(\bar x).
\]
Together with the assumed lower semicontinuity of $\varphi$, this implies
$\varphi(x_k)\to\varphi(\bar x)$.

Now let $\bar y$ be any accumulation point of $\{y_k^*(x_k)\}$. Passing to a
subsequence if necessary, assume that $y_k^*(x_k)\to\bar y$. Lemma
\ref{feasibility} gives $\bar y\in\Gamma(\bar x)$. For each $k$, choose
$y^*(x_k)\in\arg\max_{y\in\Gamma(x_k)} f(x_k,y)$. Since $y_k^*(x_k)$ maximizes
$\LL_k(x_k,\cdot)$ over $Y$ and $y^*(x_k)$ is feasible for the original inner
problem, we have
\[
f(x_k,y_k^*(x_k))
\ge \varphi_k(x_k)  \ge f(x_k,y^*(x_k))
   -\frac{\sigma_k}{2}\|y^*(x_k)\|^2
 = \varphi(x_k)-\frac{\sigma_k}{2}\|y^*(x_k)\|^2 .
\]
Because $Y$ is compact and $\sigma_k\to0$, the last term converges to zero.
Taking limits along the subsequence yields
    $f(\bar x,\bar y)
    =\lim_{k\to\infty}f(x_k,y_k^*(x_k))
    \ge \lim_{k\to\infty}\varphi(x_k)
    =\varphi(\bar x).$
Since $\bar y\in\Gamma(\bar x)$, we also have
$f(\bar x,\bar y)\le\varphi(\bar x)$. Thus
$\bar y\in\arg\max_{y\in\Gamma(\bar x)} f(\bar x,y)$.
\end{proof}

\subsection{Proof for Theorem \ref{thm:local-to-type1}}
\begin{proof}
Fix any $r\in(0,\delta)$ and set
$X_r:=X\cap B_r(x^*)$. Then $X_r$ is compact and convex, and
$x^*$ is a global minimizer of $\varphi$ on $X_r$. For each $k$, define
$m_k:=\min_{x\in X_r}\varphi_k(x),
\;
\hat x_k\in\arg\min_{x\in X_r}\varphi_k(x).$
The minimizer $\hat x_k$ exists because $\varphi_k$ is continuous and
$X_r$ is compact.

Pass to a subsequence, relabeled by $k$, such that
$m_k\to\liminf_{\ell\to\infty}m_\ell$. By compactness of $X_r$, after
passing to a further subsequence if necessary, we may assume that
$\hat x_k\to\hat x\in X_r$. Applying Theorem~\ref{epiconvergence} to the
restricted problem on $X_r$, we obtain
$\hat x\in\arg\min_{x\in X_r}\varphi(x)$. Since $x^*$ is also a minimizer
of $\varphi$ on $X_r$, $\varphi(\hat x)=\varphi(x^*)$. Lemma~\ref{liminf}
then gives
\[
\liminf_{k\to\infty} m_k
=\liminf_{k\to\infty}\varphi_k(\hat x_k)
\ge \varphi(\hat x)
= \varphi(x^*).
\]
On the other hand, $m_k\le \varphi_k(x^*)$, and Lemma~\ref{limsup} gives
$
\limsup_{k\to\infty}\varphi_k(x^*)\le \varphi(x^*).
$
Hence, along this subsequence,
$
    m_k\to\varphi(x^*),
$
and $  \varphi_k(x^*)-m_k\to0. $
Set $    \delta_k:=\sqrt{\varphi_k(x^*)-m_k+k^{-2}}.$
Then $\delta_k \to 0$ and $x^*$ is a $\delta_k^2$-minimizer of $\varphi_k$ over $X_r$, that is,
\[
    \varphi_k(x^*)\le \inf_{x\in X_r}\varphi_k(x)+\delta_k^2.
\]
By Ekeland's variational principle
\cite[Proposition~1.43]{rockafellar2009variational}, there exists
$x_k\in X_r$ such that $\|x_k-x^*\|\le \delta_k$ and $x_k$ minimizes
    $x\mapsto \varphi_k(x)+\delta_k\|x-x_k\|$
over $X_r$. In particular, $x_k\to x^*$.
Since $\varphi_k$ is differentiable, Fermat's rule yields
\[
0\in \nabla\varphi_k(x_k)+\delta_k\,\partial\|x-x_k\|_{x=x_k}+\mathcal N_{X_r}(x_k).
\]
Since $\partial\|x-x_k\|_{x=x_k}\subseteq B_1(0)$, there exists
$u_k\in\IR^n$ such that
\[
u_k\in \nabla\varphi_k(x_k)+\mathcal N_{X_r}(x_k),
\qquad
\|u_k\|\le \delta_k\to0.
\]
Since $x_k\to x^*$ and $r>0$, we have $x_k\in\operatorname{int}
B_r(x^*)$ for all sufficiently large $k$. Hence the ball constraint is
inactive, and
$\mathcal N_{X_r}(x_k)=\mathcal N_X(x_k),
\;\text{for all sufficiently large }k.$
After discarding finitely many indices if necessary, we therefore have
\[
u_k\in \nabla\varphi_k(x_k)+\mathcal N_X(x_k),
\qquad
u_k\to0,
\qquad
x_k\to x^*.
\]

By compactness of $Y$, after passing to a subsequence if necessary, assume that
$y_k^*(x_k)\to y^*$. Since $x_k\to x^*$, Proposition~\ref{ystarconvergence}
yields
$y^*\in \arg\max_{y\in Y}\{f(x^*,y)\mid c(x^*,y)\le 0\}.$
By assumption, GP\L CQ holds at $(x^*,y^*)$. Applying
Theorem~\ref{thm:approx-to-type1}, we conclude that $(x^*,y^*)$ is a Type-I
enhanced KKT point.
\end{proof}
\subsection{Proof for Theorem \ref{thm:strict-to-type2}}
\begin{proof}
Fix any $r\in(0,\delta)$ and set $X_r:=X\cap B_r(x^*)$. By strict local
minimality,$\arg\min_{x\in X_r}\varphi(x)=\{x^*\}.$
For each $k$, choose  $x_k\in\arg\min_{x\in X_r}\varphi_{\rho_k,\sigma_k}(x).$
Applying Theorem~\ref{epiconvergence} to the restricted problem on $X_r$, every
accumulation point of $\{x_k\}$ belongs to
$\arg\min_{x\in X_r}\varphi(x)=\{x^*\}$. Since $X_r$ is compact, this implies
$x_k\to x^*$.

For all sufficiently large $k$, we have $\|x_k-x^*\|<r/2$. Therefore
$X\cap B_{r/2}(x_k)\subset X_r,$
so $x_k$ is a local minimizer of $\varphi_{\rho_k,\sigma_k}$ on $X$. Since
$\varphi_{\rho_k,\sigma_k}$ is differentiable, Fermat's rule gives
\[
0\in \nabla \varphi_{\rho_k,\sigma_k}(x_k)+\mathcal N_X(x_k).
\]

By compactness of $Y$, the sequence
$y^*_{\rho_k,\sigma_k}(x_k)$
has an accumulation point. Passing to a subsequence if necessary, assume
$y^*_{\rho_k,\sigma_k}(x_k)\to y^*$. Since $x_k\to x^*$,
Proposition~\ref{ystarconvergence} gives
\[
y^*\in \arg\max_{y\in Y}\{f(x^*,y)\mid c(x^*,y)\le 0\}.
\]
By assumption, GP\L CQ holds at $(x^*,y^*)$. Applying
Theorem~\ref{thm:approx-to-type1} with $u_k\equiv0$, we have
$\rho_k\|u_k\|=0$ for all $k$, and therefore $(x^*,y^*)$ is a Type-II enhanced
KKT point.
\end{proof}

\subsection{Proof for Proposition \ref{prop:spurious-kkt}}

\begin{proof}
Let
$c(x,y):=\mathbf e^\top y-\|x\|^2$ 
and denote the objective function in \eqref{example} by \(f(x,y)\).
Since \((\mathbf 0,\mathbf 0)\in\operatorname{int}(X)\times\operatorname{int}(Y)\),
the normal cones vanish at this point. Moreover, \(c(\mathbf 0,\mathbf 0)=0\),
\[
\nabla_y f(\mathbf 0,\mathbf 0)=\mathbf e,\qquad
\nabla_y c(\mathbf 0,\mathbf 0)=\mathbf e,
\]
and hence the \(y\)-stationarity condition in \eqref{KKT} gives
\((1-\lambda)\mathbf e=\mathbf 0\). Thus the KKT multiplier is uniquely
determined by \(\bar\lambda=1\). With this multiplier, the \(x\)-stationarity
condition also holds because
\[
\nabla_x f(\mathbf 0,\mathbf 0)-\bar\lambda\nabla_x c(\mathbf 0,\mathbf 0)
=\mathbf 0.
\]
Thus \((\mathbf 0,\mathbf 0)\) is a standard KKT point.

Next, we show that \(x=\mathbf 0\) is not a local minimizer of the value
function. For \(x=a\mathbf e\) with \(a\) sufficiently close to zero, \(a<0\),
the inner maximizer is \(y=a^2\mathbf e\), which satisfies
\(\mathbf e^\top y-\|x\|^2=0\). Indeed, the corresponding inner multiplier is
\(1+a/2-a^2>0\) for all sufficiently small \(|a|\), so the KKT conditions of
the strongly concave inner problem are satisfied. A direct substitution gives
$
\varphi(a\mathbf e)=f(a\mathbf e,a^2\mathbf e)=a^3$, and $
\varphi(\mathbf 0)=0.$
Therefore \(\varphi(a\mathbf e)<\varphi(\mathbf 0)\) for all small \(a<0\),
so \(x=\mathbf 0\) is not a local minimizer.

It remains to prove the enhanced KKT statements. Choose any sequences
\(\rho_k\to\infty\) and \(\sigma_k\to0\), and set
\(x_k\equiv\mathbf 0\).
Since the box constraint \(Y\) is inactive at the maximizer for all sufficiently
large \(k\), the unique maximizer of
\(\psi_{\rho_k,\sigma_k}(\mathbf 0,\cdot)\) over \(Y\) is
\[
y_k^*:=y^*_{\rho_k,\sigma_k}(\mathbf 0)
=
\frac{1}{1+\sigma_k+2\rho_k}\mathbf e.
\]
Hence \(y_k^*\to \mathbf 0\), and
$c(\mathbf 0,y_k^*)
=
\frac{2}{1+\sigma_k+2\rho_k}>0.$
Moreover, since \(\mathbf 0\in\operatorname{int}(X)\),
\[
u_k:=\nabla\varphi_{\rho_k,\sigma_k}(\mathbf 0)
=\nabla_x\psi_{\rho_k,\sigma_k}(\mathbf 0,y_k^*)
=\frac{1}{2(1+\sigma_k+2\rho_k)}\mathbf e
\to \mathbf 0.
\]
Since \(\bar\lambda=1>0\) and \(c(\mathbf 0,y_k^*)>0\) for all \(k\), the
strict-positivity requirement is also satisfied. Hence, together with
\(u_k\to\mathbf 0\), all requirements in the definition of a Type-I enhanced
KKT point are verified.

We finally show that \((\mathbf 0,\mathbf 0)\) is not a Type-II enhanced KKT
point. Suppose, to the contrary, that it is Type-II. Since the KKT multiplier
at \((\mathbf 0,\mathbf 0)\) is uniquely given by \(\bar\lambda=1\), there exist
sequences
\[
x_k\to \mathbf 0,\qquad
y_k^*:=y^*_{\rho_k,\sigma_k}(x_k)\to \mathbf 0,\qquad
\rho_k\to\infty,\qquad
\sigma_k\to0,
\]
and
\[
u_k\in \nabla\varphi_{\rho_k,\sigma_k}(x_k)+\mathcal N_X(x_k),
\qquad
\rho_k\|u_k\|\to0,
\]
such that $t_k:=c(x_k,y_k^*)>0$ for all sufficiently large \(k\). Since \(x_k\to\mathbf 0\) and
\(y_k^*\to\mathbf 0\), both \(X\) and \(Y\) are inactive for all sufficiently
large \(k\). Thus \(\mathcal N_X(x_k)=\{0\}\), and
$u_k=\nabla_x\psi_{\rho_k,\sigma_k}(x_k,y_k^*).$
Set
$A_k:=1+\sigma_k,\; B_k:=A_k+2\rho_k,\;
s_k:=\mathbf e^\top x_k,\; q_k:=\|x_k\|^2.$
The first-order condition for the inner maximizer gives
$A_k y_k^*
=
(1-\rho_k t_k)\mathbf e+\frac{x_k}{2}.$

Taking the inner product with \(\mathbf e\) and using
\(t_k=\mathbf e^\top y_k^*-q_k\), we obtain
$t_k=\frac{2+\frac12 s_k-A_k q_k}{B_k}.$
On the other hand,
$u_k
=
(q_k+2+2\rho_k t_k)x_k
-2s_k\mathbf e+\frac12 y_k^*.$
Taking the inner product with \(\mathbf e\), and using
\(\mathbf e^\top y_k^*=t_k+q_k\), yields
$\mathbf e^\top u_k
=
(q_k-2+2\rho_k t_k)s_k+\frac12(t_k+q_k).$
Substituting the expression of \(t_k\) into this identity gives
\[
\rho_k\,\mathbf e^\top u_k
=
\frac{\rho_k}{A_kB_k}
\left[
A_k
+
C_k s_k
+
\rho_k A_k s_k^2
+
\rho_k A_k q_k(1-2\sigma_k s_k)
\right],
\]
where
$C_k:=A_k(A_k q_k-2A_k+\tfrac14)\to -\frac74 .$

Since \(x_k\to\mathbf 0\) and \(\sigma_k\to0\), we have \(s_k\to0\),
\(q_k\to0\), \(A_k\to1\), \(C_k s_k\to0\), and
\(\rho_k/(A_kB_k)\to1/2\). Moreover, \(1-2\sigma_k s_k>0\) for all
sufficiently large \(k\). Hence 
\[
    \rho_k\,\mathbf e^\top u_k
    \ge
    \frac{\rho_k}{A_kB_k}\bigl(A_k+C_k s_k\bigr)
    \ge \frac14
\]
for all sufficiently large \(k\). On the other hand, the Type-II condition
\(\rho_k\|u_k\|\to0\) implies
\[
    |\rho_k\,\mathbf e^\top u_k|
    \le \sqrt{2}\,\rho_k\|u_k\|
    \to0,
\]
which is a contradiction. Thus \((\mathbf 0,\mathbf 0)\) is not a Type-II
enhanced KKT point.
Combining the above statements, Type-II enhanced KKT excludes this spurious
standard KKT point.
\end{proof}

\subsection{Proof for Theorem \ref{thm:approx-to-type1}}
\begin{proof}
Passing to the subsequence that realizes the accumulation point \(\bar y\), and
relabeling the subsequence if necessary, we may assume $
y_k^*(x_k)\to\bar y$.
Recall that
$p(x,y):=\frac12\|[c(x,y)]_+\|^2.$
By Lemma~\ref{feasibility}, \(\bar y\in\Gamma(\bar x)\). Define the penalty
multipliers
$\lambda_k:=\rho_k[c(x_k,y_k^*(x_k))]_+\in\IR^p_+.$
We first prove that \(\{\lambda_k\}\) is bounded. Let
\[
z_k:=\begin{bmatrix}x_k\\ y_k^*(x_k)\end{bmatrix},\qquad
d_k:=
\begin{bmatrix}
\nabla_x p(x_k,y_k^*(x_k))\\
-\nabla_y p(x_k,y_k^*(x_k))
\end{bmatrix},
\qquad
h_k:=
\begin{bmatrix}
-\nabla_x f(x_k,y_k^*(x_k))\\
\nabla_y f(x_k,y_k^*(x_k))-\sigma_k y_k^*(x_k)
\end{bmatrix}.
\]
Since \(x_k\to\bar x\), \(Y\) is compact, and \(\sigma_k\to0\), there is a
constant \(C>0\) such that \(\|h_k\|\le C\) for all sufficiently large \(k\).
Moreover,
\[
\begin{bmatrix}
\nabla_x\LL_k(x_k,y_k^*(x_k))\\
-\nabla_y\LL_k(x_k,y_k^*(x_k))
\end{bmatrix}
=-\rho_k d_k-h_k.
\]
Because GP\L{}CQ holds at \((\bar x,\bar y)\), there exist
\(\beta,\gamma>0\) such that, for all sufficiently large \(k\),
\[
\sqrt{p(x_k,y_k^*(x_k))}
\le
\frac{\gamma}{\beta}
\left\|z_k-\mathcal P_{X\times Y}(z_k+\beta d_k)\right\|.
\]
Set \(\alpha_k:=\beta/\rho_k\). Since
\(p(x_k,y_k^*(x_k))=\frac12\|[c(x_k,y_k^*(x_k))]_+\|^2\), the preceding inequality gives
\[
\|\lambda_k\|
=\sqrt2\,\rho_k\sqrt{p(x_k,y_k^*(x_k))}
\le
\frac{\sqrt2\,\gamma}{\alpha_k}
\left\|z_k-\mathcal P_{X\times Y}(z_k+\alpha_k\rho_k d_k)\right\|.
\]

We now bound the projected residual on the right. By Proposition~\ref{differentiable},
\(\nabla\varphi_k(x_k)=\nabla_x\LL_k(x_k,y_k^*(x_k))\). Hence the assumed approximate stationarity gives 
\begin{equation}\label{thm26_eq1}
u_k \in \nabla_x\LL_k(x_k,y_k^*(x_k))+ \mathcal N_X(x_k) .
\end{equation}
Using the projection characterization of the normal cone,
\[
x_k=\mathcal P_X\!\left(x_k-\alpha_k
\big(\nabla_x\LL_k(x_k,y_k^*(x_k))-u_k\big)\right),
\]
and therefore
\[
\left\|x_k-\mathcal P_X\!\left(x_k-\alpha_k\nabla_x\LL_k(x_k,y_k^*(x_k))\right)\right\|
\le \alpha_k\|u_k\|.
\]
On the other hand, since \(y_k^*(x_k)\) maximizes \(\LL_k(x_k,\cdot)\) over \(Y\),
\[
y_k^*(x_k)=\mathcal P_Y\!\left(y_k^*(x_k)+\alpha_k\nabla_y\LL_k(x_k,y_k^*(x_k))\right).
\]
Combining the two projection relations yields
\[
\left\|z_k-\mathcal P_{X\times Y}\!\left(
z_k-\alpha_k
\begin{bmatrix}
\nabla_x\LL_k(x_k,y_k^*(x_k))\\
-\nabla_y\LL_k(x_k,y_k^*(x_k))
\end{bmatrix}
\right)\right\|
\le \alpha_k\|u_k\|.
\]
Since
$z_k+\alpha_k\rho_k d_k
=z_k-\alpha_k
\begin{bmatrix}
\nabla_x\LL_k(x_k,y_k^*(x_k))\\
-\nabla_y\LL_k(x_k,y_k^*(x_k))
\end{bmatrix}
-\alpha_k h_k,$
nonexpansiveness of the projection gives
\begin{equation}\label{multiplierbound:eq}
    \begin{aligned}
        \left\|z_k-\mathcal P_{X\times Y}(z_k+\alpha_k\rho_k d_k)\right\|
\le \alpha_k(\|u_k\|+C).
    \end{aligned}
\end{equation}
Consequently,
\[
\|\lambda_k\|\le \sqrt2\,\gamma(\|u_k\|+C),
\]
so \(\{\lambda_k\}\) is bounded. Passing to a further subsequence and relabeling,
we may assume \(\lambda_k\to\bar\lambda\in\IR^p_+\).

We next pass to the limiting KKT system. Since
\(\lambda_k=\rho_k[c(x_k,y_k^*(x_k))]_+\), we have
\begin{align*}
    \nabla_x\LL_k(x_k,y_k^*(x_k))=&\nabla_x\ML(x_k,y_k^*(x_k),\lambda_k),\\
\nabla_y\LL_k(x_k,y_k^*(x_k))=&\nabla_y\ML(x_k,y_k^*(x_k),\lambda_k)-\sigma_k y_k^*(x_k).
\end{align*}
Letting \(k\to\infty\) in \eqref{thm26_eq1} and using \(u_k\to0\), the continuity of
\(\nabla_x\ML\), and the closedness of the graph of the normal cone to the
closed convex set \(X\), we obtain
$0\in\nabla_x\ML(\bar x,\bar y,\bar\lambda)+\mathcal N_X(\bar x).$

Similarly, the optimality of \(y_k^*(x_k)\) gives
\[
0\in-\nabla_y\LL_k(x_k,y_k^*(x_k))+\mathcal N_Y(y_k^*(x_k))
=-\nabla_y\ML(x_k,y_k^*(x_k),\lambda_k)+\sigma_k y_k^*(x_k)+\mathcal N_Y(y_k^*(x_k)).
\]
Passing to the limit and using the closedness of the graph of
\(\mathcal N_Y\) yields
\[
0\in-\nabla_y\ML(\bar x,\bar y,\bar\lambda)+\mathcal N_Y(\bar y).
\]

It remains to verify feasibility and complementarity. Since \(\{\lambda_k\}\)
is bounded and \(\rho_k\to\infty\),
$[c(x_k,y_k^*(x_k))]_+=\frac{\lambda_k}{\rho_k}\to0.$
Hence \(c(\bar x,\bar y)\le0\). If \(c_i(\bar x,\bar y)<0\), then
\(c_i(x_k,y_k^*(x_k))<0\) for all sufficiently large \(k\), so
\((\lambda_k)_i=0\) eventually and therefore \(\bar\lambda_i=0\). This proves
\(\bar\lambda^\top c(\bar x,\bar y)=0\). Hence
\((\bar x,\bar y,\bar\lambda)\) satisfies the KKT system~\eqref{KKT}.

Finally, take any \(i\in I^+(\bar\lambda)\). Then
\((\lambda_k)_i>0\) for all sufficiently large \(k\), and since
$(\lambda_k)_i=\rho_k[c_i(x_k,y_k^*(x_k))]_+,$
we have \(c_i(x_k,y_k^*(x_k))>0\) eventually. After discarding finitely many terms and
relabeling, this is exactly the strict-positivity requirement in
Definition~\ref{def:type1-ekkt}. Therefore \((\bar x,\bar y)\) is a Type-I
enhanced KKT point. If the residual sequence satisfies
\(\rho_k\|u_k\|\to0\), this property is preserved by the subsequence extraction
and finite relabeling above. The same certifying sequence then satisfies
Definition~\ref{def:type2-ekkt}, so \((\bar x,\bar y)\) is a Type-II enhanced
KKT point.
\end{proof}

\section{Proof for Section \ref{convergence analysis}}\label{proof for convergence analysis}

\subsection{Proof for Lemma \ref{gradientLipschitz}}
\begin{proof}
Let \(z_i:=(x_i,y_i)\in X\times Y\), \(i=1,2\). For each constraint \(j\), using the
nonexpansiveness of \(t\mapsto[t]_+\), the bounds
\(|[c_j(z)]_+|\le M\), \(\|\nabla c_j(z)\|\le M\), and the Lipschitz
continuity of \(\nabla c_j\), we have
\[
\begin{aligned}
&\|[c_j(z_2)]_+\nabla c_j(z_2)-[c_j(z_1)]_+\nabla c_j(z_1)\|  \\
\le\;&
|[c_j(z_2)]_+-[c_j(z_1)]_+|\,\|\nabla c_j(z_2)\|
+|[c_j(z_1)]_+|\,\|\nabla c_j(z_2)-\nabla c_j(z_1)\|  \\
\le\;&
M|c_j(z_2)-c_j(z_1)|+ML_c\|z_2-z_1\|.
\end{aligned}
\]
By the mean-value theorem and the bound \(\|\nabla c_j\|\le M\),
\(|c_j(z_2)-c_j(z_1)|\le M\|z_2-z_1\|\),
\[
\|[c_j(z_2)]_+\nabla c_j(z_2)-[c_j(z_1)]_+\nabla c_j(z_1)\|
\le (M^2+ML_c)\|z_2-z_1\|.
\]
Summing over \(j=1,\ldots,p\), we obtain
\[
\begin{aligned}
\|\nabla \LL_k(z_2)-\nabla \LL_k(z_1)\|
&\le L_f\|z_2-z_1\|+\sigma_k\|z_2-z_1\| +\rho_k p\bigl(M^2+ML_c\bigr)\|z_2-z_1\|  \\
&=
\bigl(L_f+\rho_k pM^2+\rho_k pML_c+\sigma_k\bigr)\|z_2-z_1\|.
\end{aligned}
\]
Thus \(\nabla\LL_k\) is \(L_k\)-Lipschitz continuous on \(X\times Y\) with
\(L_k=L_f+\rho_k pML_c+\rho_k pM^2+\sigma_k\).

It remains to prove the uniform lower bound. By Assumption~\ref{assum1}, for
each \(x\in X\) there exists \(\tilde y(x)\in\Gamma(x)\). Hence the penalty
term vanishes at \((x,\tilde y(x))\), and
\[
\begin{aligned}
\varphi_k(x)
&=\max_{y\in Y}\LL_k(x,y)
\ge \LL_k(x,\tilde y(x))  \\
&=f(x,\tilde y(x))-\frac{\sigma_k}{2}\|\tilde y(x)\|^2
\ge \inf_{(x,y)\in X\times Y}f(x,y)-\frac{\sigma_k}{2}D_y^2 .
\end{aligned}
\]
Since \(X\times Y\) is compact and \(f\) is continuous, the infimum above is
finite. Also, \(\sigma_k\to0\), so \(\{\sigma_k\}\) is bounded. Hence, with
\(\bar\sigma:=\sup_k\sigma_k<\infty\),
\[
\varphi_k(x)\ge
\inf_{(x,y)\in X\times Y}f(x,y)-\frac{\bar\sigma}{2}D_y^2
:=\underline{\varphi}
\]
for all \(x\in X\) and all \(k\).
\end{proof}


\subsection{Proof for Lemma \ref{lineardescent}}
\begin{proof}
Apply \citep[Lemma~21]{pmlr-v206-xiao23a} to
\(g_k(\cdot):=-\LL_k(x^k,\cdot)\). Since \(g_k\) is
\(\sigma_k\)-strongly convex and \(L_k\)-smooth on \(Y\), the asserted
one-step bound follows immediately.
\end{proof}

\subsection{Proof for Lemma \ref{ydescentlemma}}

Before presenting the formal proof of Lemma \ref{ydescentlemma}, we first establish two auxiliary lemmas that quantify the variations of the exact inner maximizer $y^*_k(x)$ with respect to the primal variable $x$ and the dynamically varying parameters $(\rho_k, \sigma_k)$. 

 \begin{lemma}\label{yxk+1xk}
	 	For any $x,x'\in X$, the solution mapping
$x\mapsto y_k^*(x)$ satisfies
\begin{equation}\label{eq:lip-y-star}
    \|y_k^*(x')-y_k^*(x)\|
    \le
    \frac{L_k}{\sigma_k}\|x'-x\|.
\end{equation}
Moreover, $\varphi_k$ has a Lipschitz continuous gradient on $X$:
\begin{equation}\label{eq:lip-grad-phi}
    \|\nabla \varphi_k(x')-\nabla \varphi_k(x)\|
    \le
    L_{\varphi_k}\|x'-x\|,
\end{equation}
where $L_{\varphi_k}
    :=
    \frac{L_k(L_k+\sigma_k)}{\sigma_k}.$
\end{lemma}

\begin{proof}
The result follows from \citep[Lemma~4.3]{lin2020gradient} applied to
\(\LL_k\).
\end{proof}

The second auxiliary lemma bounds the drift of the inner maximizer induced by the updates of the penalty and regularization parameters for any fixed $x$.
\begin{lemma}\label{yk+1yk}
Let $\{\rho_k\}$ and $\{\sigma_k\}$ be sequences such that
$\rho_{k+1} \ge \rho_{k}>0$ and $\sigma_{k} \ge\sigma_{k+1}>0$.
Then, for any fixed $x \in X$, we have
\begin{align}
\| y_{k+1}^*(x) - y_{k}^*(x)\|
\le \frac{\rho_{k+1}-\rho_{k}}{\sigma_k}M^2
+\frac{\sigma_{k}-\sigma_{k+1}}{\sigma_k} D_y.
\end{align}
\end{lemma}

\begin{proof}
By the first-order optimality conditions of the lower-level problems, we have
\begin{align}\label{yk+1yk:eq1}
0\in  -\nabla_y \LL_{k}(x,y^*_{k}(x))+\mathcal{N}_{Y}( y^*_k(x)),\; 0\in  -\nabla_y \LL_{k+1}(x,y^*_{k+1}(x))+\mathcal{N}_{Y}( y^*_{k+1}(x)).
\end{align}

We first quantify the change in the gradient of the lower-level objective
induced by the update of the penalty parameters.
Expanding the difference between the gradients of $\LL_{k}$ and $\LL_{k+1}$
at $y^*_{k+1}(x)$ yields
\[
\begin{aligned}
&-\nabla_y \LL_{k+1}(x,y^*_{k+1}(x))
+ \nabla_y \LL_{k}(x,y^*_{k+1}(x)) \\
=\,& (\rho_{k+1}-\rho_{k})
[c(x,y^*_{k+1}(x))]_+\nabla_y c(x,y^*_{k+1}(x))
+ (\sigma_{k+1} - \sigma_k)y_{k+1}^*(x).
\end{aligned}
\]

Recalling that \(\|c(x,y)\|\le M\), \(\|\nabla_y c(x,y)\|\le M\), and
\(\|y\|\le D_y\) for all \((x,y)\in X\times Y\), we obtain
\begin{equation}\label{yk+1yk:eq2}
\|\nabla_y \LL_{k+1}(x,y^*_{k+1}(x))
- \nabla_y \LL_{k}(x,y^*_{k+1}(x)) \|
\le  (\rho_{k+1} - \rho_k)M^2 + (\sigma_{k} - \sigma_{k+1})D_y.
\end{equation}

Next, using the optimality condition \eqref{yk+1yk:eq1},
together with the $\sigma_k$-strong concavity of $\LL_k$ with respect to $y$
and the monotonicity of the normal cone $\mathcal{N}_{Y}$,
we obtain
\[
\begin{aligned}
\sigma_k\| y_{k+1}^*(x) - y_{k}^*(x)\|^2
\le\,&
\langle \nabla_y \LL_{k+1}(x,y^*_{k+1}(x))
- \nabla_y \LL_{k}(x,y^*_{k+1}(x)),
\, y^*_{k+1}(x) - y^*_{k}(x) \rangle \\
\le\,&
\|\nabla_y \LL_{k+1}(x,y^*_{k+1}(x))
- \nabla_y \LL_{k}(x,y^*_{k+1}(x)) \|
\| y^*_{k+1}(x) - y^*_{k}(x)\|.
\end{aligned}
\]

Substituting \eqref{yk+1yk:eq2} into the above inequality
 yields the desired result.
\end{proof}

\begin{proof}[Proof for Lemma \ref{ydescentlemma}]
By Young's inequality, for any \(\hat\delta>0\),
\begin{equation}\label{ydescentlemma:eq1}
\begin{aligned}
\E[\|y^{k+1}-y_{k+1}^*(x^{k+1})\|^2\mid \F_k]
\le\;&(1+\hat{\delta})
\E[\|y^{k+1}-y_{k}^*(x^{k})\|^2\mid \F_k]\\
&+\left(1+\frac{1}{\hat{\delta}}\right)
\E[\|y_{k}^*(x^{k})-y_{k+1}^*(x^{k+1})\|^2\mid \F_k].
\end{aligned}
\end{equation}
Take \(\hat{\delta}=\frac{1}{2}\beta_k\sigma_k\). By
Lemma~\ref{lineardescent},
\[
\begin{aligned}
&(1+\hat\delta)\E[\|y^{k+1}-y_{k}^*(x^{k})\|^2\mid \F_k]\\
\le\;&
\left(1+\frac{1}{2}\beta_k\sigma_k\right)(1-\beta_k\sigma_k)
\|y^{k}-y_{k}^*(x^{k})\|^2
+\left(1+\frac{1}{2}\beta_k\sigma_k\right)\beta_k^2\delta^2\\
\le\;&
\left(1-\frac{1}{2}\beta_k\sigma_k\right)
\|y^{k}-y_{k}^*(x^{k})\|^2
+\left(1+\frac{1}{2}\beta_k\sigma_k\right)\beta_k^2\delta^2.
\end{aligned}
\]
Moreover, Lemmas~\ref{yxk+1xk} and~\ref{yk+1yk} give
\[
\begin{aligned}
&\left(1+\frac{1}{\hat{\delta}}\right)
\E[\|y_{k}^*(x^{k})-y_{k+1}^*(x^{k+1})\|^2\mid \F_k]\\
\le\;&
2\left(1+\frac{2}{\beta_k\sigma_k}\right)
\left(\E[\|y_{k}^*(x^{k})-y_{k}^*(x^{k+1})\|^2\mid \F_k]+
\E[\|y_{k}^*(x^{k+1})-y_{k+1}^*(x^{k+1})\|^2\mid \F_k]\right)\\
\le\;&
2\left(1+\frac{2}{\beta_k\sigma_k}\right)
\left[
\frac{L_k^2}{\sigma_k^2}\E[\|x^{k+1}-x^{k}\|^2\mid \F_k]
+\frac{2(\rho_{k+1}-\rho_{k})^2}{\sigma_k^2}M^4
+\frac{2(\sigma_{k}-\sigma_{k+1})^2}{\sigma_k^2}D_y^2
\right].
\end{aligned}
\]
Combining these two estimates with \eqref{ydescentlemma:eq1} yields the
desired inequality.
\end{proof}

\subsection{Proof for Lemma \ref{xdescentlemma}}

Before presenting the proof of Lemma \ref{xdescentlemma}, we establish an auxiliary result that bounds the variation of the approximate value function $\varphi_k(x)$ induced by the updates of  $\rho_k$ and $\sigma_k$.
\begin{lemma}\label{phik+1phik}
	Let $\{\rho_k\}$ and $\{\sigma_k\}$  be sequences such that $\rho_{k+1} \ge \rho_{k}>0$, $\sigma_{k} \ge\sigma_{k+1}>0$. Then, for any $x \in X$, we have 
\begin{align}
    \varphi_{k+1}(x) - \varphi_{k}(x)
        \le \;&\frac{1}{2}(\sigma_{k}-\sigma_{k+1})D_y^2.
\end{align}
\end{lemma}
\begin{proof}
    Recall the definition of $\varphi_{k}(x)$
        $\varphi_{k}(x):=\max_{y\in Y}\LL_{k}(x,y)\ge \LL_{k}(x,y_{k+1}^*(x)).$
    Thus, we have the following inequality:
    \begin{align*}
        		\varphi_{k+1}(x) - \varphi_{k}(x)
        \le\;&\LL_{k+1}(x,y_{k+1}^*(x))-\LL_{k}(x,y_{k+1}^*(x))
        \le \;\frac{1}{2}(\sigma_{k}-\sigma_{k+1})D_y^2,
    \end{align*}
    which completes the proof.
\end{proof}

\begin{proof}[Proof of Lemma \ref{xdescentlemma}]
Let
$\Delta x^k:=x^{k+1}-x^k,\; y_k^*:=y_k^*(x^k).$
We first condition on \(\F_{k+\frac12}\). At this stage, \(y^{k+1}\) is
\(\F_{k+\frac12}\)-measurable, while the remaining randomness in \(x^{k+1}\)
comes from the sample \(\xi_k^x\). By Lemma~\ref{phik+1phik}, applied at
\(x^{k+1}\),
\begin{equation}\label{xdescentlemma:eq1}
\E[\varphi_{k+1}(x^{k+1})-\varphi_k(x^{k+1})\mid \F_{k+\frac12}]
\le \frac{1}{2}(\sigma_k-\sigma_{k+1})D_y^2 .
\end{equation}
Moreover, the Lipschitz continuity of \(\nabla\varphi_k\) from
Lemma~\ref{yxk+1xk}, combined with the standard descent inequality  \citep[Lemma 5.7]{beck2017first},   gives
\begin{equation}\label{xdescentlemma:eq2}
\E[\varphi_k(x^{k+1})\mid \F_{k+\frac12}]-\varphi_k(x^k)
\le
\E[\langle \nabla\varphi_k(x^k),\Delta x^k\rangle\mid \F_{k+\frac12}]
+\frac{L_{\varphi_k}}{2}\E[\|\Delta x^k\|^2\mid \F_{k+\frac12}].
\end{equation}

The projection update
$x^{k+1}=\P_X\bigl(x^k-\alpha_kd_x^k\bigr)
=\underset{x\in X}{\arg\min}\left\{\left\langle d_x^k,x\right\rangle+\frac{1}{2\alpha_k}\|x-x^k\|^2\right\}.$
implies
$\frac{1}{\alpha_k}\|\Delta x^k\|^2
\le \langle -d_x^k,\Delta x^k\rangle .$
Combining this inequality with \eqref{xdescentlemma:eq2} and using
\(\nabla\varphi_k(x^k)=\nabla_x\LL_k(x^k,y_k^*)\), we obtain
\[
\begin{aligned}
&\E[\varphi_k(x^{k+1})\mid \F_{k+\frac12}]-\varphi_k(x^k)
+\left(\frac{1}{\alpha_k}-\frac{L_{\varphi_k}}{2}\right)
\E[\|\Delta x^k\|^2\mid \F_{k+\frac12}] \\
\le \, &
\E[\langle \nabla_x\LL_k(x^k,y_k^*)-\nabla_x\LL_k(x^k,y^{k+1}),
\Delta x^k\rangle\mid \F_{k+\frac12}]+
\E[\langle \nabla_x\LL_k(x^k,y^{k+1})-d_x^k,
\Delta x^k\rangle\mid \F_{k+\frac12}].
\end{aligned}
\]
For the first term on the right-hand side, Young's inequality gives
\[
\begin{aligned}
&\E[\langle \nabla_x\LL_k(x^k,y_k^*)-\nabla_x\LL_k(x^k,y^{k+1}),
\Delta x^k\rangle\mid \F_{k+\frac12}]\\
\le&
\frac{\alpha_k}{2}L_k^2
\|y^{k+1}-y_k^*\|^2
+\frac{1}{2\alpha_k}\E[\|\Delta x^k\|^2\mid \F_{k+\frac12}],
\end{aligned}
\]
where \(y^{k+1}\) and \(y_k^*\) are \(\F_{k+\frac12}\)-measurable.
For the second term, by the definition of \(e_x^k\),
\[
\begin{aligned}
\E[\langle \nabla_x\LL_k(x^k,y^{k+1})-d_x^k,\Delta x^k\rangle
\mid \F_{k+\frac12}]\le
\alpha_k\E[\|e_x^k\|^2\mid \F_{k+\frac12}]
+\frac{1}{4\alpha_k}\E[\|\Delta x^k\|^2\mid \F_{k+\frac12}].
\end{aligned}
\]
Substituting these two bounds into the preceding inequality, moving the terms
\(\frac{1}{2\alpha_k}\E[\|\Delta x^k\|^2\mid\F_{k+\frac12}]\) and
\(\frac{1}{4\alpha_k}\E[\|\Delta x^k\|^2\mid\F_{k+\frac12}]\) to the
left-hand side, and then adding \eqref{xdescentlemma:eq1}, yields
\[
\begin{aligned}
&\E[\varphi_{k+1}(x^{k+1})\mid \F_{k+\frac12}]-\varphi_k(x^k)
+\left(\frac{1}{4\alpha_k}-\frac{L_{\varphi_k}}{2}\right)
\E[\|\Delta x^k\|^2\mid \F_{k+\frac12}]\\
\le \, &
\frac{\alpha_k}{2}L_k^2\|y^{k+1}-y_k^*\|^2
+\alpha_k\E[\|e_x^k\|^2\mid \F_{k+\frac12}]
+\frac{1}{2}(\sigma_k-\sigma_{k+1})D_y^2.
\end{aligned}
\]
Finally, taking conditional expectation with respect to \(\F_k\), using the
tower property, and applying Lemma~\ref{lineardescent} give
\[
\E[\|y^{k+1}-y_k^*\|^2\mid\F_k]
\le \|y^k-y_k^*\|^2+\beta_k^2\delta^2.
\]
Substituting this bound and using
\(\E[\E[\|e_x^k\|^2\mid\F_{k+\frac12}]\mid\F_k]
=\E[\|e_x^k\|^2\mid\F_k]\) proves the stated inequality.
\end{proof}

\subsection{Proof for Lemma \ref{variancereduction}}
\begin{proof}
Recall that $\F_k$ is the $\sigma$-algebra generated by the history up to iteration
\(k\). Denote the exact directions by
$D_x^{j}:=\nabla_x \LL_{j}(x^{j},y^{j+1}),\qquad j=k-1,k,$
so that \(e_x^{j}=d_x^{j}-D_x^{j}\). For brevity, write
$G_k:=\nabla_x \Psi_{k}(x^{k},y^{k+1};\xi_{k}^x),\qquad
\widehat G_k:=\nabla_x \Psi_{k-1}(x^{k-1},y^{k};\xi_{k}^x).$

Since the oracle is unbiased and \(\xi_k^x\) is independent of \(\F_{k+\frac12}\), we have
$\E[G_k\mid \F_{k+\frac12}]=D_x^{k},
\;
\E[\widehat G_k\mid \F_{k}]=D_x^{k-1}.$
Consequently, by the tower property and the \(\F_k\)-measurability of
\(e_x^{k-1}\),
we have
$\E\!\left[\left\langle e_x^{k-1},\,G_k-D_x^{k}\right\rangle\mid \F_{k}\right]=0,\;\E\!\left[\left\langle e_x^{k-1},\,D_x^{k-1}-\widehat G_k\right\rangle\mid \F_{k}\right]=0.$

Thus, the variance at the \(k\)-th iteration can be decomposed as
\begin{equation}\label{varaiancereduction:eq0}
\begin{aligned}
\E[\|e_x^{k}\|^2\mid \F_{k}]
=&\E[\|G_k-D_x^{k}+(1-\eta_{k})(D_x^{k-1}-\widehat G_k)\|^2\mid \F_{k}]
+(1-\eta_{k})^2\|e_x^{k-1}\|^2.
\end{aligned}
\end{equation}
Moreover,
$G_k-D_x^{k}+(1-\eta_k)(D_x^{k-1}-\widehat G_k)
=(1-\eta_k)(G_k-\widehat G_k+D_x^{k-1}-D_x^{k})+\eta_k(G_k-D_x^{k}).$

Hence, by Caychr-schwarz inequality, the bounded variance
assumption, and \(0\le\eta_k\le1\),
\begin{equation}\label{variancereduction:eq1}
\begin{aligned}
&\E[\|G_k-D_x^{k}+(1-\eta_{k})(D_x^{k-1}-\widehat G_k)\|^2\mid \F_{k}]
\le\;
2\E[\|G_k-\widehat G_k+D_x^{k-1}-D_x^{k}\|^2\mid \F_{k}]
+2\eta_{k}^2\delta^2.
\end{aligned}
\end{equation}

Now define
\begin{align*}
    A_k:=&G_k-\nabla_x \Psi_{k}(x^{k-1},y^{k};\xi_{k}^x),\\
    B_k:=&\nabla_x \Psi_{k}(x^{k-1},y^{k};\xi_{k}^x)-\widehat G_k
-\bigl(\nabla_x \LL_{k}(x^{k-1},y^{k})-\nabla_x \LL_{k-1}(x^{k-1},y^{k})\bigr),\\
C_k:=&\nabla_x \LL_{k}(x^{k},y^{k+1})-\nabla_x \LL_{k}(x^{k-1},y^{k}).
\end{align*}
Then we have 
$G_k-\widehat G_k+D_x^{k-1}-D_x^{k}=A_k+B_k-C_k.$
Therefore, by Caychr-schwarz inequality,
\[
\E[\|G_k-\widehat G_k+D_x^{k-1}-D_x^{k}\|^2\mid \F_k]
\le
3\E[\|A_k\|^2\mid \F_k]
+3\E[\|B_k\|^2\mid \F_k]
+3\E[\|C_k\|^2\mid \F_k].
\]

For the middle term, note that the stochasticity arises solely from the stochastic
objective realization \(F(\cdot,\cdot;\xi)\), while the penalty and regularization terms are
deterministic and independent of the random sample \(\xi\). Therefore,
$B_k=0,$ and hence $\E[\|B_k\|^2\mid \F_k]=0.$

For the first term, combining Assumption~\ref{stochasticsmooth} with the same estimate as
in the proof of Lemma~\ref{gradientLipschitz}, we obtain
\begin{align*}
\E[\|A_k\|^2\mid \F_k]
=&\E[\|\nabla_x \Psi_{k}(x^{k},y^{k+1};\xi_{k}^x)-\nabla_x \Psi_{k}(x^{k-1},y^{k};\xi_{k}^x)\|^2\mid \F_{k}]\\
\le\;&
2(L_{f}+\rho_{k}pML_c+\rho_{k}pM^2)^2
\E[\|x^{k}-x^{k-1}\|^2+\|y^{k+1}-y^{k}\|^2\mid \F_{k}]\\
\le\;&
2L_{k}^2\|x^{k}-x^{k-1}\|^2
+2L_{k}^2\E[\|y^{k+1}-y^{k}\|^2\mid \F_{k}].
\end{align*}
For the third term, since \(\nabla_x\LL_k\) is \(L_k\)-Lipschitz continuous,
\[
\E[\|C_k\|^2\mid \F_k]
\le
L_k^2\|x^k-x^{k-1}\|^2
+L_k^2\E[\|y^{k+1}-y^k\|^2\mid \F_k].
\]
Combining the above estimates yields
\begin{equation}\label{variancereduction:eq2}
\begin{aligned}
&\E[\|G_k-\widehat G_k+D_x^{k-1}-D_x^{k}\|^2\mid \F_k]
\le\;
9L_k^2\|x^k-x^{k-1}\|^2
+9L_k^2\E[\|y^{k+1}-y^k\|^2\mid \F_k].
\end{aligned}
\end{equation}
Combining \eqref{variancereduction:eq1} and \eqref{variancereduction:eq2}, we have
\[
\E[\|G_k-D_x^{k}+(1-\eta_{k})(D_x^{k-1}-\widehat G_k)\|^2\mid \F_{k}]
\le
18L_{k}^2\Bigl(\|x^{k}-x^{k-1}\|^2+\E[\|y^{k+1}-y^{k}\|^2\mid \F_{k}]\Bigr)
+2\eta_{k}^2\delta^2.
\]
Further, by the non-expansiveness of projection and the boundedness of \(y_k^*(x^k)\),
we have
\begin{align*}
&\E[\|y^{k+1}-y^{k}\|^2\mid \F_{k}]\\
=\;&\E[\|\P_Y(y^{k}+\beta_{k}d^{y}_{k})-\P_Y(y^{k})\|^2\mid \F_{k}]\\
\le\;&3\beta_{k}^2\E[\|d^{y}_{k}-\nabla_{y}\LL_{k}(x^{k},y^{k})\|^2\mid \F_{k}]
+3\beta_{k}^2L_{k}^2\|y^{k}-y^*_{k}(x^{k})\|^2\\
&+3\beta_{k}^2\|\nabla_y f(x^{k},y^*_{k}(x^{k}))
-\rho_{k}\sum_{i=1}^p[c_i(x^{k},y^*_{k}(x^{k}))]_+\nabla_y c_i(x^{k},y^*_{k}(x^{k}))
-\sigma_k y^*_{k}(x^{k})\|^2\\
\le\;&3\beta_{k}^2\delta^2
+3\beta_{k}^2(M+\rho_{k}pM^2+\sigma_k D_y)^2
+3\beta_{k}^2L_{k}^2\|y^{k}-y^*_{k}(x^{k})\|^2.
\end{align*}
By substituting this estimate into the previous inequality, and using \eqref{varaiancereduction:eq0}, the proof is completed.
\end{proof}

\subsection{Proof for Proposition \ref{lyapunovproposition}}

\begin{proof}
If necessary, replace the variance constant \(\delta\) in
Assumption~\ref{assumgradient} by \(\max\{\delta,1\}\).
We first recall the coefficient sequences and parameter choices. For \(k\ge1\),
\[
a_k=(k+1)^{-2t},\; b_k=(k+1)^{-3t},\; c_k=(k+1)^{-7t},\; d_k=(k+1)^{-4t}, 
\]
\[
\alpha_k=\alpha_0(k+1)^{-6t-s},\;
\beta_k=\beta_0(k+1)^{-t-s},\;
\eta_k=\eta_0(k+1)^{-s},\;
\sigma_k=\sigma_0(k+1)^{-t},\;
\rho_k=\rho_0(k+1)^t.
\]
Moreover,
$L_k=L_f+\rho_k pML_c+\rho_k pM^2+\sigma_k=\mathcal O(k^t),
\;
L_{\varphi_k}=\mathcal O(k^{3t}).$
It follows that \(\beta_kL_k=\mathcal O((k+1)^{-s})\to0\), so
\(0<\beta_k\le 1/L_k\) for all sufficiently large \(k\). Since
\(\eta_k\to0\), we also have \(0\le\eta_k\le1\) for all sufficiently large
\(k\). In the remainder of the proof, we consider only such \(k\).

Set
$\Delta_k:=\E[V_{k+1}\mid \F_k]-V_k.$
Combining Lemmas~\ref{ydescentlemma}, \ref{xdescentlemma}, and \ref{variancereduction},
and using $a_{k+1}\le a_k$, $b_{k+1}\le b_k$, $\eta_k\le 1$, and, for all
sufficiently large \(k\), \(L_{k+1}\ge L_k\),
we obtain the raw recursion
\begin{equation}\label{eq:prop35-raw}
\Delta_k
\le
-A_k \E[\|x^{k+1}-x^k\|^2\mid \F_k]
-B_k \|y^{k}-y_k^*(x^k)\|^2
+T_k
-D_k\|x^k-x^{k-1}\|^2
+\zeta_k+\mathcal R_k^{\mathrm{raw}},
\end{equation}
where 
\begin{align*}
A_k &:=
\frac{a_k}{4\alpha_k}
-\frac{a_kL_{\varphi_k}}{2}
-2b_k\Bigl(1+\frac{2}{\beta_k\sigma_k}\Bigr)\frac{L_k^2}{\sigma_k^2}
-d_k,\;
B_k :=
\frac12 b_k\beta_k\sigma_k
-\frac12 a_k\alpha_kL_k^2
-54c_k\beta_k^2L_k^4,\\
T_k &:=
(-2c_k\eta_k+c_k\eta_k^2)\|e_x^{k-1}\|^2
+a_k\alpha_k\E[\|e_x^k\|^2\mid \F_k],\;
D_k :=
d_k-18c_kL_k^2,
\end{align*}
and
$\zeta_k
:=
\frac12 a_k(\sigma_k-\sigma_{k+1})D_y^2
+2b_k\Bigl(1+\frac{2}{\beta_k\sigma_k}\Bigr)
(
\frac{2(\rho_{k+1}-\rho_k)^2}{\sigma_k^2}M^4
+
\frac{2(\sigma_k-\sigma_{k+1})^2}{\sigma_k^2}D_y^2
).$
Using the convention \(\delta\ge1\), \(\mathcal R_k^{\mathrm{raw}}\) collects the
remaining nonnegative terms of order
\[
\mathcal R_k^{\mathrm{raw}}
=
\mathcal O\!\left(
a_k\alpha_kL_k^2\beta_k^2
+b_k\beta_k^2
+c_k\eta_k^2
+c_k\beta_k^2L_k^2
+c_k\beta_k^2L_k^4
\right)\delta^2.
\]

\noindent\textbf{Step 1: convert the $x$-difference term to the generalized gradient residual.}
By the projection inequality and Lemma~\ref{lineardescent},
\begin{align*}
\E[\|\mathcal{G}_k(x^k)\|^2\mid \F_k]
\le
\frac{3}{\alpha_k^2}\Big(
&\E[\|x^{k+1}-x^k\|^2\mid \F_k]
+\alpha_k^2L_k^2\|y^k-y_k^*(x^k)\|^2\\
&+\alpha_k^2\E[\|e_x^k\|^2\mid \F_k]
+\alpha_k^2L_k^2\beta_k^2\delta^2
\Big).
\end{align*}
Hence, whenever \(A_k\ge a_k/(8\alpha_k)\),
\begin{equation}
    \begin{aligned}
-A_k\E[\|x^{k+1}-x^k\|^2\mid \F_k]
\le\;
&-\frac{a_k\alpha_k}{24}\E[\|\mathcal{G}_k(x^k)\|^2\mid \F_k] +\frac{a_k\alpha_kL_k^2}{8}\|y^k-y_k^*(x^k)\|^2\\
&
+\frac{a_k\alpha_k}{8}\E[\|e_x^k\|^2\mid \F_k]+C\,a_k\alpha_kL_k^2\beta_k^2\delta^2.\\
\label{eq:prop35-gstep}
\end{aligned}
\end{equation}
\noindent\textbf{Step 2: absorb the estimator-variance block.}
Substituting \eqref{eq:prop35-gstep} into \eqref{eq:prop35-raw}, the remaining
$e_x^k$-term becomes
$\widetilde T_k
:=
(-2c_k\eta_k+c_k\eta_k^2)\|e_x^{k-1}\|^2
+\frac98 a_k\alpha_k\E[\|e_x^k\|^2\mid \F_k].$
Applying Lemma~\ref{variancereduction} to $\widetilde T_k$ yields
\begin{align*}
\widetilde T_k
\le\;
&-\hat{T}_k\|e_x^{k-1}\|^2+\frac{81}{4}a_k\alpha_kL_k^2\|x^k-x^{k-1}\|^2
+\frac{243}{4}a_k\alpha_k\beta_k^2L_k^4\|y^k-y_k^*(x^k)\|^2
+\mathcal R_k^{\mathrm{var}},
\end{align*}
where,
$\hat{T}_k:=\Bigl(2c_k\eta_k-c_k\eta_k^2-\frac98 a_k\alpha_k\Bigr)$
and using $(M+\rho_{k}pM^2+\sigma_k D_y )^2=\mathcal O(L_k^2)$,
$\mathcal R_k^{\mathrm{var}}
=
\mathcal O\!\left(
a_k\alpha_k\eta_k^2
+a_k\alpha_k\beta_k^2L_k^2
+a_k\alpha_k\beta_k^2L_k^4
\right)\delta^2.$
Therefore,
\begin{equation}\label{eq:prop35-pre-final}
\Delta_k
\le
-\frac{a_k\alpha_k}{24}\E[\|\mathcal{G}_k(x^k)\|^2\mid \F_k]
-\hat{B}_k\|y^k-y_k^*(x^k)\|^2
-\hat{D}_k\|x^k-x^{k-1}\|^2-\hat{T}_k\|e_x^{k-1}\|^2
+\zeta_k+r_k,
\end{equation}
where \(r_k:=\mathcal R_k^{\mathrm{raw}}+\mathcal R_k^{\mathrm{var}}\) and
\begin{align*}
\hat{B}_k
:=
\frac12 b_k\beta_k\sigma_k
-\frac58 a_k\alpha_kL_k^2
-54c_k\beta_k^2L_k^4
-\frac{243}{4}a_k\alpha_k\beta_k^2L_k^4,\;\;
\hat{D}_k
:=
d_k-18c_kL_k^2-\frac{81}{4}a_k\alpha_kL_k^2.
\end{align*}

\medskip
\noindent\textbf{Step 3: order comparisons.}
It remains to verify that the positive terms in
\eqref{eq:prop35-pre-final} dominate the lower-order corrections. First,
\begin{equation*}
    \begin{aligned}
        \frac{a_k}{\alpha_k}=\mathcal O((k+1)^{4t+s}),\qquad
a_kL_{\varphi_k}=\mathcal O((k+1)^t),\\
b_k\left(1+\frac{2}{\beta_k\sigma_k}\right)\frac{L_k^2}{\sigma_k^2}
=\mathcal O((k+1)^{3t+s}),\qquad
d_k=\mathcal O((k+1)^{-4t}).
    \end{aligned}
\end{equation*}
Since \(t>0\), we have \(4t+s>\max\{t,3t+s,-4t\}\). Hence, for all
sufficiently large \(k\), $A_k\ge \frac{a_k}{8\alpha_k}.$
Similarly,
\[
b_k\beta_k\sigma_k=\mathcal O((k+1)^{-5t-s}),\qquad
a_k\alpha_kL_k^2=\mathcal O((k+1)^{-6t-s}),
\]
\[
c_k\beta_k^2L_k^4=\mathcal O((k+1)^{-5t-2s}),\qquad
a_k\alpha_k\beta_k^2L_k^4=\mathcal O((k+1)^{-6t-3s}).
\]
Thus \(\hat B_k\ge \frac14 b_k\beta_k\sigma_k\) for all sufficiently large
\(k\). Moreover,
\[
d_k=\mathcal O((k+1)^{-4t}),\qquad
c_kL_k^2=\mathcal O((k+1)^{-5t}),\qquad
a_k\alpha_kL_k^2=\mathcal O((k+1)^{-6t-s}),
\]
which yields \(\hat D_k\ge0\) eventually. For the estimator-error coefficient,
\[
2c_k\eta_k=\mathcal O((k+1)^{-7t-s}),\qquad
c_k\eta_k^2=\mathcal O((k+1)^{-7t-2s}),\qquad
a_k\alpha_k=\mathcal O((k+1)^{-8t-s}),
\]
and therefore \(\hat T_k\ge0\) for all sufficiently large \(k\).

Finally, the parameter-drift term satisfies
$\zeta_k=\mathcal O((k+1)^{-1-3t})+\mathcal O((k+1)^{3t+s-2}).$
Because \(t>0\) and \(3t+s<1\), the sequence \(\{\zeta_k\}\) is summable.
The remaining term satisfies
$r_k=\mathcal O((k+1)^{-5t-2s})=\mathcal O(b_k\beta_k^2),$
and hence \(r_k\le Cb_k\beta_k^2\delta^2\), using again the convention
\(\delta\ge1\). Substituting these estimates into
\eqref{eq:prop35-pre-final} gives
\[
\E[V_{k+1}\mid \F_k]-V_k
\le
-\frac{a_k\alpha_k}{24}\E[\|\mathcal{G}_k(x^k)\|^2\mid \F_k]
-\frac14 b_k\beta_k\sigma_k\|y^k-y_k^*(x^k)\|^2
+C b_k\beta_k^2\delta^2
+\zeta_k,
\]
where $\{\zeta_k\}$ is a summable nonnegative sequence. This completes the proof.
\end{proof}

\subsection{Proof for Theorem \ref{convergethm}}
Before proving Theorem \ref{convergethm}, we present a tail-selection lemma showing that the Lyapunov descent recursion in Proposition \ref{lyapunovproposition} yields iterates with the desired rates.
\begin{lemma}\label{lem:shared-tail-selection}
Let $\{(x^k,y^k)\}$ be generated by SPACO under the assumptions of
Proposition~\ref{lyapunovproposition}, and suppose that \(2s+5t\neq1\). Define
$I_K:=\{\lfloor K/2\rfloor,\ldots,K\},\; m_K:=\lfloor K/2\rfloor.$
Then there exists a constant $C>0$ such that, for all sufficiently large $K$,
one can choose an index $k_K\in I_K$ satisfying
\begin{align}
\mathbb E\big[\|\mathcal{G}_{k_K}(x^{k_K})\|^2\big]
&\le C\Big(K^{-(1-8t-s)}+K^{-(s-3t)}\Big),\label{eq:shared-tail-G}\\
\mathbb E\big[\|y^{k_K}-y_{k_K}^*(x^{k_K})\|^2\big]
&\le C\Big(K^{-(1-5t-s)}+K^{-s}\Big).\label{eq:shared-tail-y}
\end{align}
Moreover, if $K_j:=4^j$ and $k_j:=k_{K_j}$, then $\{k_j\}$ is strictly increasing and
\begin{equation}\label{eq:shared-tail-sum}
\sum_{j=1}^\infty \mathbb E\Big[
\|\mathcal{G}_{k_j}(x^{k_j})\|^2
+\|y^{k_j}-y_{k_j}^*(x^{k_j})\|^2
\Big] < \infty.
\end{equation}
If, in addition, $s>5t$ and $10t+s<1$, then the same choice also satisfies
\begin{equation}\label{eq:shared-tail-sum-scaled}
\sum_{j=1}^\infty \mathbb E\Big[
\rho_{k_j}^2\|\mathcal{G}_{k_j}(x^{k_j})\|^2
+\rho_{k_j}^2\|y^{k_j}-y_{k_j}^*(x^{k_j})\|^2
\Big] < \infty.
\end{equation}
\end{lemma}

\begin{proof}
By Proposition~\ref{lyapunovproposition}, there exist $k_0\ge1$, $C>0$, and a summable nonnegative
sequence $\{\zeta_k\}$ such that for all $k\ge k_0$,
\begin{equation}\label{eq:shared-tail-one-step}
\frac{a_k\alpha_k}{24}\,\mathbb E\big[\|\mathcal{G}_k(x^k)\|^2\big]
+\frac{b_k\beta_k\sigma_k}{4}\,\mathbb E[\|y^k-y_k^*(x^k)\|^2]
\le
\mathbb E[V_k]-\mathbb E[V_{k+1}]
+C\,b_k\beta_k^2\delta^2+\zeta_k.
\end{equation}
Take $K$ large enough so that $m_K\ge k_0$. Summing
\eqref{eq:shared-tail-one-step} from $k=m_K$ to $K$ gives
\begin{equation}
    \begin{aligned}
&\sum_{k=m_K}^K a_k\alpha_k\,\mathbb E\big[\|\mathcal{G}_k(x^k)\|^2\big]
+\sum_{k=m_K}^K b_k\beta_k\sigma_k\,\mathbb E[\|y^k-y_k^*(x^k)\|^2]
\\\le\;&
C\Big(
\mathbb E[V_{m_K}]
+\sum_{k=m_K}^K b_k\beta_k^2
+\sum_{k=m_K}^K \zeta_k
\Big). 
\end{aligned}\label{eq:shared-tail-block}
\end{equation}
Using the assumption \(2s+5t\neq1\) imposed in this lemma,
summing \eqref{eq:shared-tail-one-step} from $k=k_0$ to
$m_K-1$ and using $V_{m_K}\ge0$ yields
\[
\mathbb E[V_{m_K}]
\le
\mathbb E[V_{k_0}]
+C\sum_{k=k_0}^{m_K-1} b_k\beta_k^2
+\sum_{k=k_0}^{m_K-1}\zeta_k
\le
C\bigl(1+K^{1-5t-2s}\bigr).
\]
Since $2s+5t\neq 1$, \eqref{eq:shared-tail-block} becomes
\begin{equation}\label{eq:shared-tail-tailbound}
\sum_{k\in I_K} a_k\alpha_k\,\mathbb E\big[\|\mathcal{G}_k(x^k)\|^2\big]
+\sum_{k\in I_K} b_k\beta_k\sigma_k\,\mathbb E[\|y^k-y_k^*(x^k)\|^2]
\le
C\bigl(1+K^{1-5t-2s}\bigr).
\end{equation}
For $k\in I_K$, Proposition~\ref{lyapunovproposition} gives
\[
a_k\alpha_k=\mathcal{O}(k^{-(8t+s)})\ge c_1K^{-(8t+s)},
\qquad
b_k\beta_k\sigma_k=\mathcal{O}(k^{-(5t+s)})\ge c_2K^{-(5t+s)}
\]
for some constants $c_1,c_2>0$, while $|I_K|\ge K/2$. Therefore,
\[
\frac{1}{|I_K|}\sum_{k\in I_K}
\Big(
\mathbb E[\|\mathcal{G}_k(x^k)\|^2]
+K^{3t}\mathbb E[\|y^k-y_k^*(x^k)\|^2]
\Big)
\le
C\Big(K^{-(1-8t-s)}+K^{-(s-3t)}\Big).
\]
Choosing $k_K\in I_K$ no larger than this average gives
\[
\mathbb E[\|\mathcal{G}_{k_K}(x^{k_K})\|^2]
+K^{3t}\mathbb E[\|y^{k_K}-y_{k_K}^*(x^{k_K})\|^2]
\le
C\Big(K^{-(1-8t-s)}+K^{-(s-3t)}\Big),
\]
which immediately implies \eqref{eq:shared-tail-G} and
\eqref{eq:shared-tail-y}.

Now take $K_j:=4^j$ and $k_j:=k_{K_j}$. Since
$I_{K_j}=\{2\cdot 4^{j-1},\ldots,4^j\},$
the sequence $\{k_j\}$ is strictly increasing. Moreover,
\eqref{eq:shared-tail-G}--\eqref{eq:shared-tail-y} yield
\begin{equation*}
    \begin{aligned}
        \mathbb E[\|\mathcal{G}_{k_j}(x^{k_j})\|^2]
\le
C\Big(4^{-j(1-8t-s)}+4^{-j(s-3t)}\Big),\;
\mathbb E[\|y^{k_j}-y_{k_j}^*(x^{k_j})\|^2]
\le
C\Big(4^{-j(1-5t-s)}+4^{-js}\Big).
    \end{aligned}
\end{equation*}
Since $8t+s<1$ and $s>3t$, all exponents above are positive, and
\eqref{eq:shared-tail-sum} follows.

Under the stronger assumptions $s>5t$ and $10t+s<1$, we additionally have
$\rho_{k_j}^2\le C K_j^{2t}=C4^{2jt}$. Multiplying the previous bounds by
$\rho_{k_j}^2$ gives
\begin{align*}
    \mathbb E[\rho_{k_j}^2\|\mathcal{G}_{k_j}(x^{k_j})\|^2]
\le
C\Big(4^{-j(1-10t-s)}+4^{-j(s-5t)}\Big),\\
\mathbb E[\rho_{k_j}^2\|y^{k_j}-y_{k_j}^*(x^{k_j})\|^2]
\le
C\Big(4^{-j(1-7t-s)}+4^{-j(s-2t)}\Big).
\end{align*}
Again all exponents are positive, so \eqref{eq:shared-tail-sum-scaled}
holds.
\end{proof}

\begin{proof}[Proof of Theorem \ref{convergethm}]
Let $k:=k_K$ be selected by Lemma~\ref{lem:shared-tail-selection}, and write
$y_k^*:=y_{k}^*(x^{k}).$
By Jensen's inequality and Lemma~\ref{lem:shared-tail-selection},
\begin{align}\label{eq:std-kkt}
\E[\|\mathcal G_k(x^k)\|]
&\le
C\Bigl(K^{-\frac{1-8t-s}{2}}+K^{-\frac{s-3t}{2}}\Bigr),\;
\E[\|y^k-y_k^*\|]
\le
C\Bigl(K^{-\frac{1-5t-s}{2}}+K^{-\frac{s}{2}}\Bigr).
\end{align}

To bound the feasibility violation, choose any
\(\tilde y^k\in\Gamma(x^{k})\), whose existence follows from
Assumption~\ref{assum1}. Since \(y_{k}^*\) maximizes
\(\LL_{k}(x^{k},\cdot)\) over \(Y\), we have
\[
\frac{\rho_{k}}{2}
\big\|[c(x^{k},y_{k}^*)]_+\big\|^2
\le
f(x^{k},y_{k}^*)
-f(x^{k},\tilde y^k)
+\frac{\sigma_{k}}{2}
\Big(\|\tilde y^k\|^2-\|y_{k}^*\|^2\Big).
\]
By compactness of $X\times Y$, the right-hand side is bounded by a constant,
hence
$\big\|[c(x^{k},y_{k}^*)]_+\big\|
\le
C\rho_{k}^{-\frac{1}{2}}
\le
CK^{-\frac{t}{2}}.$
Since $ [\cdot]_+$ and $c$ are Lipschitz on the compact set $X\times Y$,
recalling the definition of $M$, we have
\[
\big\|[c(x^{k},y^{k})]_+\big\|
\le
\big\|[c(x^{k},y_{k}^*)]_+\big\|
+M\|y^{k}-y_{k}^*\|.
\]
Taking expectations and using \eqref{eq:std-kkt}, we obtain
\begin{equation}\label{eq:thm38-39-feas}
\mathbb E\big[\|[c(x^{k},y^{k})]_+\|\big]
\le
C\Big(
K^{-\frac{t}{2}}
+K^{-\frac{1-5t-s}{2}}
+K^{-s/2}
\Big).
\end{equation}
Furthermore, set $\bar\lambda_k:=\rho_k[c(x^k,y^k)]_+.$
For the $y$-stationarity residual, using
$\nabla_y\ML(x^k,y^k,\bar\lambda_k)
=
\nabla_y\LL_k(x^k,y^k)+\sigma_k y^k,$ 
the nonexpansiveness of $\mathcal P_Y$, and the optimality of
$y_k^*,$
we obtain
\[
\begin{aligned}
&\|y^k-\mathcal P_Y(y^k+\nabla_y\ML(x^k,y^k,\bar\lambda_k))\|\\
\le\;&
\|y^k-y_k^*\|
+\Big\|
\mathcal P_Y\bigl(y_k^*+\nabla_y\LL_k(x^k,y_k^*)\bigr)
-\mathcal P_Y\bigl(y^k+\nabla_y\LL_k(x^k,y^k)+\sigma_k y^k\bigr)
\Big\|\\
\le\;&
\|y^k-y_k^*\|
+\|y_k^*-y^k+\nabla_y\LL_k(x^k,y_k^*)-\nabla_y\LL_k(x^k,y^k)-\sigma_k y^k\|\\
\le\;&
(2+L_k)\|y^k-y_k^*\|+\sigma_k\|y^k\|.
\end{aligned}
\]
Since \(L_k=O(K^t)\), \(\sigma_k=O(K^{-t})\), and \(Y\) is compact, this gives
\begin{equation}\label{eq:std-kkt-yres}
\E[\|y^k-\mathcal P_Y(y^k+\nabla_y\ML(x^k,y^k,\bar\lambda_k))\|]
\le
C\Bigl(
K^{-\frac{1-7t-s}{2}}
+K^{-\frac{s-2t}{2}}
+K^{-t}
\Bigr).
\end{equation}

For the $x$-stationarity residual, note that
$\nabla_x\ML(x^k,y^k,\bar\lambda_k)=\nabla_x\LL_k(x^k,y^k).$
Since $\alpha_k\to0$, we have $\alpha_k\le1$ for all sufficiently large $K$.
Using the standard monotonicity of the projected gradient mapping in the
stepsize,
$\|x^k-\mathcal P_X(x^k-\nabla_x\ML(x^k,y^k,\bar\lambda_k))\|
\le
\frac1{\alpha_k}
\|x^k-\mathcal P_X(x^k-\alpha_k\nabla_x\LL_k(x^k,y^k))\|.$

Applying the nonexpansiveness of $\mathcal P_X$, we further get
\[
\begin{aligned}
&\|x^k-\mathcal P_X(x^k-\nabla_x\ML(x^k,y^k,\bar\lambda_k))\|\\
\le\;&
\frac1{\alpha_k}
\|x^k-\mathcal P_X(x^k-\alpha_k\nabla_x\LL_k(x^k,y_k^*))\|
+\|\nabla_x\LL_k(x^k,y^k)-\nabla_x\LL_k(x^k,y_k^*)\|\\
\le\;&
\|\mathcal G_k(x^k)\|+L_k\|y^k-y_k^*\|.
\end{aligned}
\]
Hence, by \eqref{eq:std-kkt},
\begin{equation}\label{eq:std-kkt-xres}
\E[\|x^k-\mathcal P_X(x^k-\nabla_x\ML(x^k,y^k,\bar\lambda_k))\|]
\le
C\Bigl(
K^{-\frac{1-8t-s}{2}}
+K^{-\frac{s-3t}{2}}
+K^{-\frac{1-7t-s}{2}}
+K^{-\frac{s-2t}{2}}
\Bigr).
\end{equation}

Now set
$\tau:=\min\Bigl\{\frac{1-8t-s}{2},\,\frac{s-3t}{2},\,\frac{t}{2}\Bigr\}>0.$
Since
$\frac{1-5t-s}{2}
=
\frac{1-8t-s}{2}+\frac{3t}{2}\ge \tau,\;
\frac{s}{2}
=
\frac{s-3t}{2}+\frac{3t}{2}\ge \tau,$
and
$\frac{1-7t-s}{2}
=
\frac{1-8t-s}{2}+\frac{t}{2}\ge \tau,\;
\frac{s-2t}{2}
=
\frac{s-3t}{2}+\frac{t}{2}\ge \tau,\;
t\ge \tau,$
combining \eqref{eq:std-kkt},
\eqref{eq:thm38-39-feas}, \eqref{eq:std-kkt-yres}, and
\eqref{eq:std-kkt-xres} gives
\[
\max\Bigl\{
\mathbb E\big[\|\mathcal{G}_{k}(x^{k})\|\big],
\mathbb E\big[\|y^{k}-y_{k}^*(x^{k})\|\big],
\mathbb E\big[\|[c(x^{k},y^{k})]_+\|\big],
\mathcal R_{\rm KKT}(x^k,y^k,\bar\lambda_k)
\Bigr\}
\le
CK^{-\tau}.
\]
In particular,
$\mathcal R_{\rm KKT}(x^k,y^k,\bar\lambda_k)\le CK^{-\tau}.$
Since \(k=k_K\in I_K\subset\{0,\ldots,K\}\), there exists an iterate \(k\le K\)
with the same bound. Therefore it suffices to choose
\(K=\mathcal O(\epsilon^{-1/\tau})\), which proves the theorem.
\end{proof}

\subsection{Proof of Theorem \ref{thm:converge_to_enhanced_kkt}}
\begin{proof}
Let $K_j:=4^j$ and choose $k_j:=k_{K_j}$ as in
Lemma~\ref{lem:shared-tail-selection}. Then \eqref{eq:shared-tail-sum}
implies
\[
\sum_{j=1}^\infty \mathbb E\Big[
\|\mathcal{G}_{k_j}(x^{k_j})\|^2
+\|y^{k_j}-y_{k_j}^*(x^{k_j})\|^2
\Big]<\infty.
\]
Since the summands are nonnegative, Tonelli's theorem yields
$\sum_{j=1}^\infty
\Big(
\|\mathcal{G}_{k_j}(x^{k_j})\|^2
+\|y^{k_j}-y_{k_j}^*(x^{k_j})\|^2
\Big)<\infty
\;\text{a.s.}$
Hence,
$\|\mathcal{G}_{k_j}(x^{k_j})\|\to0,
\;
\|y^{k_j}-y_{k_j}^*(x^{k_j})\|\to0
\;\text{a.s.}$
If, in addition, $s>5t$ and $10t+s<1$, then
\eqref{eq:shared-tail-sum-scaled} similarly gives
$\rho_{k_j}\|\mathcal{G}_{k_j}(x^{k_j})\|\to0,
\;
\rho_{k_j}\|y^{k_j}-y_{k_j}^*(x^{k_j})\|\to0
\qquad\text{a.s.}$

Fix any sample path in this probability-one event. Since $X\times Y$ is compact,
there exists a further subsequence, indexed by $j_\ell$, such that
$(x^{k_{j_\ell}},y^{k_{j_\ell}})\to(\bar x,\bar y).$
Relabeling $\{k_{j_\ell}\}$ as $\{k_i\}$ already gives the first part of the theorem.

For the stationarity statement, set
\[
x_\ell:=x^{k_{j_\ell}},\;
y_\ell:=y^{k_{j_\ell}},\;
y_\ell^*:=y_{k_{j_\ell}}^*(x_\ell),\;
\rho_\ell:=\rho_{k_{j_\ell}},\;
\sigma_\ell:=\sigma_{k_{j_\ell}},\;
\alpha_\ell:=\alpha_{k_{j_\ell}},\;
\varphi_\ell:=\varphi_{k_{j_\ell}}.
\]
Then
$\|y_\ell-y_\ell^*\|\to0,$
and therefore $y_\ell^*\to\bar y$ as well. Also, $\rho_\ell\to\infty$ and
$\sigma_\ell\to0$.
Define
$\widehat x_\ell:=P_X\bigl(x_\ell-\alpha_\ell\nabla\varphi_\ell(x_\ell)\bigr),
\;
\widehat y_\ell:=y_{k_{j_\ell}}^*(\widehat x_\ell),$
and
$u_\ell:=
\mathcal{G}_{k_{j_\ell}}(x_\ell)
+\nabla\varphi_\ell(\widehat x_\ell)-\nabla\varphi_\ell(x_\ell).$
By the projection optimality condition,
$u_\ell\in \nabla\varphi_\ell(\widehat x_\ell)+N_X(\widehat x_\ell).$

Moreover,
$\|\widehat x_\ell-x_\ell\|
=
\alpha_\ell\|\mathcal{G}_{k_{j_\ell}}(x_\ell)\|
\to0,$
and, since $\alpha_\ell L_{\varphi_\ell}=O(k_{j_\ell}^{-3t-s})\to0$,
$\|u_\ell\|
\le
\bigl(1+\alpha_\ell L_{\varphi_\ell}\bigr)
\|\mathcal{G}_{k_{j_\ell}}(x_\ell)\|
\to0.$
Using Lemma~\ref{yxk+1xk},
\[
\|\widehat y_\ell-y_\ell^*\|
\le
\frac{L_{k_{j_\ell}}}{\sigma_\ell}\|\widehat x_\ell-x_\ell\|
\le
\frac{L_{k_{j_\ell}}}{\sigma_\ell}\alpha_\ell
\|\mathcal{G}_{k_{j_\ell}}(x_\ell)\|.
\]
Since $L_k/\sigma_k=O(k^{2t})$ and $\alpha_k=O(k^{-6t-s})$, we have
$\frac{L_{k_{j_\ell}}}{\sigma_\ell}\alpha_\ell
=
O(k_{j_\ell}^{-4t-s})\to0.$
Consequently, $\widehat y_\ell\to\bar y$.
Thus the sequence
$(\widehat x_\ell,\widehat y_\ell,\rho_\ell,\sigma_\ell,u_\ell)$
satisfies the assumptions of Theorem~\ref{thm:approx-to-type1}. Hence, if
GP\L CQ holds at $(\bar x,\bar y)$,
Theorem~\ref{thm:approx-to-type1} yields that $(\bar x,\bar y)$ is a Type-I enhanced KKT point.

Under the same GP\L CQ assumption, if the stronger conditions $s>5t$ and
$10t+s<1$ also hold, then
\[
\rho_\ell\|u_\ell\|
\le
\bigl(1+\alpha_\ell L_{\varphi_\ell}\bigr)
\rho_\ell\|\mathcal{G}_{k_{j_\ell}}(x_\ell)\|
\to0.
\]
Therefore the same sequence satisfies the scaled residual condition in
Theorem~\ref{thm:approx-to-type1}, and $(\bar x,\bar y)$ is a Type-II enhanced
KKT point.
\end{proof}

 \vskip 0.2in
\bibliography{references}

\end{document}